\documentclass[10pt]{article}

\usepackage{amsmath}
\usepackage{amsfonts}       % blackboard math symbols
\usepackage{nicefrac}       % compact symbols for 1/2, etc.
\usepackage{microtype}      % microtypography
\usepackage{wrapfig}

\usepackage{graphicx}
\usepackage{svg}
\usepackage{import}
\usepackage[affil-it]{authblk}
\usepackage{placeins}

\usepackage[
backend=biber
]{biblatex}
\usepackage{xcolor}

\newcommand\crule[3][black]{\textcolor{#1}{\rule{#2}{#3}}}
\definecolor{ColourEven}{HTML}{AAAAAA}
\definecolor{ColourOdd}{HTML}{AAAAFF}
\definecolor{ColourGap}{HTML}{FFFF00}
\definecolor{ColourGapRed}{HTML}{FF0000}
\definecolor{ColourGapBlack}{HTML}{000000}

\graphicspath{{images/}}

\usepackage{amsthm}
\usepackage{amssymb}
\theoremstyle{plain}
\newtheorem{theorem}{Theorem}

\newtheorem{proposition}[theorem]{Proposition}
\newtheorem{property}[theorem]{Tiling property}

\usepackage{chngcntr}
\counterwithin{figure}{section}
\counterwithin{table}{section}

\usepackage{geometry}
\title{Turtles, Hats and Spectres: Aperiodic structures on a Rhombic tiling}
\author{James Smith}
\affil{jpdsmith@gmail.com}
\affil{Nottingham, UK}
\date{}

\begin{document}

\maketitle

\begin{abstract}
The remarkable discovery \cite{smith2023aperiodic} of the Hat tile family answered the long-standing question of whether a single shape can tile the plane in only a non-periodic manner.

These notes derive aperiodic monotiles from a set of rhombuses with matching rules. This dual construction is used to simplify the proof of aperiodicity by considering the tiling as a colouring game on a Rhombille tiling. A simple recursive substitution system is then introduced to show the existence of a non-periodic tiling without the need for computer-aided verification.

A new cut-and-project style construction linking the Turtle tiling with 1-dimensional Fibonacci words provides a second proof of non-periodicity, and an alternative demonstration that the Turtle can tile the plane. An interactive 3D model of this is made available at \url{https://jpdsmith.github.io/AperiodicCube/}.

Deforming the Turtle into the Hat tile then provides a third proof for non-periodicity by considering the effect on the lattice underlying the  Rhombille tiling.

Finally, attention turns to the Spectre tile. In collaboration with Erhard Künzel and Yoshiaki Araki, we present two new substitution rules for generating Spectre tilings. This pair of conjugate rules show that the aperiodic monotile tilings can be considered as a 2-dimensional analog to Sturmian words.
\end{abstract}

\section{Initial observation}

The Spectre aperiodic monotile \cite{smith2023chiral} is a 13 sided shape, but can also be considered as having 14 sides of equal length, two of which are collinear.

\begin{figure}[ht]
    \centering
    \resizebox{0.08\columnwidth}{!}{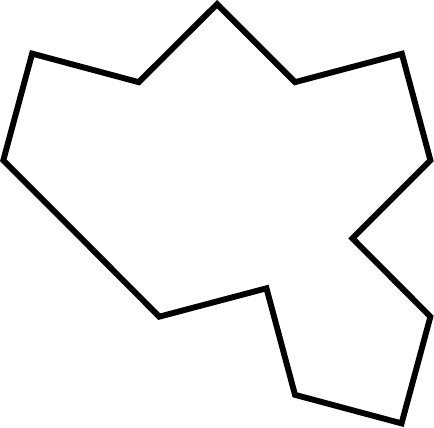}
    % \includesvg[width=0.08\columnwidth]{images/initial_observations/spectre}
\end{figure}

Notice that every second angle is $90^{\circ}$. Except, that is, for angles near the collinear edges. However, taking a $90^{\circ}$ turn inwards where the collinear edges meet, we can interpret the shape as being \emph{16 sided} with two internal sides meeting in a pair of right angles.

That is, we can interpret the Spectre tile as consisting of 8 identical pairs of lines, each of unit length meeting at a right angle.

\begin{figure}[ht]
    \centering
    \resizebox{0.08\columnwidth}{!}{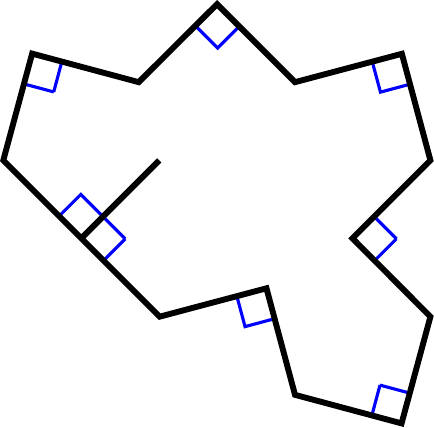}
    % \includesvg[width=0.08\columnwidth]{images/initial_observations/spectre90degrees}
\end{figure}

We can replace each pair of unit-length lines with a line of length $\sqrt{2}$ to obtain an eight sided shape as follows:

\begin{figure}[h]
    \centering
    \begin{tabular}{cc}
        \resizebox{0.08\columnwidth}{!}{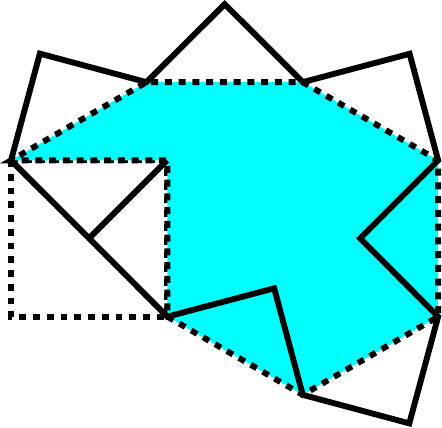}
        % \includesvg[width=0.08\columnwidth]{images/initial_observations/spectreoverlay}
        \hspace{2em}
        &
        \resizebox{0.08\columnwidth}{!}{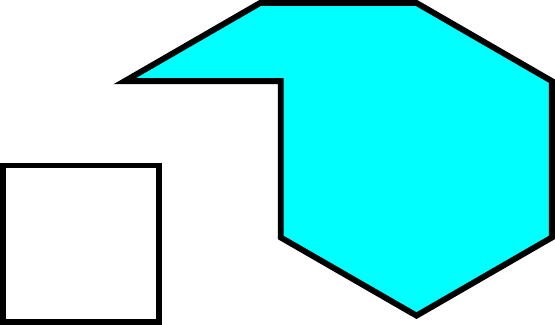}
        % \includesvg[width=0.08\columnwidth]{images/initial_observations/hexandsquare}
        \\
    \end{tabular}
\end{figure}

This process shrinks the Spectre's area and opens up a square shaped hole in any Spectre tiling. That is, Spectre tilings are dual to tilings of the plane with two shapes:

\begin{itemize}
    \item a hexagon with an attached rhombus with $30^{\circ}$ and $150^{\circ}$ angles -- the dual Spectre.
    \item a square `hole'.
\end{itemize}
\section{Matching rules}

This dual construction works for all members of the Hat/Turtle family of \cite{smith2023aperiodic}. We will start with simple shapes and rules that determine how to match these shapes in a tiling. These matching rules will be geometrically enforceable by an edge substitution that creates aperiodic monotiles.

Whereas \cite{smith2023aperiodic} parameterises this family as $\text{Tile(}a,b\text{)}$ in terms of edge lengths $a$ and $b$, we will use an angle $\alpha$ as a parameter.

For an angle $\alpha\in(0, 120]$ consider the following three shapes (all sides have length 1):
\begin{itemize}
    \item A rhombus, $\mathcal{R}_\alpha$, with internal angles $\alpha$, and $180-\alpha$.
        \begin{itemize}
            \item This will be our `hole' tile.
        \end{itemize}
    \item A rhombus $\mathcal{R}_{60}$.
        \begin{itemize}
            \item We will build a hexagon from three of these.
        \end{itemize}
    \item A rhombus $\mathcal{R}_{120-\alpha}$ with internal angles $120-\alpha$ and $\alpha+60$.
        \begin{itemize}
            \item  We will join this to the hexagon.
        \end{itemize}
    
\end{itemize}

\begin{figure}[ht]
    \centering
    \begin{tabular}{ccc}
        \resizebox{0.06\columnwidth}{!}{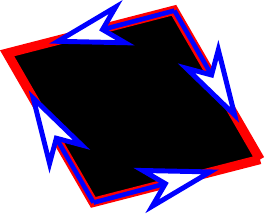}
        % \includesvg[width=0.06\columnwidth]{images/matching_rules/black}
        \hspace{2em}
        &
        \resizebox{0.06\columnwidth}{!}{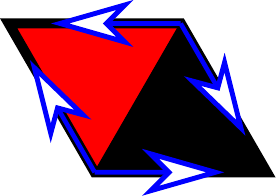}
        % \includesvg[width=0.06\columnwidth]{images/matching_rules/redblack}
        &
        \resizebox{0.06\columnwidth}{!}{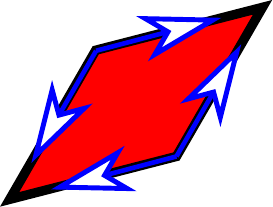}
        % \includesvg[width=0.06\columnwidth]{images/matching_rules/red}
        \hspace{2em}
        \\
        $\mathcal{R}_\alpha$ &
        $\mathcal{R}_{60}$ &
        $\mathcal{R}_{120-\alpha}$ \\
    \end{tabular}
    \caption{A Rhomb tile set}
    \label{fig:rhombtileset}
\end{figure}

We define matching rules for these rhombuses by coloring the tiles \emph{black}, half \emph{black} \& half \emph{red} and \emph{red}. Also, edges are oriented with arrows pointing towards the defining angles $\alpha$, $60$ and $120-\alpha$.

Tiles are only allowed to match along edges where the arrows point the same direction and where the colours are opposite -- red tiles can only touch black tiles.

In particular, three copies of $\mathcal{R}_{60}$ join into a hexagon, so creating the set of three tiles in Figure~\ref{fig:hexrhombs}.

\begin{figure}[h]
    \centering
    \resizebox{0.15\columnwidth}{!}{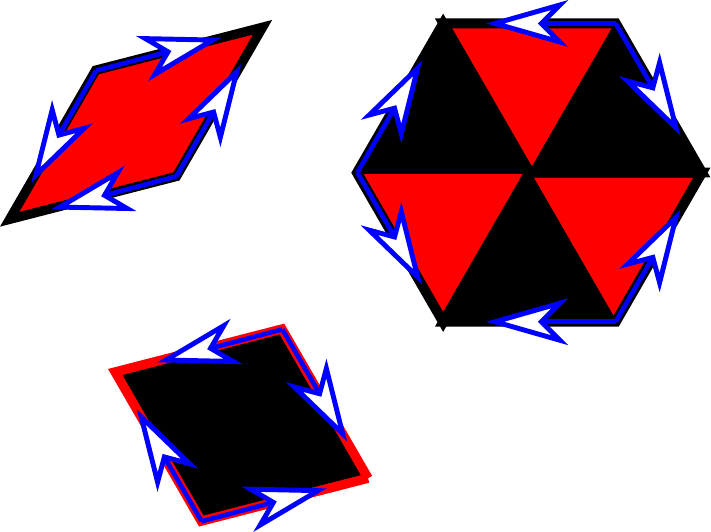}
    % \includesvg[width=0.15\columnwidth]{images/matching_rules/triple}
    \caption{A Hex and Rhomb tile set}
    \label{fig:hexrhombs}
\end{figure}

And joining the red rhombus $\mathcal{R}_{120-\alpha}$ to this hexagon forms the set of two tiles shown in Figure~\ref{fig:aperiodicpair}.

\begin{figure}[h]
    \centering
    \resizebox{0.12\columnwidth}{!}{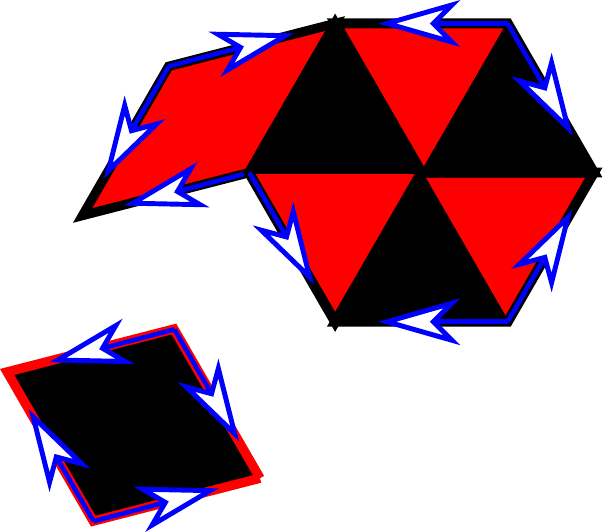}
    % \includesvg[width=0.12\columnwidth]{images/matching_rules/pair}
    \caption{An aperiodic tile set}
    \label{fig:aperiodicpair}
\end{figure}

Alternatively, these matching rules can be enforced by replacing each edge with a pair of line segments corresponding to two lines from vertices of the black rhombus to its midpoint. Replace edges with this pair of lines so that the lines cut inwards on black segments and outwards on red segments. Using this path to enforce the matching rules results in the black `hole' tile collapsing to a cross with zero area plus a single tile: the aperiodic monotile. For the Spectre tile, the arrows of the matching rules must be retained since the pair of line segments is symmetric in this case.

\begin{figure}[htb]
    \centering
    \resizebox{0.6\columnwidth}{!}{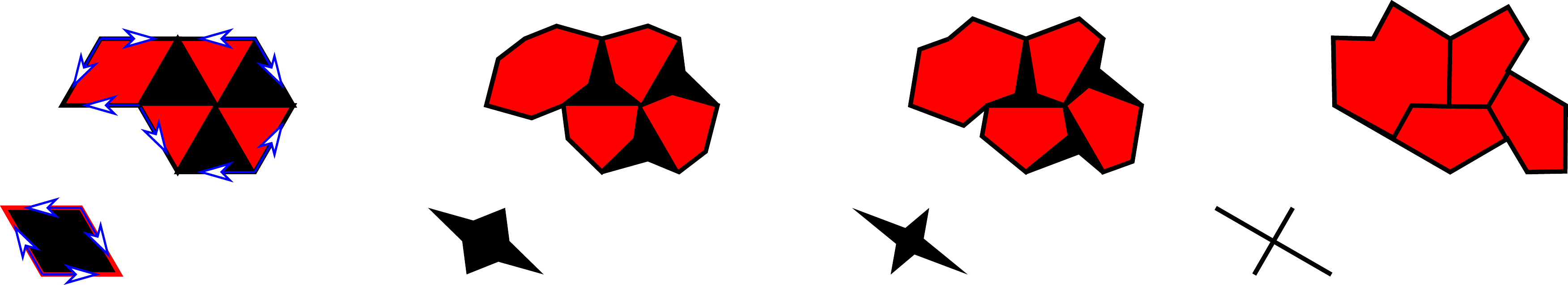}
    % \includesvg[width=0.6\columnwidth]{images/matching_rules/enforced_turtle}
    \caption{Enforcing matching rules to remove the `hole' tile}
    \label{fig:enforced}
\end{figure}

The tile sets, shown in Figure~\ref{fig:enforced}, are mutually locally derivable from one another. Since the monotile is aperiodic \cite{smith2023aperiodic}, the pair of tiles is also aperiodic. However, we will use this dual construction to provide several simpler proofs of aperiodicity.

\begin{figure}[htb]
    \centering
    \resizebox{0.35\columnwidth}{!}{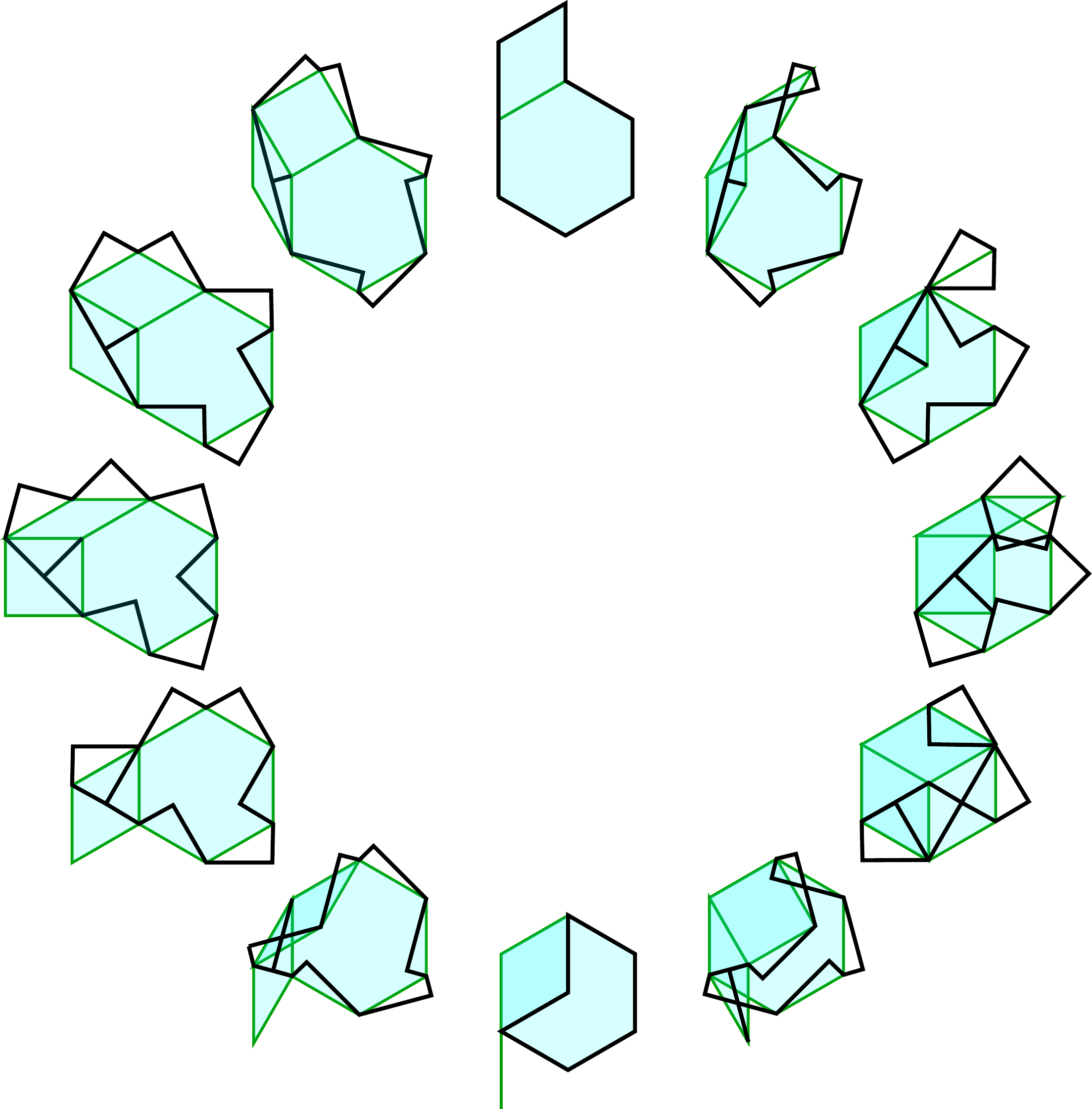}
    % \includesvg[width=0.35\columnwidth]{images/family_circle/monotile_circle}
    \caption{The aperiodic monotile's family}
    \label{fig:familycircle}
\end{figure}

Note that this construction as stated is only well defined for $\alpha\in (0, 120)$. This is fine for our purposes, although with care, it can be extended as shown in Figure~\ref{fig:familycircle}. At $120^{\circ}$, corresponding to the Hat tile, the construction degenerates so that one of the rhombuses is a straight line. Beyond $180^{\circ}$, the resulting tile degenerates further to become disconnected and in most cases can't tile the plane.

The anti-Spectre at $\alpha=270^\circ$ and anti-Turtle at $\alpha=300^\circ$ are interesting shapes that can be shown to tile the plane periodically.

\section{The Turtle}

Unless we're careful, we may find the words `hexagon' and `rhombus' become overloaded as they can refer to tiles, parts of tiles, groups of tiles or general hexagonal arrangements. To help avoid confusion, from now on we will exclusively use the words \emph{rhomb} and \emph{hex} as follows:

\begin{itemize}
    \item \emph{Rhomb} will refer to the three component tiles of the matching rules. Specifically, we will call these three shapes \emph{red}, \emph{black} and \emph{red\&black} rhombs.
    \item \emph{Hex} or \emph{red\&black hex} will refer to three red\&black rhombs joined as a single hexagonal tile following the matching rules.
\end{itemize}

Similarly, we will refer to three classes of tiling all following the matching rules:

\begin{itemize}
    \item \emph{Dual tiling} referring to the tile set of Figure~\ref{fig:aperiodicpair}. In this section, we are considering specifically the dual Turtle tile set.
    \item \emph{Hex and rhomb tiling} referring to the tile set of Figure~\ref{fig:hexrhombs}.
    \item \emph{Red and black rhomb} tiling referring to the tile set of Figure~\ref{fig:rhombtileset}.
\end{itemize}

Although our focus is on understanding the \emph{dual} tilings, we shall find that occasionally considering a dual tiling as either a rhomb or a hex and rhomb tiling is a particularly powerful technique in understanding the tiling properties. 

In this section we consider our construction in the case of the Turtle tile. The Turtle tile refers to $\text{Tile(}\sqrt{3},1\text{)}$ from \cite{smith2023aperiodic}. The dual Turtle has $\alpha=60^\circ$ and is the most symmetric dual tile. We focus on this as a convenient representative for the aperiodic monotile family.

\begin{figure}[htb]
    \centering
    \resizebox{0.15\columnwidth}{!}{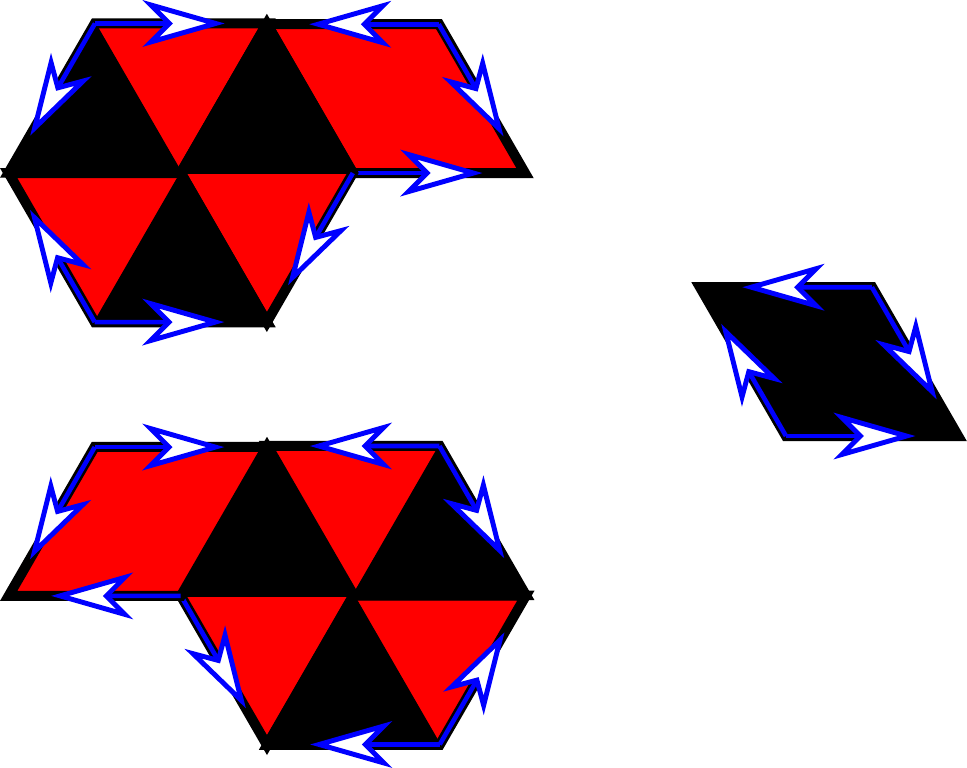}
    % \includesvg[width=0.15\columnwidth]{images/properties/dual_turtle}
    \caption{The dual Turtle, with mirror image and black rhombus `hole' tile}
\end{figure}
\FloatBarrier

\clearpage
\subsection{Tiling properties}

Breaking out the component red and black rhomb tiles, we consider the matching rules on these

\begin{figure}[h]
    \centering
    \begin{tabular}{ccc}
        \includesvg[width=0.05\columnwidth]{images/matching_rules/black60}
        \hspace{1em}
        &
        \includesvg[width=0.05\columnwidth]{images/matching_rules/redblack}
        \hspace{1em}
        &
        \includesvg[width=0.05\columnwidth]{images/matching_rules/red60}
    \end{tabular}
\end{figure}

and find that, for the dual Turtle, the arrows of the matching rules orient the edges so that the rhombs form a Rhombille tiling as in Figure~\ref{fig:rhombic_tiling}.

Thanks to this, we can dispense with the edge arrows and consider the matching rules as a colouring game played on a Rhombille tiling.

\begin{figure}[hbt]
    \centering
     \includesvg[width=0.3\columnwidth]{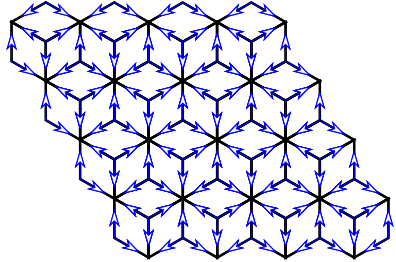}
    \caption{Rhombille tiling}
    \label{fig:rhombic_tiling}
\end{figure}

Next we make several simple observations, labelled as tiling properties. These will build up to the point where we can provide a first simple proof of the aperiodicity of the Turtle tile.

\bigskip
\fbox{\begin{minipage}{0.9\textwidth}
\begin{property}\label{property:proportion}
The proportion of black : red : red\&black rhombs in any dual Turtle tiling must be 1:1:3.  
\end{property}
\begin{proof}
    The dual Turtle tile is constructed from 1 red rhomb joined to 3 red\&black rhombs.
    The outline of this dual Turtle tile has 6 red edges and 2 black. To satisfy the matching rules, the red and black edges must be balanced. So one black rhomb is needed for each dual Turtle tile.
\end{proof}
\end{minipage}}
\bigskip

We avoid providing a precise definition of \emph{proportion}. Later on, we need property \ref{property:proportion} in proving aperiodicity. For that, it is sufficient to say that if a periodic tiling existed, then these proportions would be exact in any fundamental domain of the tiling.

The main benefit of considering our dual construction is that we can break down the tiling into component rhombs. Compared to the original monotile, these discrete components are easy to reason about. The following basic tiling property will be useful for reference in later proofs:

\bigskip
\fbox{\begin{minipage}{0.9\textwidth}
\begin{property}\label{prop:redblackhex}
Each regular hexagon in the Rhombille tiling has at least one rhomb coloured black\&red by the matching rules.
\end{property}
\begin{proof}
    See Figure~\ref{fig:redblackhex} where the grey rhombus is forced to be colored red\&black.
\end{proof}
\end{minipage}}
\bigskip

\begin{figure}[ht]
    \centering
     \includesvg[width=0.2\columnwidth]{images/properties/redblackhex}
    \caption{Tiling property \ref{prop:redblackhex}}
    \label{fig:redblackhex}
\end{figure}

We pause now to provide some definitions of Ammann Bars which are an essential part of understanding the aperiodic nature of the Turtle tile. When considering the underlying Rhombille tiling, these Ammann bars are simply lines of rhombuses in each of three directions though the honeycomb structure.
 
Placing a single dual Turtle tile, that tile's red rhomb defines a line of pairs running alongside the red\&black hex. This is shown in blue in Figure~\ref{fig:singletile}.

Following \cite{akiyama2023alternative}, we refer to this line as the \emph{Golden Ammann Bar} -- GAB -- defined by the dual Turtle tile.

We shall say that the GAB \emph{passes though} another dual Turtle tile whenever one of the rhombuses in the GAB is coloured red\&black and so forms part of the red\&black hex of the other dual Turtle tile.

Similarly, through each red\&black hex, the Rhombille tiling defines lines of rhombuses in three directions offset from each other by $120^\circ$ as shown in Figure~\ref{fig:singletile}. Following \cite{akiyama2023alternative}, we refer to these as \emph{complementary Golden Ammann Bars}.

\begin{figure}[htb]
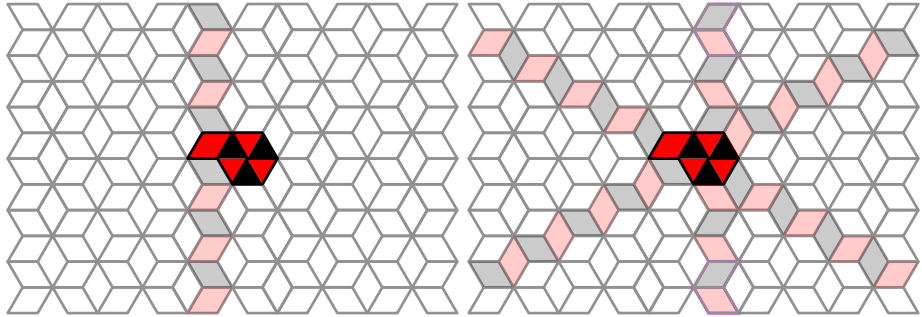

    \centering
     \includesvg[width=0.35\columnwidth]{images/properties/bars}
     \includesvg[width=0.35\columnwidth]{images/properties/complementary_bars}
    \caption{Golden Ammann Bar and complementary Golden Ammann Bars.}
    \label{fig:singletile}
\end{figure}

The effect of the matching rules along these Ammann bars forces long range order and aperiodicity for the Turtle tile. The reader is encouraged to treat the next tiling property as an exercise to help understand this.

\bigskip
\fbox{\begin{minipage}{0.9\textwidth}
\begin{property}\label{property:alongabar}
    Suppose the Golden Ammann Bar running alongside one dual Turtle tile passes through another dual Turtle. Then the two tiles must have opposite orientations.

    Similarly, two dual Turtle tiles passing through the same complementary Golden Amman Bar have the same orientation as each other.
\end{property}
\begin{proof}
    This follows immediately from the matching rules and is shown in Figure~\ref{fig:alongabar}:

    If we are to insert a red\&black rhomb into the alternating red, black, red pattern of rhombuses along a GAB, then we are forced to match colours in a way that preserves or reverses orientation as described.
\end{proof}
\end{minipage}}
\bigskip

\begin{figure}[htb]
    \centering
     \includesvg[width=0.25\columnwidth]{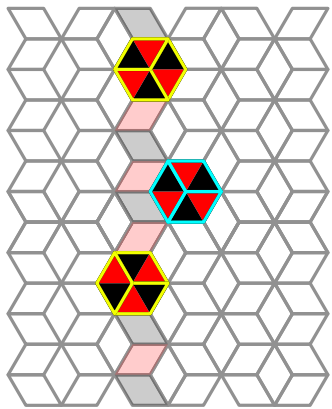}
    \caption{Mirrored tiles passing through a GAB}
    \label{fig:alongabar}
\end{figure}

It is interesting to see that local matching rules immediately force these long range effects.

More locally, we now see how dual tiles group together.

\bigskip
\fbox{\begin{minipage}{0.9\textwidth}
\begin{property}
    Any two red\&black hexes that share an edge must also share an edge with a third red\&black hex.
\end{property}
\begin{proof}
    This follows immediately from the matching rules as shown in Figure~\ref{fig:twoinarow}.
\end{proof}
\end{minipage}}
\bigskip

\begin{figure}[htb]
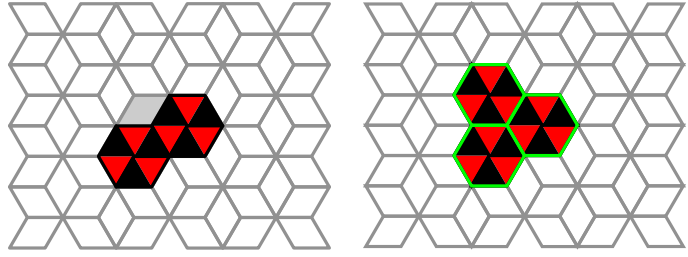

\centering
    \begin{tabular}{cc}
         \includesvg[width=0.25\columnwidth]{images/properties/two_in_a_row}
         & 
         \includesvg[width=0.25\columnwidth]{images/properties/two_in_a_row_completed}
    \end{tabular}
    \caption{Two consecutive red\&black hexes forced to include a third}
    \label{fig:twoinarow}
\end{figure}

In this way, the hex tiles form triangular arrangements. Since the dual Turtle tile has a red rhomb attached to the red\&black hex, this forced triangular arrangement can't be too large. Figure~\ref{fig:trianglearrangements} shows the allowed arrangements of 1,3 or 6 hexes. An arrangement of 10 hexagons would include a ``land-locked'' central hexagon that can't belong to a dual Turtle tile.

\begin{figure}[htb]
    \centering
    \resizebox{0.3\columnwidth}{!}{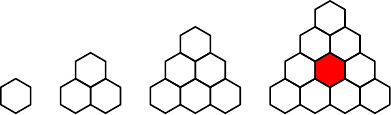}
     % \includesvg[width=0.3\columnwidth]{images/properties/triangle_arrangements}
    \caption{A single tile, clusters or size 3 and 6, and a cluster that's too large.}
    \label{fig:trianglearrangements}
\end{figure}

\bigskip

For the following properties, we define \emph{cluster} to mean one to these triangular arrangements of 3 or 6 dual Turtles whose hexes meet along an edge.

\bigskip

Is it possible to tile the plane with dual Turtles tiles in a way that avoids clusters? No!:

If a hex and rhomb tiling is to have no clusters, then the six edges of each hex must be shared with red and black rhombs only. Figure~\ref{fig:notouchinghexagons} shows tilings with this property that otherwise maximise the density of hexes. The highlighted regions are fundamental domains of these periodic tilings and consists of 3 red rhombs and 3 black rhombs for every 2 hex tiles.

That is, there are insufficient hexes to be in 1:1 correspondence with the red rhombs: a denser arrangement of hexes is required to form a dual Turtle tiling.

\begin{figure}[htb]
    \centering
    \begin{tabular}{c c}
        \resizebox{0.3\columnwidth}{!}{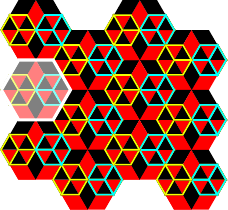}
        % \includesvg[width=0.3\columnwidth]{images/properties/no_touching_hexagons}
        &
        \resizebox{0.3\columnwidth}{!}{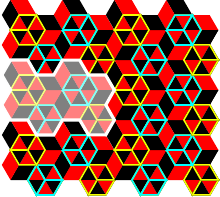}
        % \includesvg[width=0.3\columnwidth]{images/properties/no_touching_hexagons2}
    \end{tabular}
    \caption{Periodic tilings with no hexagons sharing an edge.}
    \label{fig:notouchinghexagons}
\end{figure}

\bigskip
\fbox{\begin{minipage}{0.9\textwidth}
\begin{property}\label{property:clusterorientation}
    Across a dual Turtle tiling, all clustered tiles share the same orientation.
\end{property}
\begin{proof}
    The matching rules clearly force hex tiles within a single cluster to have the same orientation.
    
    Now suppose $A$ and $B$ are dual Turtle tiles belonging to different clusters.

    There are three complementary GABs passing through each of $A$ and $B$ in three directions, so we can find two intersecting complementary GAB as shown in Figure~\ref{fig:clustersmatch}. In fact, there are at least two parallel complementary GABs from each cluster.
    
    The intersection of these two pairs of complementary GABs spans a region that includes more than one hexagon from the underlying Rhombille tiling. One of these is highlighted in Figure~\ref{fig:clustersmatch}.

    By tiling property \ref{prop:redblackhex}, there is a red\&black hex overlapping this spanned hexagon. And by property \ref{property:alongabar}, this red\&black hex at the intersection must have the same orientation as the hexes in each cluster. In particular, $A$ and $B$ share the same orientation.
\end{proof}
\end{minipage}}
\bigskip

\begin{figure}[htb]
    \centering
    \resizebox{0.35\columnwidth}{!}{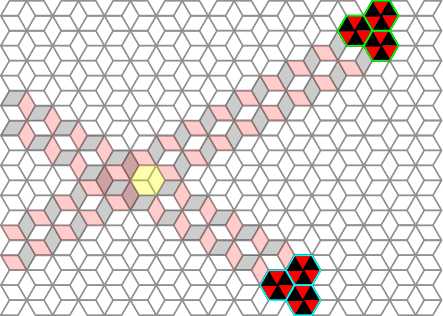}
     % \includesvg[width=0.35\columnwidth]{images/properties/clusters_match}
    \caption{Complementary GABs passing through tile clusters}
    \label{fig:clustersmatch}
\end{figure}

By now, the matching rules have strongly restricted the behaviour of any possible tiling and we can summarise the properties with the following characterisation:
\begin{itemize}
    \item All clustered tiles have the same orientation. So, tiles with the opposite orientation cannot cluster.
    \item So there is a difference between orientations and it makes sense to speak of regular tiles vs. flipped tiles.
    \item From this point, we refer to a GAB as any line of rhombuses that only passes through [the red\&black hex of] \emph{flipped} tiles.
    \item Lines of rhombuses passing through regular (non-flipped) tiles will be called \emph{complementary} GABs.
\end{itemize}

By adding one final property, possible tilings will be restricted enough that we can prove that periodic tilings are impossible.

\bigskip
\fbox{\begin{minipage}{0.9\textwidth}
\begin{property}\label{prop:intersectinggabs}
    Each flipped tile lies at the intersection of three GABs, and two non-parallel GABs intersect at a flipped tile.
\end{property}
\begin{proof}
    The first part is just a clarification: having chosen to distinguish between regular and flipped tiles, the three lines of rhombuses passing through a flipped tile can only pass through other flipped tiles (property \ref{property:alongabar}). And we refer to these lines as GABs.
    
    That two GABs intersect at a flipped tile follows from the matching rules: the intersection of two GABs consists of a single rhombus as shown in Figure~\ref{fig:intersectinggabs}. This rhombus must be coloured red\&black since otherwise the matching rules along one GAB force it to be coloured red and the matching rules along the other GAB force it to be coloured black.
    
    This red\&black rhombus determines the red\&black hex that both GABs passes though. Since GABs only pass through flipped tiles, this red\&black hex belongs to a flipped tile lying at the intersection.
\end{proof}
\end{minipage}}
\bigskip

\begin{figure}[htb]
    \centering
    \resizebox{0.3\columnwidth}{!}{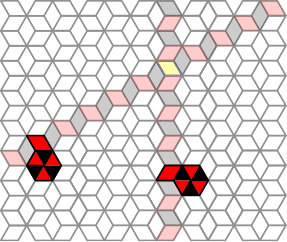}
    % \includesvg[width=0.3\columnwidth]{images/properties/intersecting_gabs}
    \caption{GABs intersecting. A flipped tile must lie at the intersection.}
    \label{fig:intersectinggabs}
\end{figure}
\FloatBarrier
\subsection{First proof of non-periodicity: via Ammann Bars}

The alternative proof of aperiodicity in \cite{akiyama2023alternative} makes use of the fact that the Golden Ammann Bars (in the non-dual tiling) lie on the Kagome lattice. Since the Kagome lattice is dual to the Rhombille tiling, the whole proof of \cite{akiyama2023alternative} carries over to our dual construction.

However, the dual construction makes it easier to reason about the aperiodic monotile and we find that an even simpler proof can be provided. The following proof is based on that found in \cite{akiyama2023alternative}, but dispenses of statistical counts of the length of Golden Amman Bar lines and instead counts exact numbers of rhombuses in a periodically repeating patch of the Rhombille tiling.

\begin{proposition}\label{proposotion:non-periodic}
    The Turtle tile is not periodic
\end{proposition}

\begin{proof}

Suppose there were a periodic tiling, let $t_1$, $t_2$ denote translations preserving it.

These translations must also preserve the underlying Rhombille tiling. Choose a basis of translation vectors $v_1$, $v_2$ preserving the Rhombille tiling (see Figure~\ref{fig:basisvectors}).

\begin{figure}[htb]
    \centering
    \resizebox{0.2\columnwidth}{!}{%% Creator: Inkscape 1.0.2 (e86c870879, 2021-01-15), www.inkscape.org
%% Accompanies image file 'basis_vectors_svg-tex.pdf' (pdf, eps, ps)
\begingroup%
  \makeatletter%
  \providecommand\color[2][]{%
    \errmessage{(Inkscape) Color is used for the text in Inkscape, but the package 'color.sty' is not loaded}%
    \renewcommand\color[2][]{}%
  }%
  \providecommand\transparent[1]{%
    \errmessage{(Inkscape) Transparency is used (non-zero) for the text in Inkscape, but the package 'transparent.sty' is not loaded}%
    \renewcommand\transparent[1]{}%
  }%
  \providecommand\rotatebox[2]{#2}%
  \newcommand*\fsize{\dimexpr\f@size pt\relax}%
  \newcommand*\lineheight[1]{\fontsize{\fsize}{#1\fsize}\selectfont}%
  \ifx\svgwidth\undefined%
    \setlength{\unitlength}{66.24590497bp}%
    \ifx\svgscale\undefined%
      \relax%
    \else%
      \setlength{\unitlength}{\unitlength * \real{\svgscale}}%
    \fi%
  \else%
    \setlength{\unitlength}{\svgwidth}%
  \fi%
  \global\let\svgwidth\undefined%
  \global\let\svgscale\undefined%
  \makeatother%
  \begin{picture}(1,0.8014125)%
    \lineheight{1}%
    \setlength\tabcolsep{0pt}%
    \put(0,0){\includegraphics[width=\unitlength,page=1]{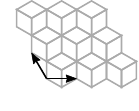}}%
    \put(0.48351549,0.04226677){\color[rgb]{0,0,0}\makebox(0,0)[lt]{\lineheight{1.25}\smash{\begin{tabular}[t]{l}v1\end{tabular}}}}%
    \put(-0.00629461,0.44448877){\color[rgb]{0,0,0}\makebox(0,0)[lt]{\lineheight{1.25}\smash{\begin{tabular}[t]{l}v2\end{tabular}}}}%
  \end{picture}%
\endgroup%
}
     % \includesvg[width=0.2\columnwidth]{images/properties/basis_vectors}
    \caption{Basis of translations preserving the Rhombille tiling}
    \label{fig:basisvectors}
\end{figure}

Writing $t_1$ and $t_2$ in terms of this basis gives

\[
\begin{pmatrix}
t_1\\
t_2
\end{pmatrix}
= 
\begin{pmatrix}
a & b \\
c & d
\end{pmatrix}
\begin{pmatrix}
    v_1\\
    v_2
\end{pmatrix}
\]
for some $a,b,c,d \in \mathbb{N}$ with $ad-bc \neq 0$.

Since
\begin{align*}
    dt_1-bt_2 &= (ab-bc)v_1 \\
    at_2-ct_1 &= (ad-bc)v_2 
\end{align*}

there is an $n = ad-bc \in\mathbb{N}$ such that translations by $nv_1$ and $nv_2$ preserve the dual Turtle tiling.

A fundamental domain for the Rhombille tiling modulo $nv_1$ and $nv_2$ is an $n\times n$ hexagonal grid of $3n^2$ rhombuses. We don't expect that this grid can be partitioned into dual Turtle tiles: some tiles may overlap the boundary. However, the grid \emph{must} decompose into $3n^2$ red, black and red\&black rhombs satifying the matching rules.

\bigskip

Let $k \in \mathbb{N}$ denote the number of GABs parallel to $v_1$ in the $n\times n$ grid of rhombic hexagons. We shall derive a contradiction by showing that the rational number $q:=k/n$ is in fact irrational.

\bigskip

First, the grid of $n\times n$ hexagons contains $3n^2$ rhombuses. The red, black and red\&black rhombs of the matching rules are in the proportion $1:1:3$ [tiling property \ref{property:proportion}], so one in five are black and in particular:

\begin{quote}
\centering
Modulo $nv_1$, $nv_2$, there are $\frac{3}{5}n^2$  black rhombs.
\end{quote}

We re-count the black rhombs in terms of the number of $k$:

 Recall that tiling property~\ref{prop:intersectinggabs} says that at the intersection of any two GABs lies a flipped tile. And each flipped tile lies at the intersection of three GABs in the directions $v_1, v_2$, and $\frac{1}{2}(v_1+v_2)$. Consider one GAB that is parallel to $\frac{1}{2}(v_1+v_2)$. This intersects all the GABs parallel to $v_1$ and all the GABs parallel to $v_2$ and the intersections points along this GAB determines a 1:1 correspondence between the GABs in the other two directions. In particular, there are the same number, $k$, of GABs in the two directions $v_1$ and $v_2$.

The same argument applies equally to each direction so we find that, modulo $nv_1$ and $nv_2$, there are $k$ GABs in each direction.

So there are $3k$ GABs meeting at a flipped tile at each of the $k^2$ intersection points. This is illustrated in Figure~\ref{fig:periodicgrid}.

Now, modulo $nv_1$ and $nv_2$, a GAB consists of $2n$ rhombuses and passes through $k$ flipped tiles. The red\&black hex of each flipped tile covers 2 rhombuses leaving the other $2(n-k)$ rhombuses in each GAB equally red or black. In particular there are $n-k$ black rhombs in each of the $k$ GABs in each of the 3 directions. This gives an exact count:

\begin{quote}
\centering
Modulo $nv_1$, $nv_2$, there are $3k(n-k)$  black rhombs
\end{quote}

\bigskip

So 
\[
 3k(n-k) = \frac{3}{5}n^2
\]
Writing $q = k/n$ gives
\[
 q(1-q) = \frac{1}{5}
\]
So $k/n = \frac{5 \pm \sqrt{5}}{10}$ is irrational. The initial assumption is false: there is no periodic Turtle tiling.

\begin{figure}[htb]
    \centering
    \resizebox{0.5\columnwidth}{!}{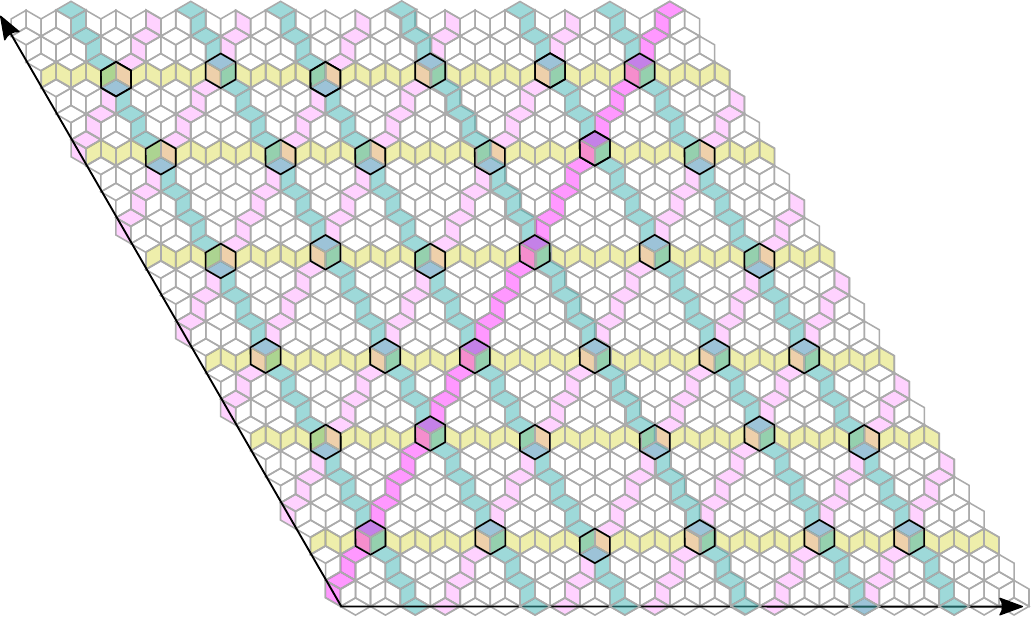}
     % \includesvg[width=0.5\columnwidth]{images/properties/periodicgrid}
    \caption{$k^2$ Flipped tiles at the intersection of $3k$ GABs.}
    \label{fig:periodicgrid}
\end{figure}

\end{proof}

Interestingly, this proof should make it impossible to draw Figure~\ref{fig:periodicgrid}. However, the keen observer may spot an additional GAB intersection point creeping into the diagram.
\FloatBarrier
\subsection{Tilability}\label{section:turtlesubs}

To conclude that the Turtle tile is aperiodic, we need to prove that it can tile the plane. There are already two different substitutions discovered in \cite{smith2023aperiodic} and a third in \cite{akiyama2023alternative}. This section will introduce new rules that are a variant of the `Golden Hex' of \cite{akiyama2023alternative}. Although independently discovered here, the same rules were first uncovered by Erhard Künzel \cite{erhard2023bam} where they are described as Brick and Mortar [BAM] rules. That name highlights their unusual construction as recursively defined palindromic 1D strips of tiles (mortar) gluing together increasingly large 2D regions (bricks) of tiles.

These new rules closely relate to the H7/H8 rules of \cite{smith2023aperiodic} and we start by considering those.

\subsubsection{Hexagonal clusters and the H7/H8 substitution rules.}

The reader is encouraged to try the excellent interactive application of \cite{csk_hat_h7h8} to explore the H7/H8 substitution rules. The rules are natural and create expanding regions of tiles that converge to a fractal shape.

Recall that tiling property \ref{prop:intersectinggabs} showed that each flipped tile lies at the intersection of three GABs, and two non-parallel GABs intersect at a flipped tile. Another way of saying this is that a flipped tile can be viewed as a `seed' that defines three sets of Golden Ammann Bars passing through it. Non-flipped tiles then arrange along the Ammann bars in these directions. Near the flipped tile, the arrangement of tiles is forced: there is only one place each tile can fit. This is shown in Figure~\ref{fig:turtle_forced}.

\begin{figure}[htb]
    \centering
    \resizebox{0.7\columnwidth}{!}{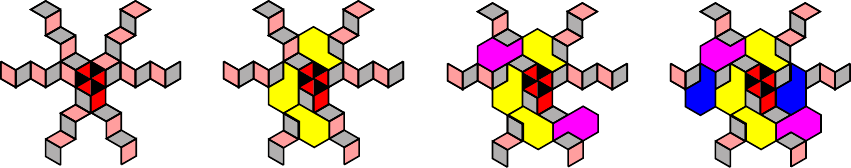}
    % \includesvg[width=0.7\columnwidth]{images/turtle_substitution/hex/turtle_forced} 
    \caption{Dual Turtle tiles forced to cluster along Ammann Bars around a flipped seed tile.}
    \label{fig:turtle_forced}
\end{figure}

The result is a cluster of 8 tiles. It is not immediately obvious that these clusters can tile the plane. Indeed they only tile if we are allowed to identify occasional Turtle tiles from adjacent clusters -- or, as in the H7/H8 rule presentation, consider a second smaller cluster of 7 tiles.

It is clear however that \emph{if} these clusters are to tile the plane, then they must do so in a hexagonal pattern: the GABs extend out of each cluster in six directions and GABs from adjacent clusters must coincide to force a honeycomb structure. And so it is possible to represent the cluster as a regular hexagon which, in Figure~\ref{fig:hex_turtle_cluster}, we orient with an arrow.

\begin{figure}[htb]
    \centering
    \resizebox{0.5\columnwidth}{!}{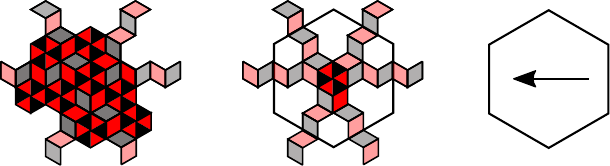}
    % \includesvg[width=0.5\columnwidth]{images/turtle_substitution/hex/hex_turtle_cluster} 
    \caption{Hexagonal Turtle cluster}
    \label{fig:hex_turtle_cluster}
\end{figure}

Iterating the H7/H8 substitution rules on this hexagonal representation and highlighting certain paths in yellow, as shown in Figure~\ref{fig:hex_turtle_h7h8}, a structure emerges. This structure is a new substitution system that we describe next.

\begin{figure}[htb]
\centering
\begin{tabular}[t]{cc}
     \resizebox{0.1\columnwidth}{!}{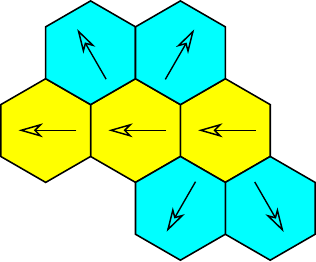}
     % \includesvg[width=0.1\columnwidth]{images/turtle_substitution/hex/cluster1}
     &
     \resizebox{0.3\columnwidth}{!}{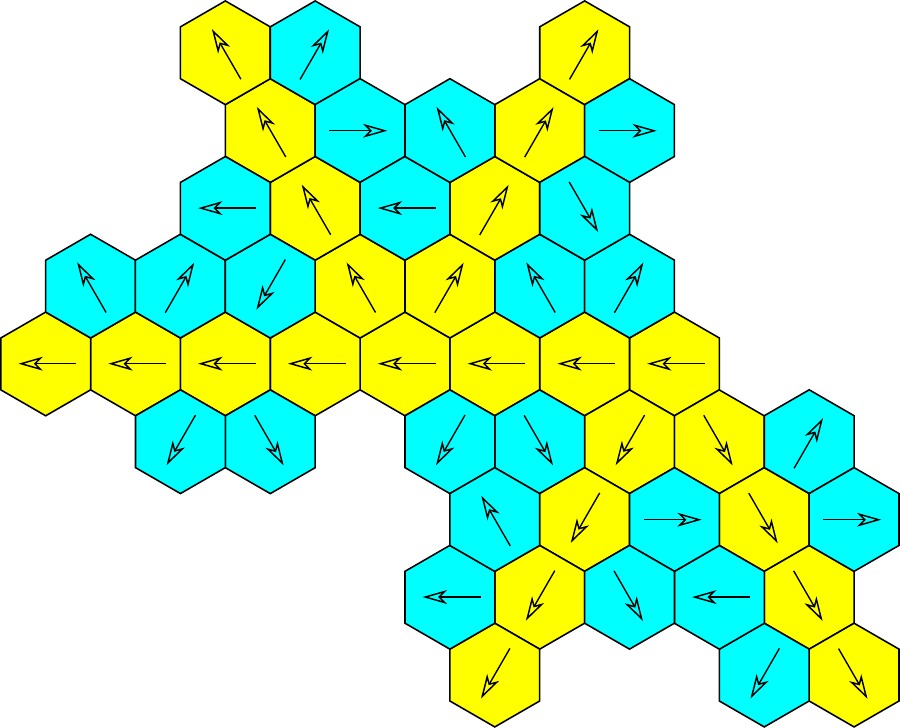}
     % \includesvg[width=0.3\columnwidth]{images/turtle_substitution/hex/cluster2}
     \vspace{1em}
     \\
     \resizebox{0.4\columnwidth}{!}{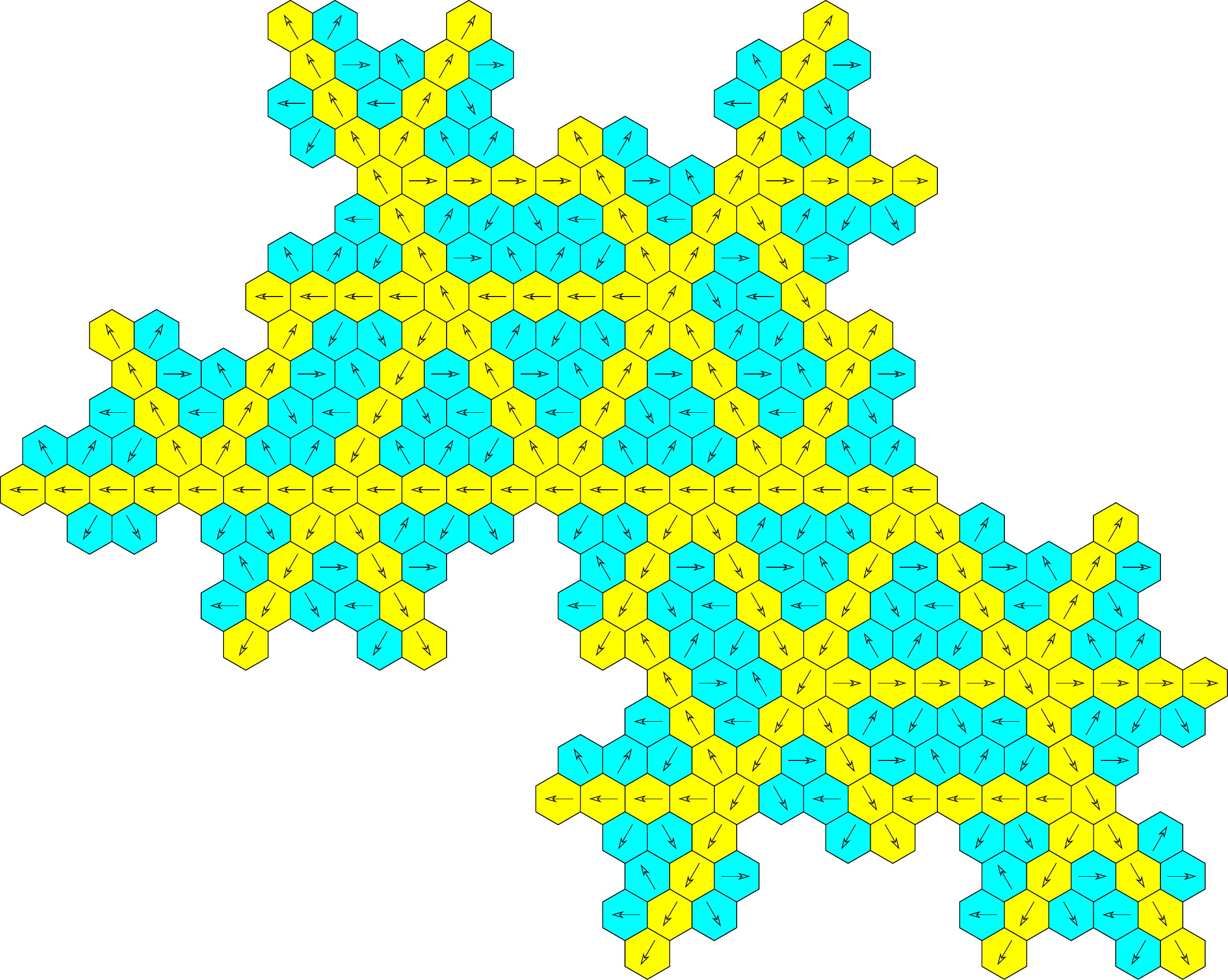}
     % \includesvg[width=0.4\columnwidth]{images/turtle_substitution/hex/cluster3} 
     &
     \includegraphics[width=0.4\columnwidth]{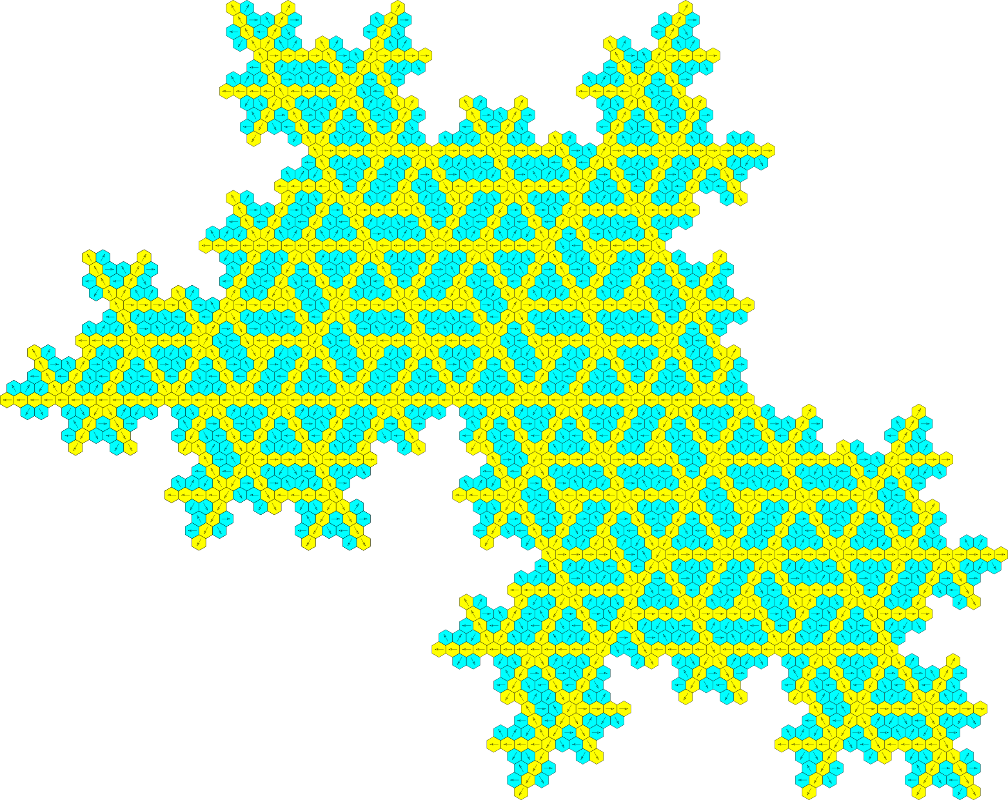}
     \\
\end{tabular}

\caption{H7/H8 cluster substitution overlaid with the new rules}
\label{fig:hex_turtle_h7h8}
\end{figure}

\FloatBarrier
\subsubsection{Bricks and Mortar substitution system}

This section introduces a recursive system that provides a simple proof that the Turtle (or Hat) can tile the plane. The rules are based on a recursively defined set of three strips of tiles that we label $J$, $A_n$ and $B_n$. These are the Golden Sturmian Patches of \cite{akiyama2023alternative}.

\begin{figure}[htb]
    \centering
    \resizebox{0.03\columnwidth}{!}{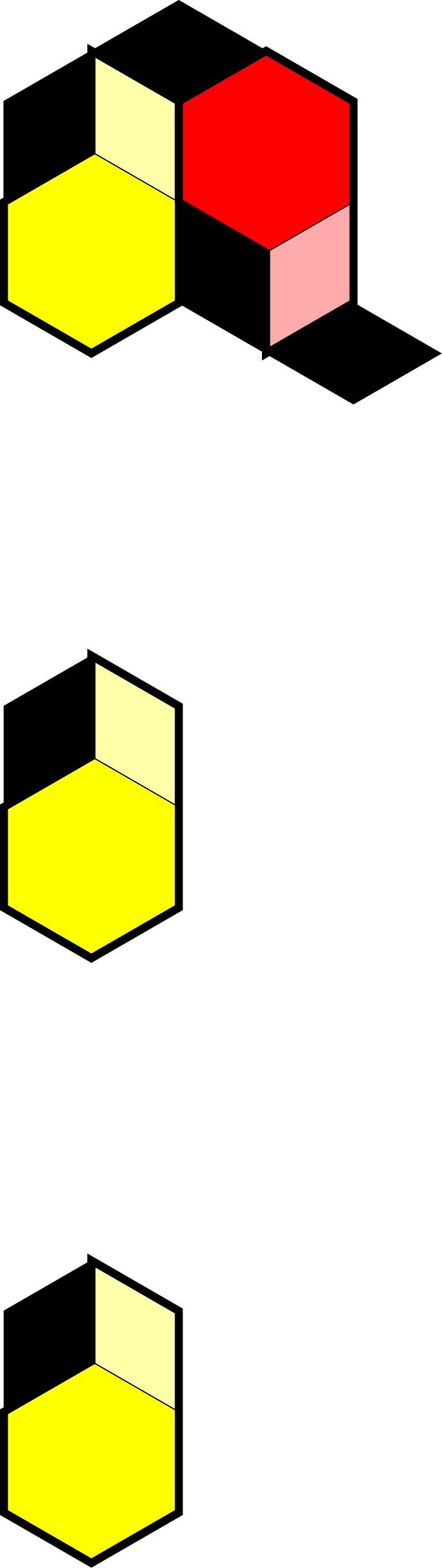}
    % \includesvg[scale=0.04]{images/turtle_substitution/Strips0} 
    \caption{Initial strips $J, A_0, B_0$}
    \label{fig:strips0}
\end{figure}

Let $XY$ denote the concatenation of two strips of tiles $X$ and $Y$. Then define

\begin{align*}
    A_{n+1} &:= B_{n}JA_{n} \\
    B_{n+1} &:= A_{n+1}B_{n} = B_{n}JA_{n}B_{n}
\end{align*}

with $J$, $A_0$ and $B_0$ as shown in Figure~\ref{fig:strips0}. The first few values of the sequence are shown in Figure~\ref{fig:strips}.

\begin{figure}[htb]
\setlength{\tabcolsep}{20pt}
\renewcommand{\arraystretch}{1.5}
\begin{tabular}{c|l}
    $J, A_1, B_1$ 
    &
    \resizebox{0.05\columnwidth}{!}{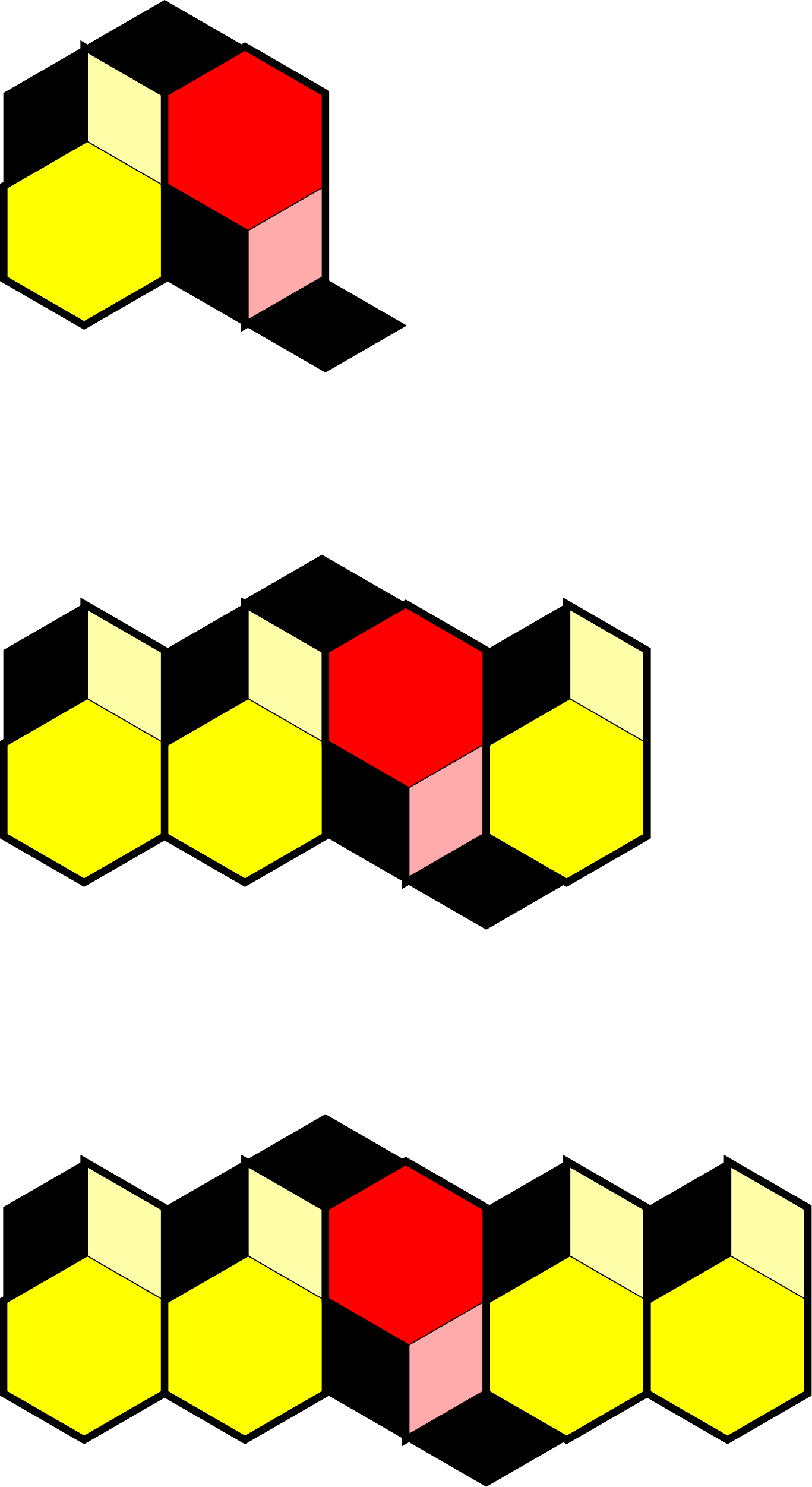} \\
    % \includesvg[scale=0.04]{images/turtle_substitution/Strips1} \\
    \hline
    \vspace{1em}\\
    $J, A_2, B_2$
    &
    \resizebox{0.15\columnwidth}{!}{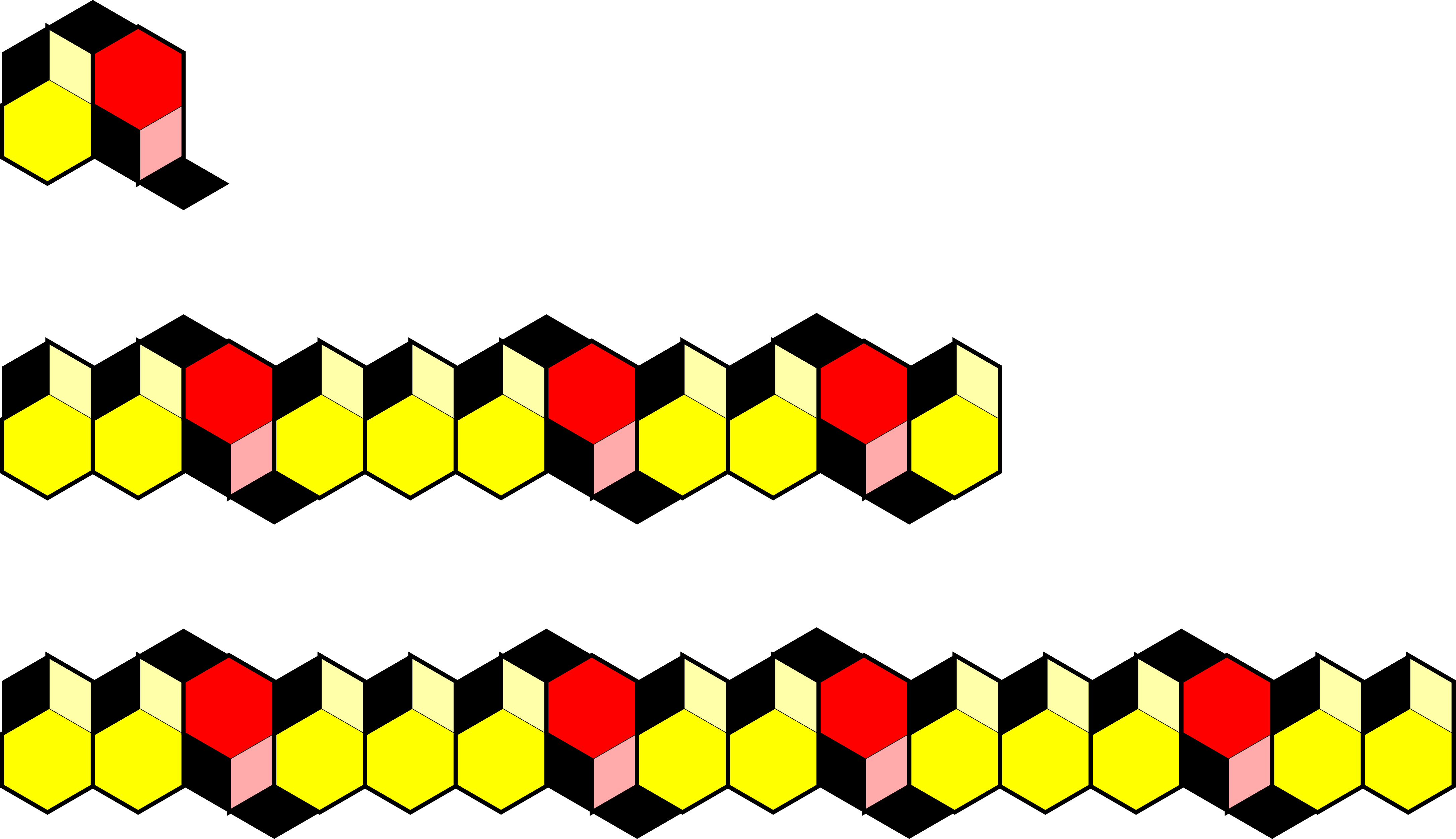} \\
    % \includesvg[scale=0.03]{images/turtle_substitution/Strips2} \\
    \hline
    \vspace{1em}\\
    $J, A_3, B_3$
    &
    \resizebox{0.45\columnwidth}{!}{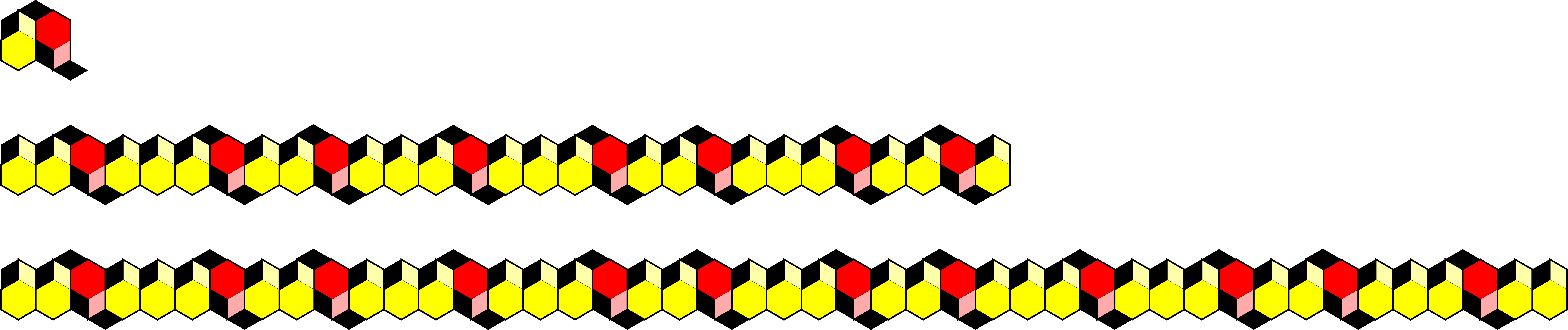} \\
    % \includesvg[scale=0.03]{images/turtle_substitution/Strips3} \\
    \hline \\
    &
    \resizebox{0.45\columnwidth}{!}{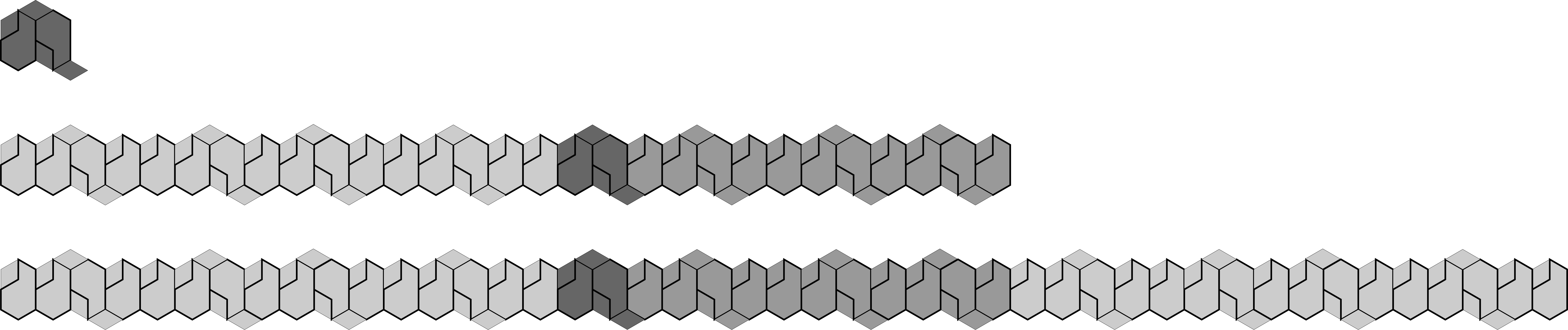} \\
    % \includesvg[scale=0.03]{images/turtle_substitution/Strips3_grey} \\
    \hline
    &
    $A_{n+1} := B_{n}JA_{n}$ \\
    &
    $B_{n+1} := A_{n+1}B_{n} = B_{n}JA_{n}B_{n}$
\end{tabular}
\caption{Recursively defined strips of tiles}
\label{fig:strips}
\end{figure}

Next, we define a series of triangular and parallelogram shaped metatiles $T_n$ and $P_n$. The initial values $T_0$ and $P_0$ are empty and $T_1$ and $P_1$ are shown in Figure~\ref{fig:step1}. To help visually separate each dual tile, we replace the red and black colouring in the hexes with various different colours.

\begin{figure}[htb]
    \centering
    \resizebox{0.2\columnwidth}{!}{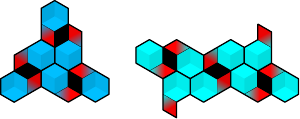}
    % \includesvg[width=0.25\columnwidth]{images/turtle_substitution/new/step1_opt}
    \caption{$T_1$ and $P_1$}
    \label{fig:step1}
\end{figure}

Figure~\ref{fig:substitution_diagram} shows the recursive step that forms $T_{n+1}$ and $P_{n+1}$. Using lower case letters to denote the number of tiles in each shape, we have the following recursions:

\begin{alignat*}{2}
    \mathrm{t}_{n+1} &= 3\mathrm{t}_n + 3\mathrm{p}_n + \mathrm{t}_{n-1} && + 6\mathrm{b}_n \\
    \mathrm{p}_{n+1} &= 2\mathrm{t}_n + 5\mathrm{p}_n + 2\mathrm{t}_{n-1} && + 8\mathrm{b}_n \\
\end{alignat*}

\begin{figure}[htb]
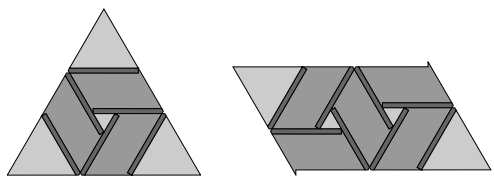

\centering
    \begin{tabular}{cc}
        \resizebox{0.15\columnwidth}{!}{\import{svg-inkscape}{TriangleDiagram_svg-tex.pdf_tex}} &
        % \includesvg[scale=0.015]{images/turtle_substitution/TriangleDiagram} &
        \resizebox{0.2\columnwidth}{!}{\import{svg-inkscape}{ParallelogramDiagram_svg-tex.pdf_tex}} \\
        % \includesvg[scale=0.015]{images/turtle_substitution/ParallelogramDiagram} \\
    \end{tabular}
\caption{$T_{n+1}$ and $P_{n+1}$ each formed from copies of $T_{n}$, $T_{n-1}$, $P_{n}$ and the strip $B_{n}$}
\label{fig:substitution_diagram}
\end{figure}

The strip $B_n$ is defined so that it fits along the bottom edges of both $T_n$ and $P_n$. The strips $J$ and $A_n$ used in generating $B_n$ don't take any further part in the substitution. The symbol $J$ is chosen to indicate the role of this metatile as a junction between strips. The strip $A_n$ fits along the side edges of the parallelogram metatile.

\begin{figure}[htb]
    \centering
   \begin{tabular}{cc}
    \resizebox{0.15\columnwidth}{!}{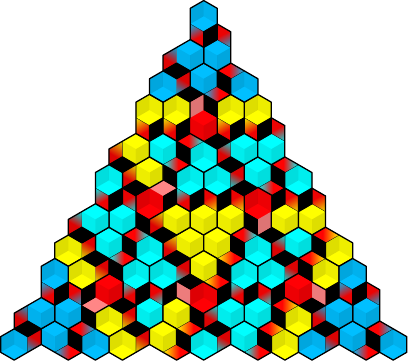}
    % \includesvg[scale=0.5]{images/turtle_substitution/new/T2_opt} 
    &
    \resizebox{0.2\columnwidth}{!}{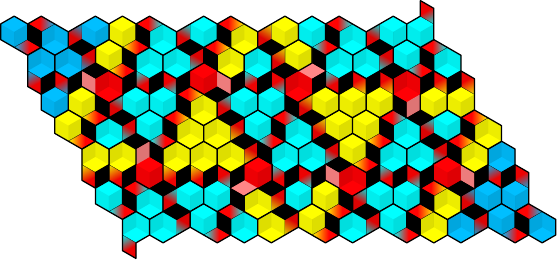}
    % \includesvg[scale=0.5]{images/turtle_substitution/new/P2_opt}
    \end{tabular}
    \caption{Step 2}
\end{figure}

\begin{figure}[htb]
    \centering
   \begin{tabular}{c c}
    \resizebox{0.3\columnwidth}{!}{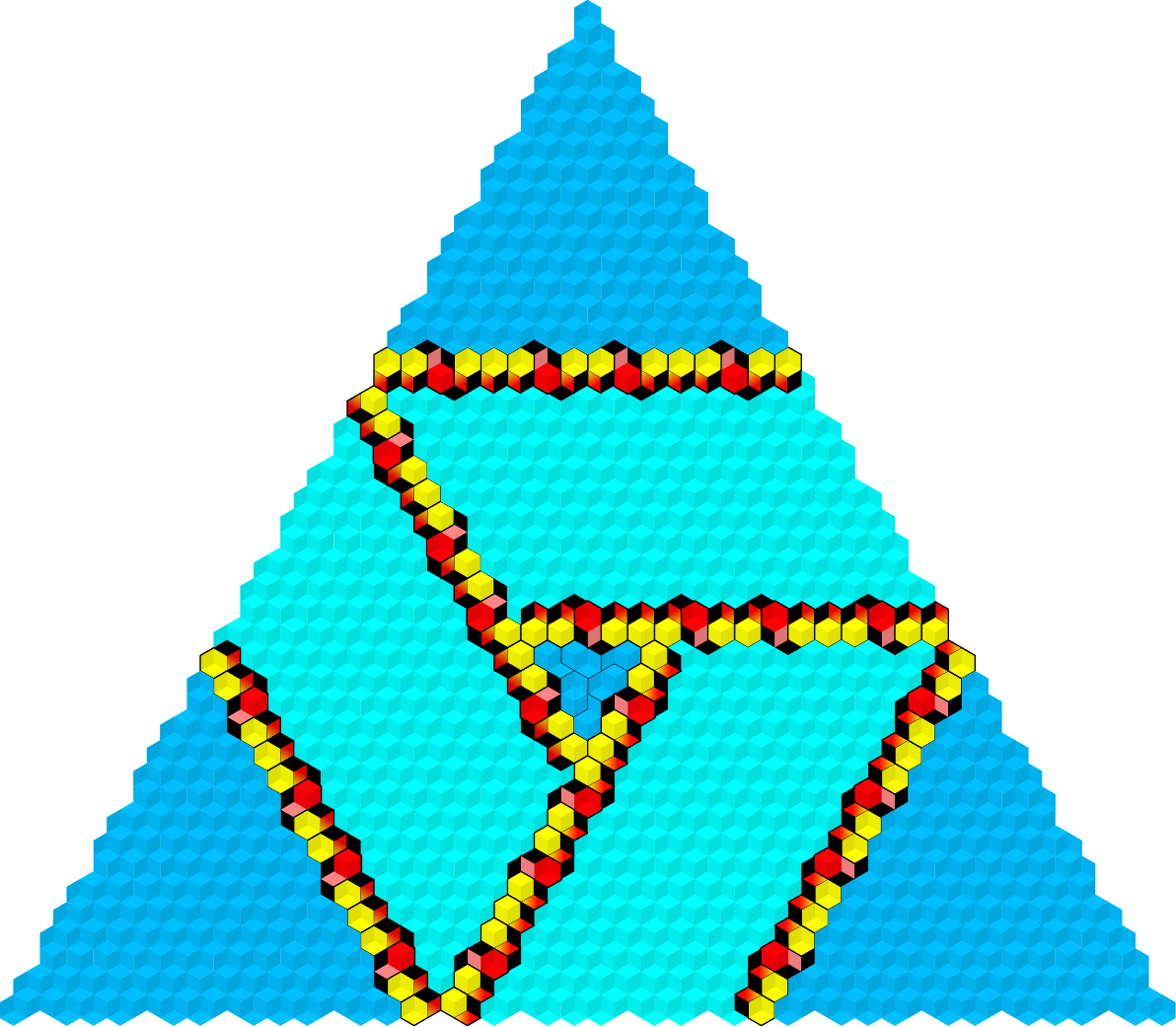} &
    % \includesvg[scale=0.25]{images/turtle_substitution/new/T3_outline_opt} &
    \resizebox{0.3\columnwidth}{!}{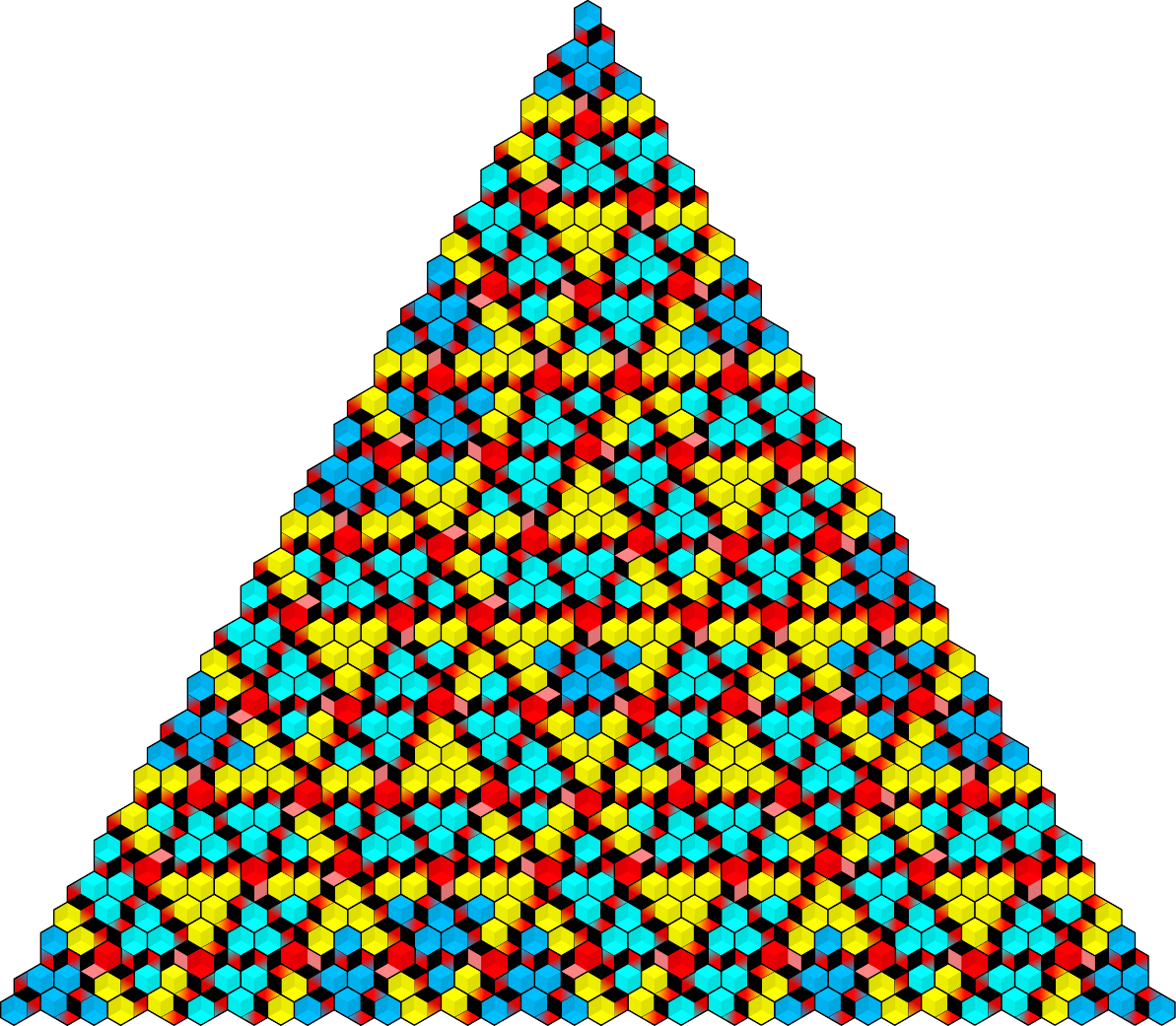} \\
    % \includesvg[scale=0.25]{images/turtle_substitution/new/T3_opt} \\
    \vspace{3em}
    \\
    \resizebox{0.4\columnwidth}{!}{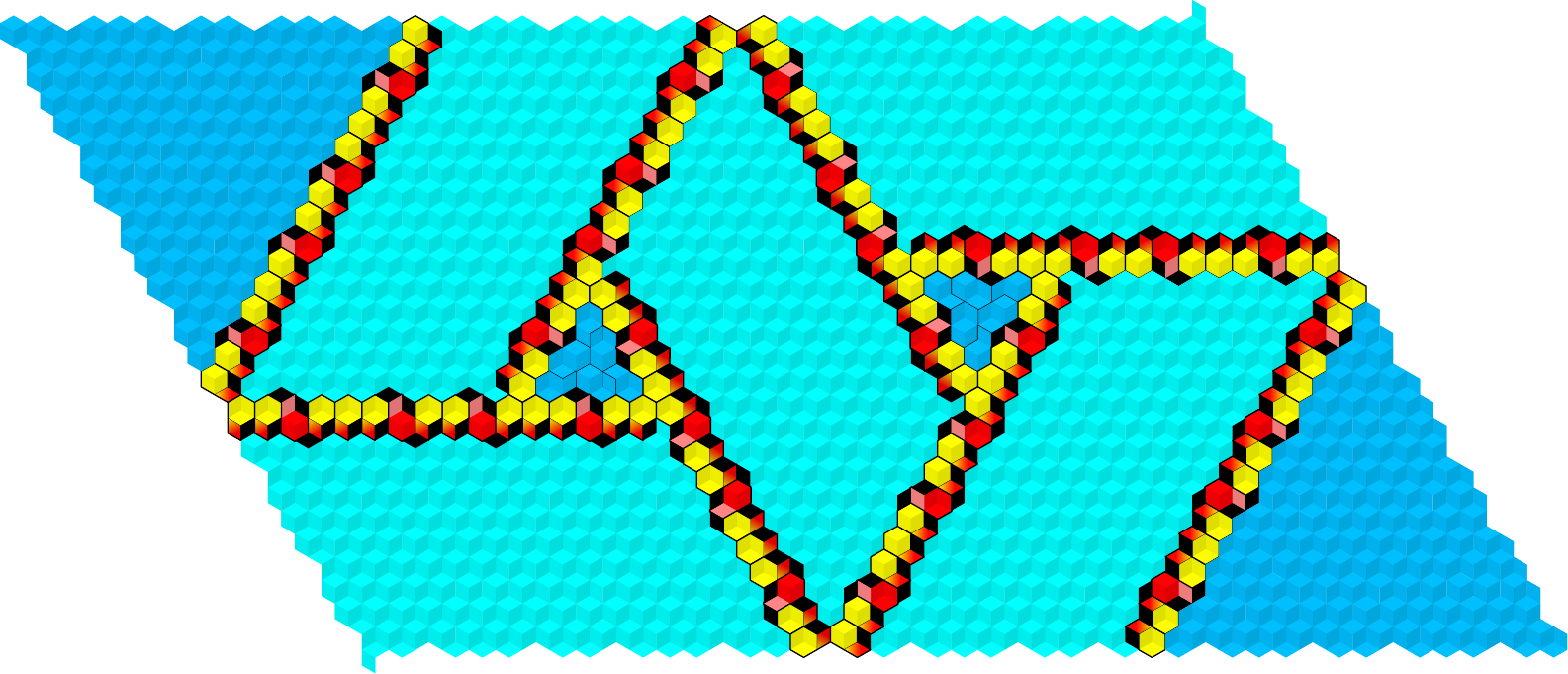} &
    % \includesvg[scale=0.25]{images/turtle_substitution/new/P3_outline_opt} &
    \resizebox{0.4\columnwidth}{!}{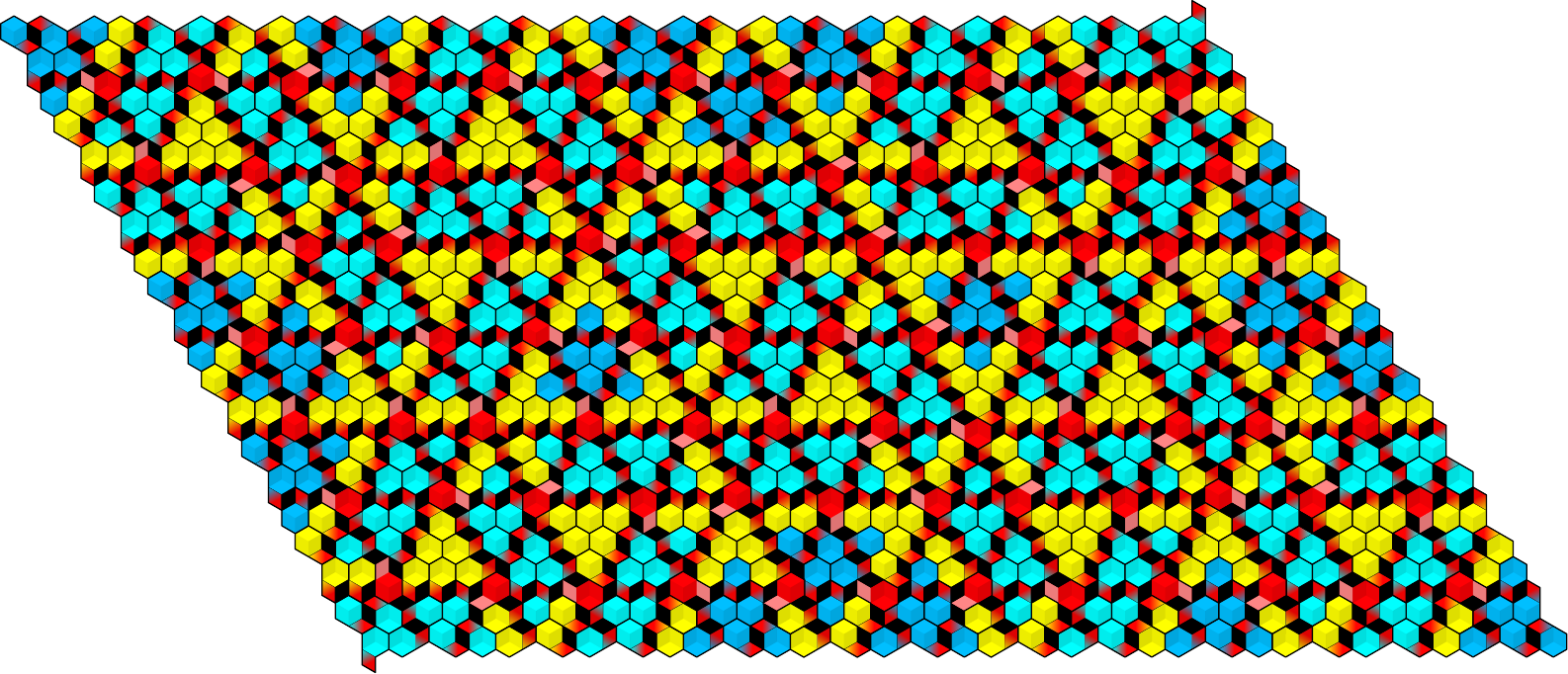}
    % \includesvg[scale=0.25]{images/turtle_substitution/new/P3_opt}
    \end{tabular}
    \caption{Step 3}
\end{figure}

\begin{figure}[htb]
    \centering
    \resizebox{0.22\columnwidth}{!}{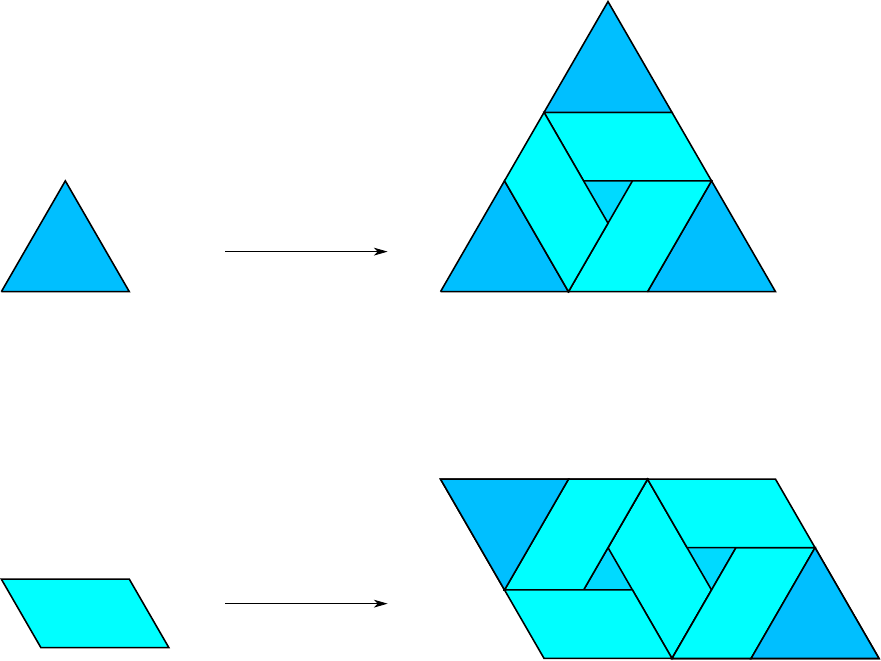}
    % \includesvg[scale=0.25]{images/turtle_substitution/new/limit_substitution} 
    \caption{Limit substitution with inflation factor $\phi^4$}
\end{figure}

\begin{figure}[htb]
    \centering
    \resizebox{0.35\columnwidth}{!}{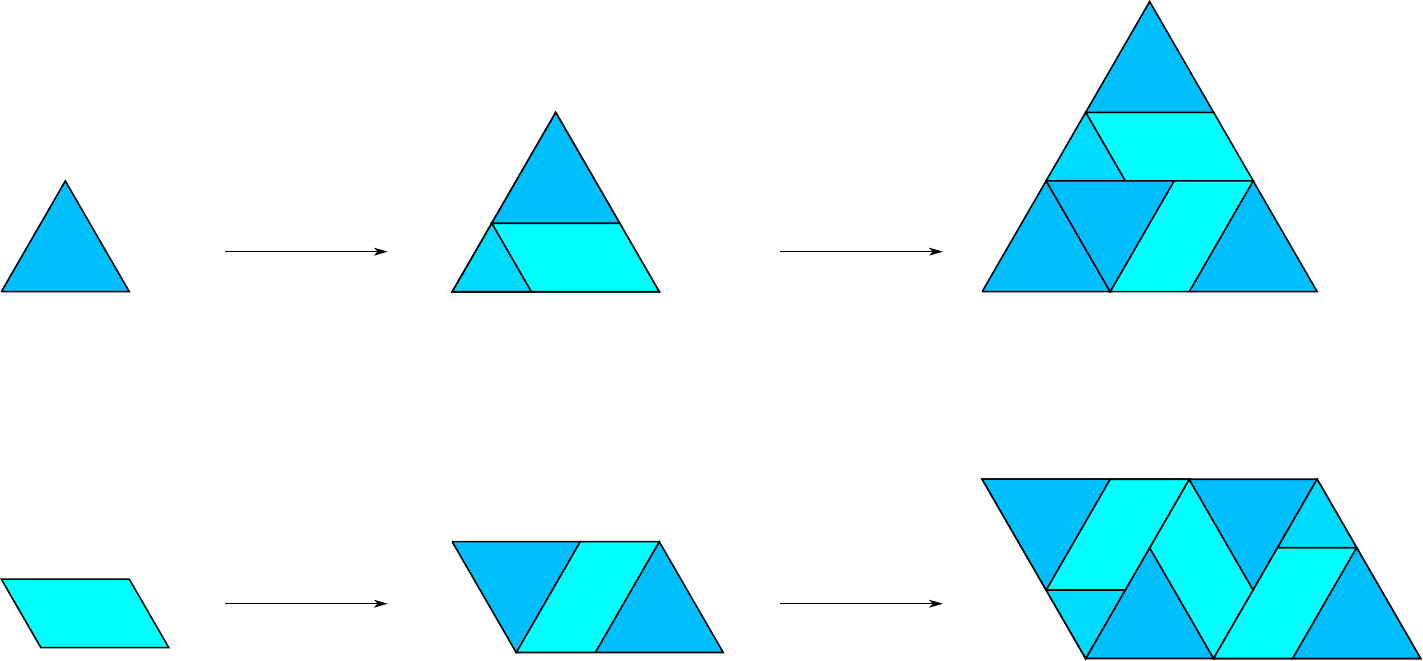}
    % \includesvg[scale=0.25]{images/turtle_substitution/new/limit_substitution_sqrt} 
    \caption{A related `square root' substitution. Long : short side lengths are in the golden ratio $\phi$, so the inflation factor is $\phi^2$}
\end{figure}

\FloatBarrier
\subsubsection{Proof of the substitution rules}

Proving that these substitution rules are well defined and form a tiling of the plane relies on the following facts:
\begin{itemize}
    \item Except for the left-most prototile, the bottom  edges of $P_n$ and $T_n$ are identical.
    \begin{itemize}
        \item Proof: This is true by inspection for $n=1$. Assume it's also true for $n$. Since the constructions of the bottom edges of $P_{n+1}$ and $T_{n+1}$ (Figure~\ref{fig:substitution_diagram}) both add the same tiles as the bottom edge extends to the right, then it's true by induction for all $n$.
    \end{itemize}
    \item The top edge of the strip $B_n$ is identical its bottom edge rotated by $180^\circ$. Similarly, the top edge of the strip $JA_n$ is identical its bottom edge rotated by $180^\circ$.
    \begin{itemize}
        \item Proof: This is true by inspection for $n=1$. Represent the prototile $B_0 = A_0$ by the symbol $0$ and the flipped prototile as $1$ so that $B_1$ is $00100$ and $JA_1$ is $010010$. It is sufficient to show that the representations for $B_n$ and $JA_n$ are palindromes. Assume $B_{n-1}$ and $JA_{n-1}$ are palindromes, then 
        \[
            B_n := B_{n-1}(JA_{n-1})B_{n-1}
        \]
        is a palindrome. Similarly,
        \begin{align*}
            JA_{n} &:= J(B_{n-1}JA_{n-1}) \\
            &:= J(A_{n-1}B_{n-2})JA_{n-1}\\
            &:= (JA_{n-1})B_{n-2}(JA_{n-1})
        \end{align*}
        is a palindrome by induction.
    \end{itemize}
    
\end{itemize}

As noted in \cite{akiyama2023alternative}, although these strips of tiles, $B_n$, are palindromes, they differ at the left and rightmost edges. According to the matching rules, these are coloured black (left edge) and red (right edge).

The top right and bottom left corners of the parallelograms, $T_n$, have a prototile that only matches the black (left) edge of the strips. Thus, the directions of the strips in the recursive step of Figure~\ref{fig:substitution_diagram} are unambiguously defined.

Proving that the substitution rules work is then an exercise in induction by comparing the recursion $B_{n+1} = A_{n+1}B_n$ with the recursive definition of the bottom edge of $P_{n+1}$ and comparing $A_{n+1} = B_nJA_n$ with the side of $P_{n+1}$.

\begin{figure}[hbt]
    \centering
    \includegraphics[width=\linewidth]{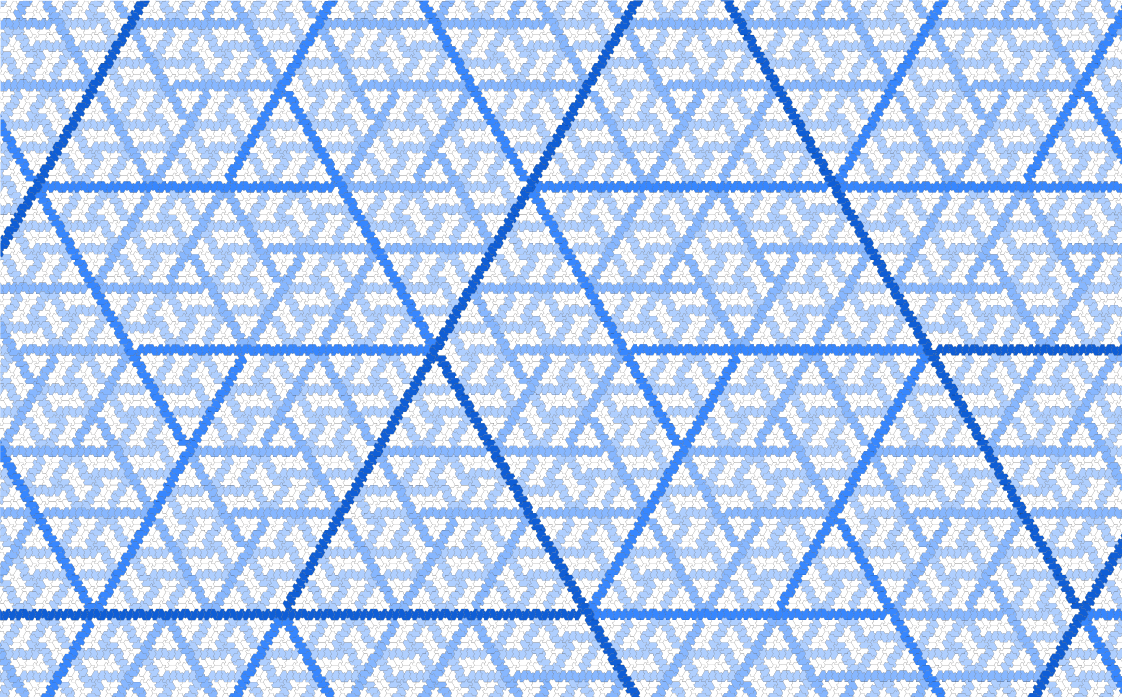}
    \caption{The Turtle tiling}
    \label{fig:turtlepatch}
\end{figure}
\FloatBarrier
\subsection{Fibonacci cube construction}\label{section:fibinacci}

The Rhombille tiling can be viewed as a slice through a 3D cubic lattice -- each hexagon clearly resembles the projection of a cube. In this section we demonstrate a way of viewing the Turtle tiling on this slice. This is a cut-and-project style construction, except that we will not cut a periodic structure at an irrational angle. Instead, we will cut an aperiodic structure at a rational angle.

The strips of tiles $A_n$ and $B_n$ defined in the previous section form sequences of tiles and their mirror images, coloured yellow and red respectively in Figure~\ref{fig:strips}. These are the Golden Sturmian Patches of \cite{akiyama2023alternative} and, as noted on page 10 of \cite{akiyama2023alternative}, we find that these are related to a Sturmian sequence:

Representing regular and flipped tiles with the symbols 0 and 1, we have the following values:
\begin{align*}
        A_0 &= 0 \\
        B_0 &= 0 \\
        A_1 &= 0010 \\
        B_1 &= 00100 \\
        A_2 &= 00100010010 \\
        B_2 &= 0010001001000100 \\
        ... \\
        A_n &:= B_{n-1}01A_{n-1} \\
        B_n &:= A_nB_{n-1}
\end{align*}

Comparing this to the following sequence of Fibonacci words

\begin{align*}
        \mathcal{F}_1 &= 0 \\
        \mathcal{F}_2 &= 001 \\
        \mathcal{F}_3 &= 0010 \\
        \mathcal{F}_4 &= 0010001 \\
        \mathcal{F}_5 &= 00100010010 \\
        \mathcal{F}_6 &= 001000100100010001 \\
        ... \\
        \mathcal{F}_n &:= \mathcal{F}_{n-1}\mathcal{F}_{n-2}
\end{align*}

We see that
\begin{align*}
    \mathcal{F}_{2k-1} &= A_k \\
    \mathcal{F}_{2k} &= B_k01
\end{align*}
Letting $F_n$ denote the standard numerical Fibonacci sequence $0,1,1,2,3,5,...$ with initial values $F_0 := 0$ and $F_1 := 1$ and $F_n := F_{n-1} + F_{n-2}$, we see that the total counts of symbols 0 and 1 is given by

\begin{align*}
\left| \mathcal{F}_n \right|_0 &= F_{n+1} \\
\left| \mathcal{F}_n \right|_1 &= F_{n-1} 
\end{align*}

and so the number of regular [even] and mirrored [odd] tiles in each strip of tiles is given by

\begin{align*}
\left| B_n \right|_{even} &= F_{2n+1} - 1 \\
\left| B_n \right|_{odd} &= F_{2n-1} - 1 \\
\end{align*}

An interesting exploration of the tiling properties of 1D Fibonacci sequences is found in \cite{baake2023fibonacci}. That exploration includes comparisons made between the aperiodic monotile and Fibonacci word sequences; in particular, the asymmetric role of 0 vs. 1 in the 1D Fibonacci case and regular vs. mirrored tiles in the 2D Hat/Turtle case. Next, we show how these comparisons are not coincidental: the 2D Turtle substitution rules can be derived directly from Fibonacci words.

For each $n \in \mathbb{N}$ consider the Cartesian product $B_n \times B_n \times B_n$. As above, $B_n$ refers to the $2n$-th Fibonacci word $\mathcal{F}_{2n}$ truncated to remove the final two characters (this is known as the central word).

Each element $(i,j,k) \in B_n \times B_n \times B_n$ labels a cube in an $n \times n \times n$ grid. We categorize the cubes in this 3D arrangement by the following types:

\begin{itemize}
    \item Type \emph{even} if $i=j=k=0$  [ coloured \crule[ColourEven]{0.5em}{0.5em} in fig.~\ref{fig:fibonaccicube}]
    \item Type \emph{odd} if $i=j=k=1$  [ coloured \crule[ColourOdd]{0.5em}{0.5em} ]
    \item Type \emph{gap} otherwise  [ coloured \crule[ColourGap]{0.5em}{0.5em}, or \crule[ColourGapRed]{0.5em}{0.5em} \& \crule[ColourGapBlack]{0.5em}{0.5em} ]
\end{itemize}

\begin{figure}[htb]
    \centering
    \includegraphics[width=0.25\columnwidth]{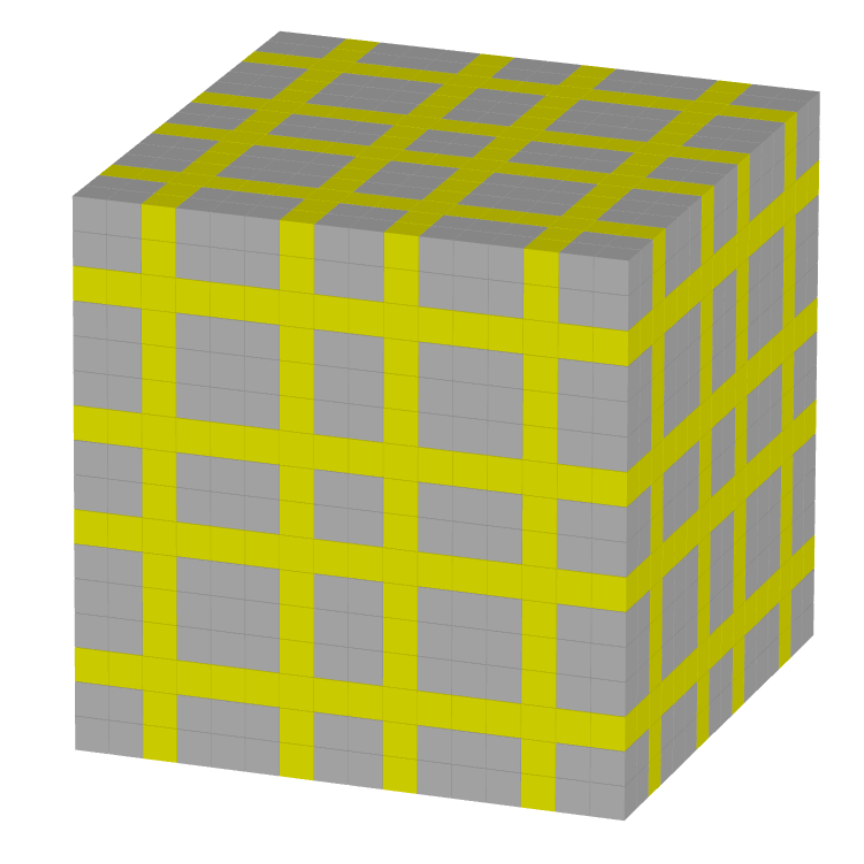}
    \caption{Fibonacci cube for $B_2=0010001001000100$}
    \label{fig:fibonaccicube}
\end{figure}

\bigskip
Slicing this cube just below the diagonal reveals the triangular Turtle substitution steps from the previous section.

\begin{figure}[htb]
    \centering
    \begin{tabular}{c c}
         \includegraphics[width=0.25\columnwidth]{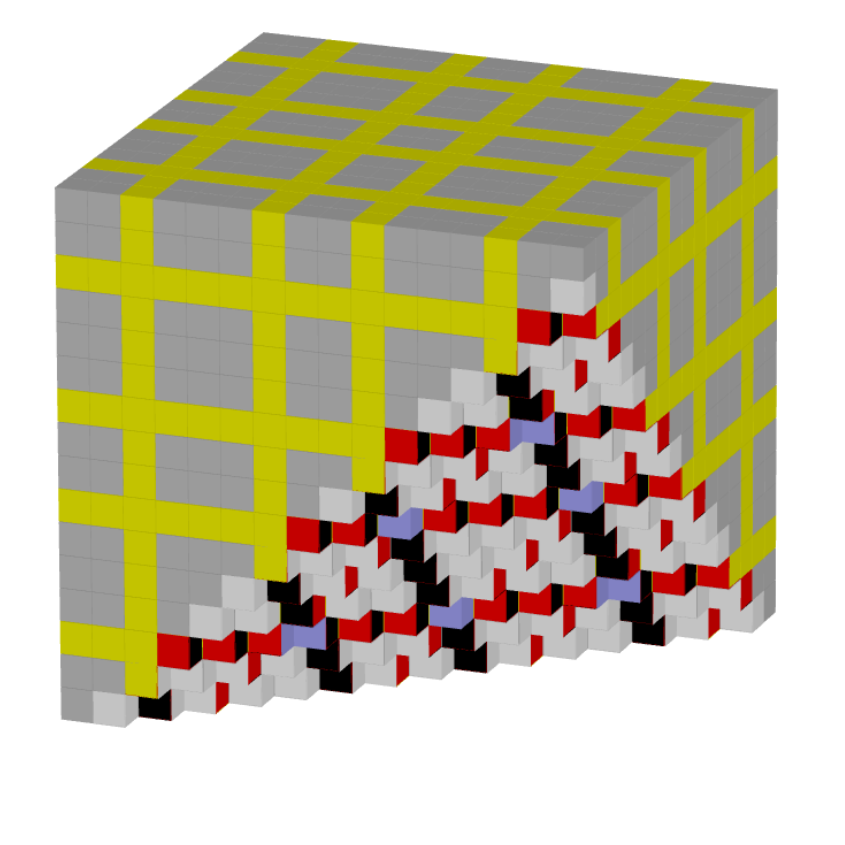}
         &
         \includegraphics[width=0.25\columnwidth]{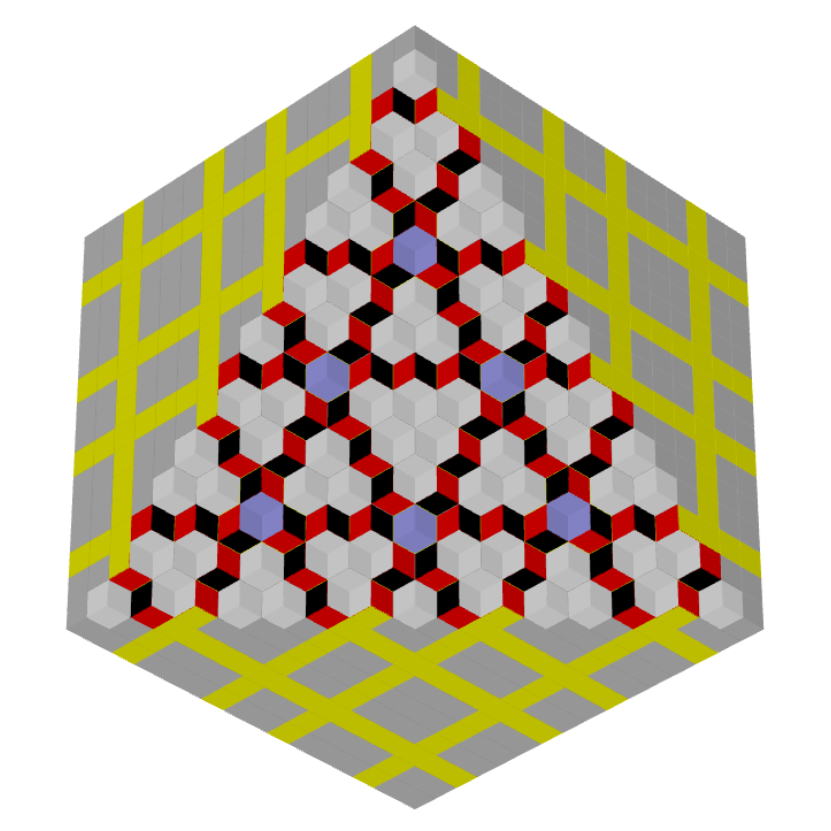}
    \end{tabular}
    \caption{Sliced Cube}
    \label{fig:fibonaccicubecut}
\end{figure}

An interactive 3D visualisation is provided at \url{https://jpdsmith.github.io/AperiodicCube/}

\begin{figure}[htb]
    \centering
    \begin{tabular}{ccc}
         \includegraphics[width=0.25\columnwidth]{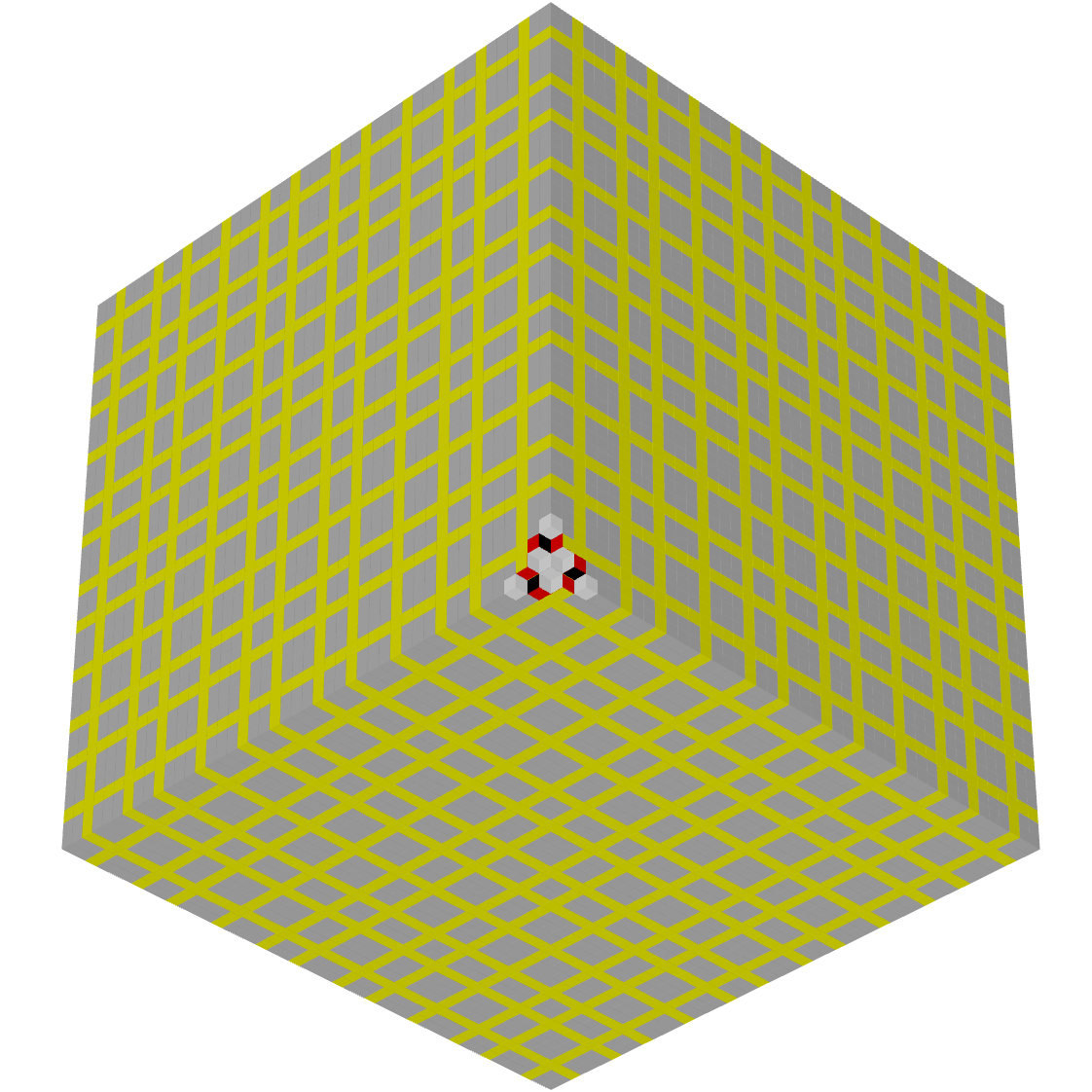}
         &
         \includegraphics[width=0.25\columnwidth]{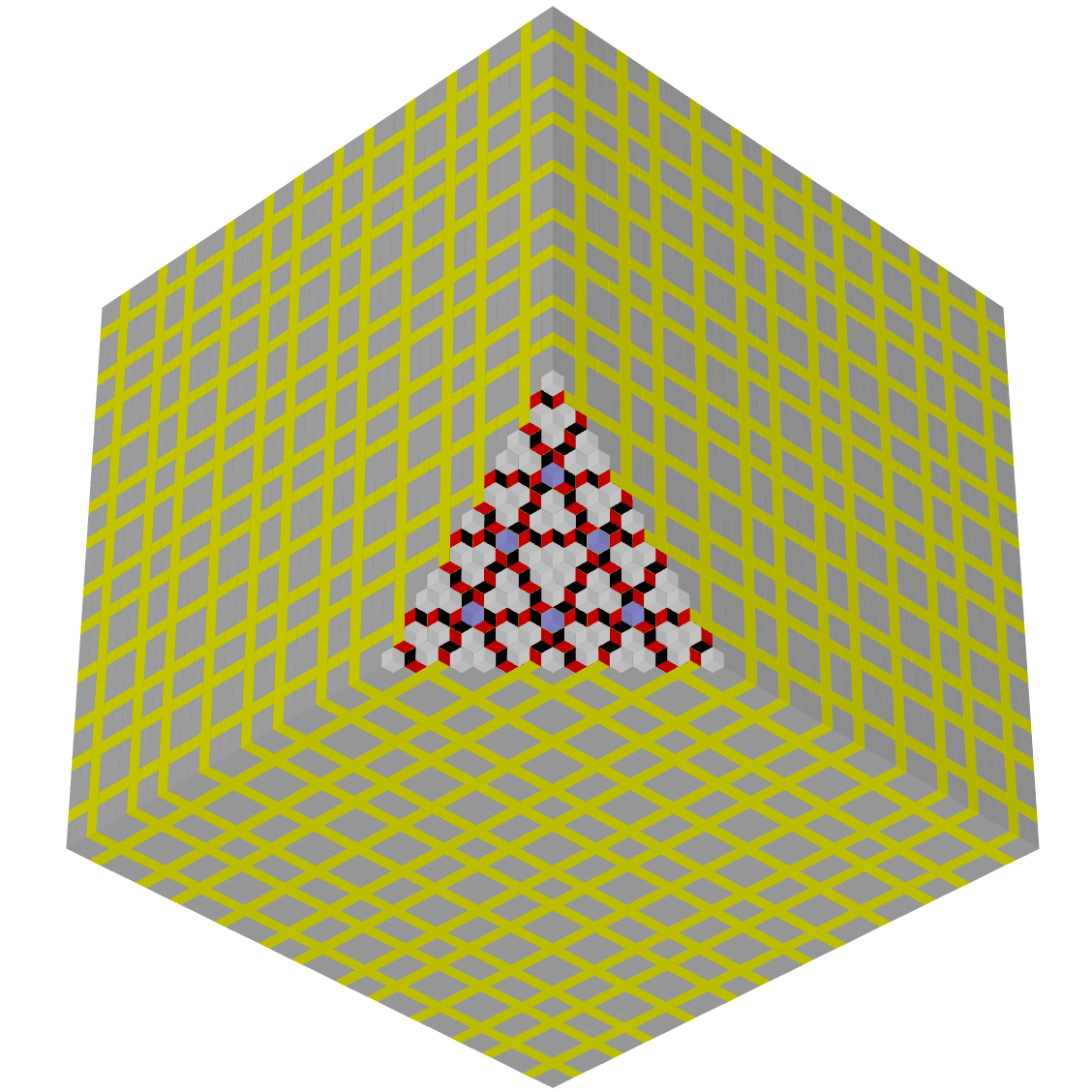}
         &
         \includegraphics[width=0.25\columnwidth]{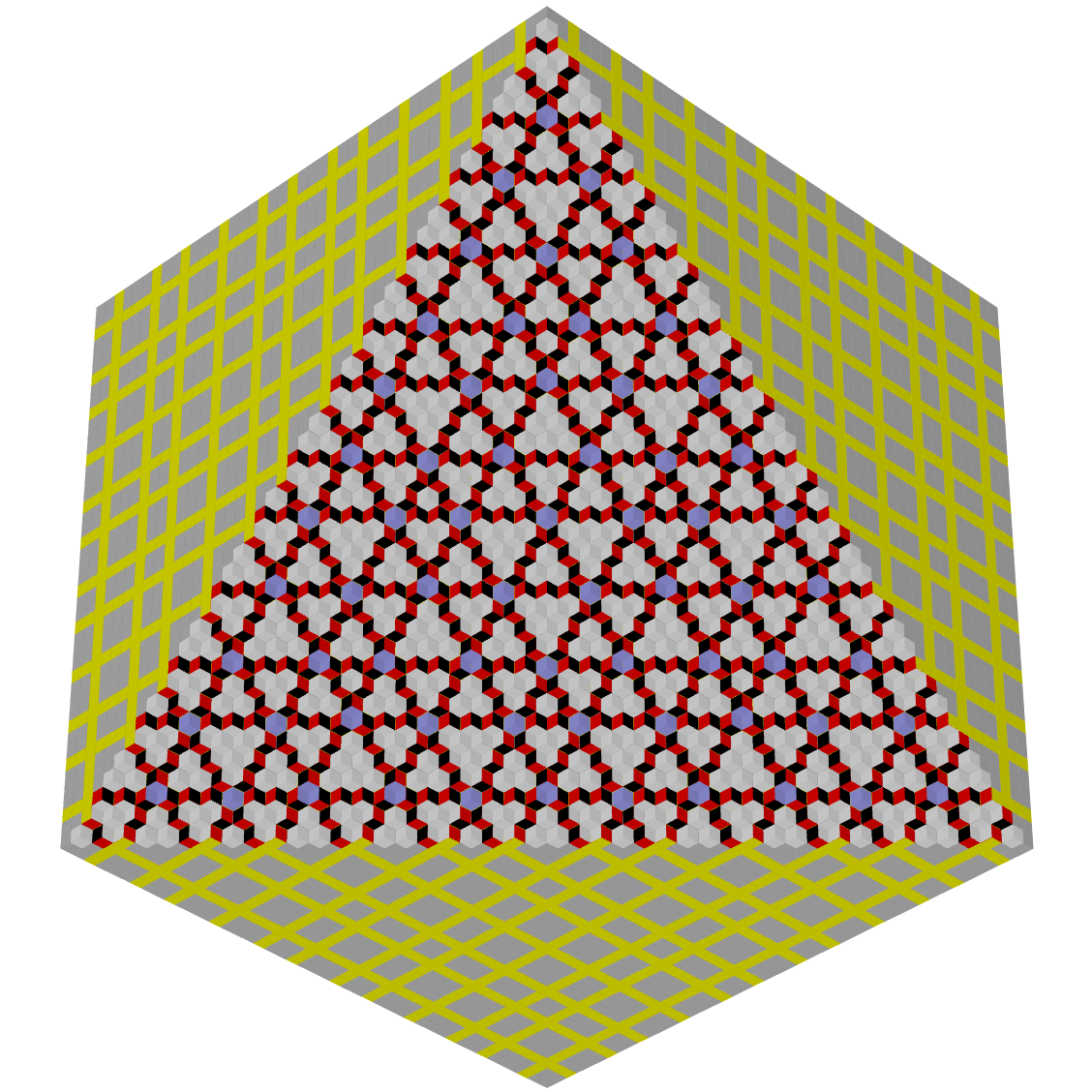}
         \\
    \end{tabular}
    \caption{$T_n$, for $n=1,2,3$ }
    \label{fig:fibonaccicube_slices}
\end{figure}

\subsubsection{Second proof of non-periodicity: via projection}

\begin{proof}
The Rhombille tiling can be viewed as the projection of faces of cubes to a plane. For example, a slice of cubes in $\mathbf{R}^3$ having vertices at $\mathbf{Z}^3$ can be defined that lies close to the plane defined by $x+y+z=0$ so that a projection to the same plane results in the Rhombille tiling. Avoiding a precise definition of this map, it is nevertheless clear that there are three projections $X$,$Y$ and $Z$: $\mathbf{R}^3 \rightarrow \mathbf{R}^2$ defined by $x$, $y$ or $z$ $\mapsto 0$ and that these three maps each project the vertices of the cubes to $\mathbf{Z}^2$. Abusing notation, this defines three maps $X, Y, Z:\mathbf{R}^2 \rightarrow \mathbf{R}^2$ that map the vertices of the Rhombille tiling to integer lattice points. Each projection maps one third of the rhombuses to squares and contracts the other two rhombuses to line segments.

\begin{figure}[htb]
    \centering
    \resizebox{0.3\columnwidth}{!}{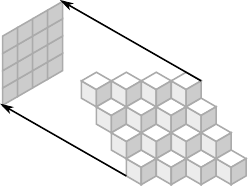}
    % \includesvg[width=0.3\columnwidth]{images/fibonacci_cube/projection}
    \caption{One of three projections}
    \label{fig:projection}
\end{figure}

\begin{figure}[htb]
    \centering
    \resizebox{0.2\columnwidth}{!}{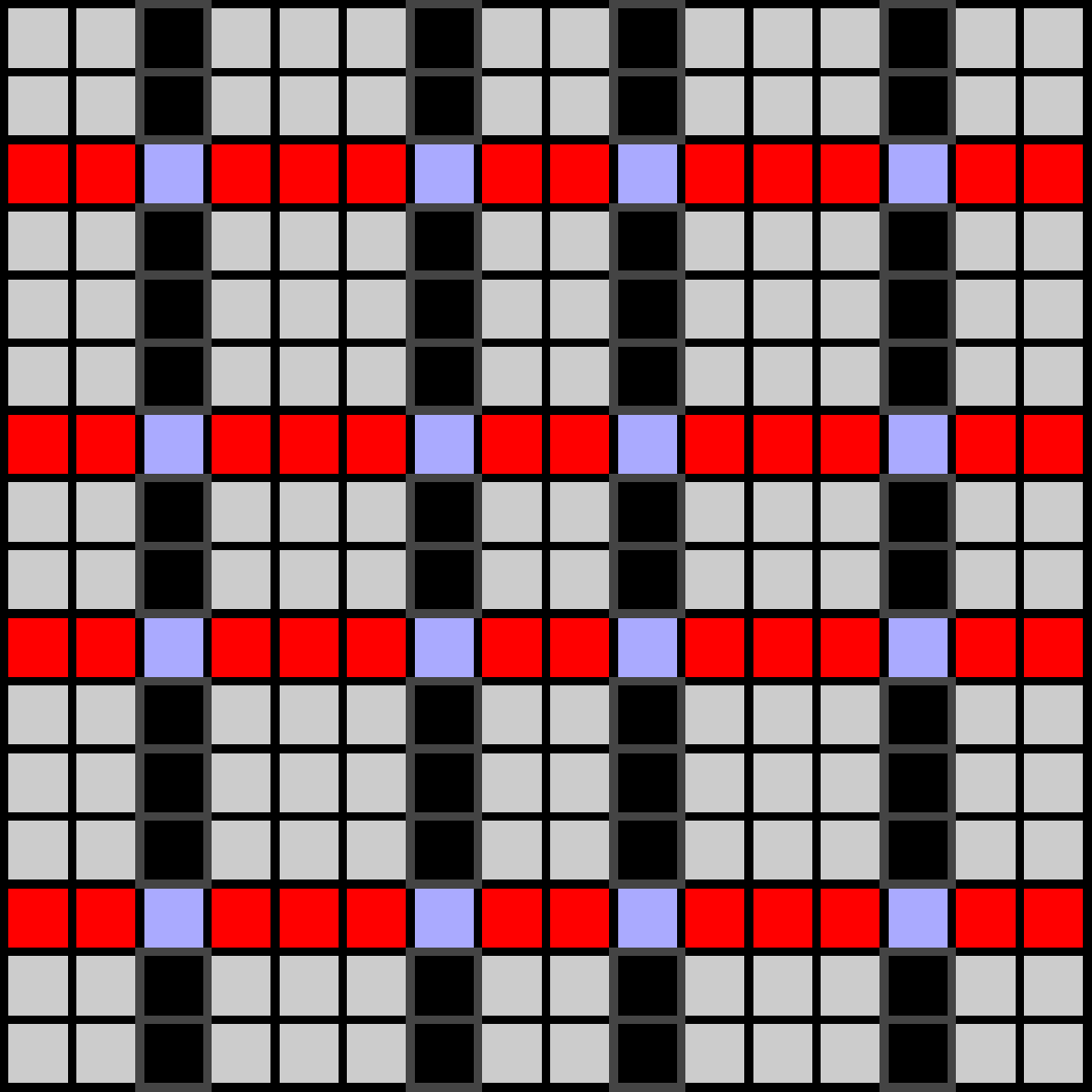}
    % \includesvg[width=0.2\columnwidth]{images/fibonacci_cube/square}
    \caption{A projection of the dual Turtle tiling modulo $nv_1$, $nv_2$}
    \label{fig:squareprojection}
\end{figure}

As in the first proof of non-periodicity, assuming the dual Turtle tiling is periodic, we can find some $n\in\mathbb{N}$ for which $nv_1$ and $nv_2$ preserve the Turtle tiling where $v_1$ and $v_2$ are basis vectors of translations preserving the Rhombille tiling. Modulo $nv_1$ and $nv_2$, the three distinct images under the projections $X$, $Y$ and $Z$ are three $n \times n$ square grids.

Each red\&black hex of a dual Turtle tiling consists of three rhombuses in three orientations, so each red\&black hex contributes one square to each $n \times n$ square grid image.

By definition, each GAB is a line of pairs of rhombuses. Consider the three images of each GAB under the projections $X$, $Y$ and $Z$. One image will be pairs of line segments (no area) and the other two images will be rows or columns of squares (with half the rhombs mapped to squares and the other half contracted to line segments).

Recall that tiling property~\ref{prop:intersectinggabs} says each flipped tile lies at the intersection of three GABs, and two non-parallel GABs intersect at a flipped tile. Under our three projections, this tiling property corresponds to the image of a flipped tile being a square lying at the intersection of a row and column. These rows and columns intersecting at the flipped tile are the images of GABs.

Following the first proof of non-periodicity, let $k$ be the number of GABs (modulo $nv_1$ and $nv_2$) in each direction. Writing $a=k$ and $b=n-k$, we find that each of the three $n\times n$ square grids is partitioned into
\begin{itemize}
    \item $a^2$ images of flipped hex tiles.
    \item $b^2$ images of non-flipped hex tiles.
    \item $ab$ images of black rhomb tiles from the GABs.   
    \item $ab$ images of red rhomb tiles from the GABs.    
\end{itemize}

The proportion of (red : black : red\&black) rhombs in the dual tiling is ($1:1:3$), and we now know these values to be ($ab$ : $ab$ : $a^2+b^2$). In particular, $a$ and $b$ are integer solutions of $a^2 + b^2 = 3ab$.

Since the only integer solutions are $a=0$ and $b=0$, we arrive at a contradiction. Alternatively, we can introduce the ratio of flipped:non-flipped tiles, $r:=\frac{a}{b}$, and solve $r^2+1=3r$ to get the irrational value $r=\frac{1}{2}(3\pm \sqrt{5})$. A contradiction, since $r$ is rational.

\end{proof}

So these projections can be used to show that periodic tilings are not possible. They can also show that non-periodic tilings \emph{are} possible.

If we take the Fibonacci cube construction defined earlier from the palindromic sequences $B_n$, then the projections of the resulting slice are of the form shown in Figure~\ref{fig:squarelowerdiagonal}: the projection consists of the squares that lie strictly below the diagonal. Note that the diagonal itself consists of the images of $a$ flipped hexes and $b$ non-flipped hexes (using notation from the proof). The proportion of (red : black : red\&black) rhombs is then
($ab$ : $ab$ : $a^2+b^2-(a+b)$). Whenever we have an integer solutions to $a^2 + b^2 -a - b= 3ab$, we find that the hexes and rhombs are in the correct proportions as required by a dual Turtle tiling.

\begin{figure}[htb]
    \centering
    \resizebox{0.2\columnwidth}{!}{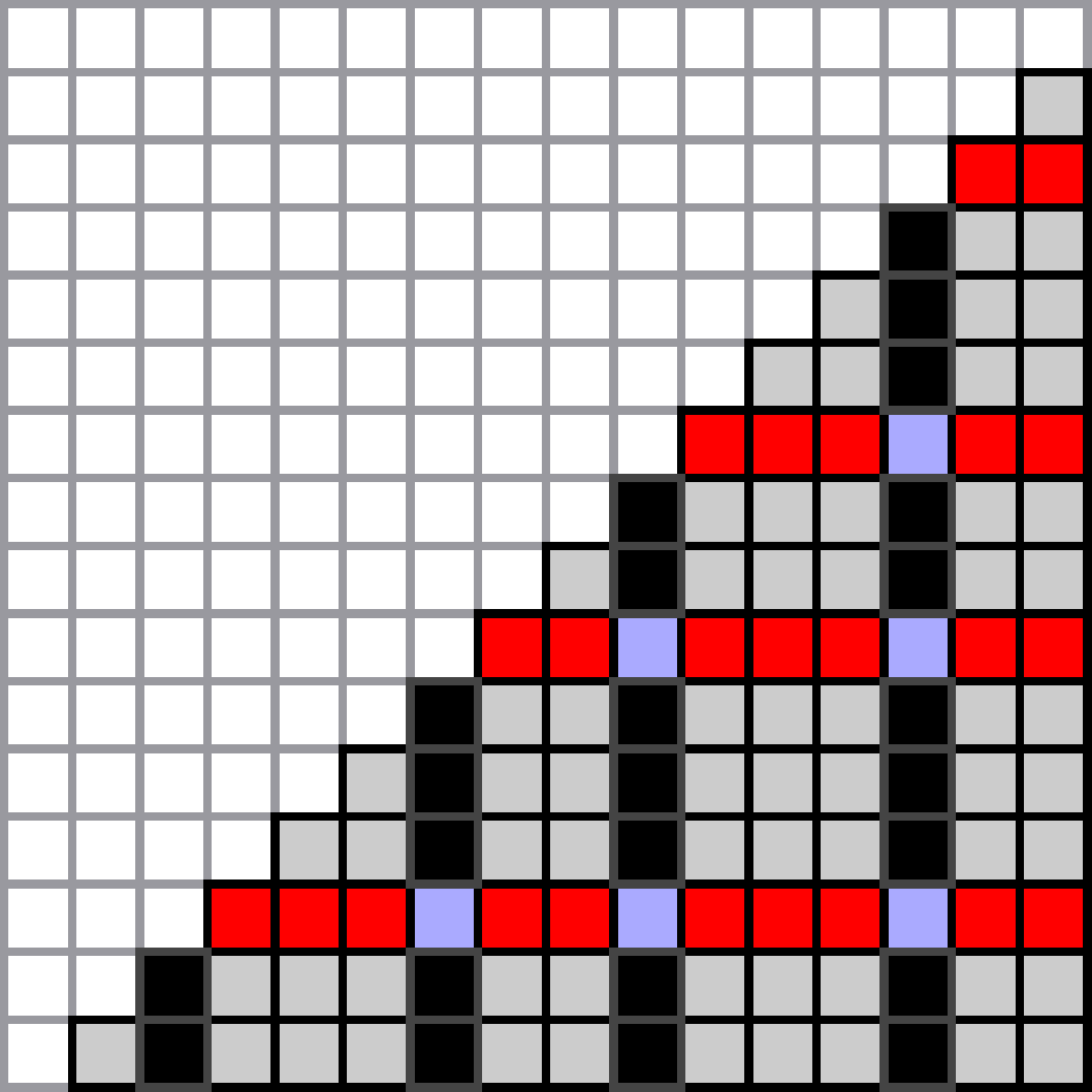}
    % \includesvg[width=0.2\columnwidth]{images/fibonacci_cube/square_lower_diagonal}
    \caption{Fibonacci cube projection}
    \label{fig:squarelowerdiagonal}
\end{figure}

Indeed, for the sequence $B_n$, we have $a = \left| B_n \right|_{odd} = F_{2n-1} - 1$ and $b=\left| B_n \right|_{even} = F_{2n+1} - 1$. Substituting $a = x-1$ and $b=y-1$ and rearranging, we require integer solutions to $xy = (x-y)^2 + 1$.

Fibonacci numbers satisfy Cassini's identity $F_{j-1}F_{j+1} = F_j^2 + (-1)^j$, which in the case of $j=2n$ gives
\[
F_{2n-1}F_{2n+1} = F_{2n}^2 + 1 = (F_{2n+1}-F_{2n-1})^2 +1
\]
so that $xy = (x-y)^2+1$ as required.

A full proof that the Fibonacci cube construction generates Turtle tilings would require some extra work. Specifically, we should show not just that the red rhombs and red\&black hexes appear in equal numbers, but that the 1:1 correspondence is from neighbouring tiles that can be joined into dual Turtles satisfying the matching rules. Since we've already encountered the same construction in a different guise in Section~\ref{section:turtlesubs}, a full proof is not attempted. 
\FloatBarrier
\section{The Hat}

With respect to the matching rules, the Hat tile is degenerate since the red rhombus $\mathcal{R}_{120-\alpha}$ collapses to a line as $\alpha \rightarrow 120^{\circ}$.

\begin{figure}[h]
    \centering
    \resizebox{0.1\columnwidth}{!}{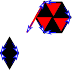}
    % \includesvg[width=0.1\columnwidth]{images/the_hat/hat_tiles_matching_rules}
\end{figure}

The Turtle tile, with $\alpha = 60^\circ$, acts as a canonical representative for the family of tiles. As we have seen in the previous section, the way that it lines up with an underlying Rhombille tiling makes it possible to prove aperiodicity.

However, despite the obvious degeneration, the Hat tile is also well behaved with respect to the matching rules. This is because the red\&black hexagon can be decomposed in to component rhombuses in a way that is complementary to the Turtle tile's decomposition. Figure~\ref{fig:hattilescomponents} shows this choice of component rhombuses with respect to which the Hat tiling `snaps' to a complementary Rhombille tiling

\begin{figure}[htb]
    \centering
    \begin{tabular}{cc}
        \resizebox{0.1\columnwidth}{!}{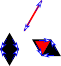}
          % \includesvg[width=0.1\columnwidth]{images/the_hat/hat_tiles_components}
          &
         \resizebox{0.2\columnwidth}{!}{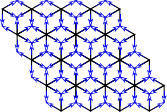}
          % \includesvg[width=0.2\columnwidth]{images/the_hat/complementary_rhombic_tiling}
          \\
    \end{tabular}
    \caption{Components of the dual Hat tile set and their underlying rhombic lattice}
    \label{fig:hattilescomponents}
\end{figure}

\subsection{Alternative proofs of aperiodicity}

The proofs of non-periodicity in proposition \ref{proposotion:non-periodic}, can be informally summarized as follows:

\begin{itemize}
    \item In reality, the Turtle tiling exists and is uniquely forced by the matching rules.
    \item So study the effects of the matching rules to pin down implied tiling properties.
    \item Because the tiling is not periodic, we can collect enough basic properties to rule out the possibility of periodicity.
\end{itemize}

Previously, we ruled out periodicity by following the proof in \cite{akiyama2023alternative} and counting GABs. Since these occur with irrational frequency, they can't belong to a periodic tiling.

We then added a similar proof based on counting the ratio of flipped to non-flipped tiles.

The original paper \cite{smith2023aperiodic} contains a novel proof by comparing the two degenerate end points of the $\text{Tile(}a,b\text{)}$ family. We now add a similar proof by comparing the dual Turtle with the dual Hat tiling:

\subsubsection{Third proof of non-periodicity (sketch): morphing Turtles into Hats}
\begin{proof}
As before, assuming periodicity, we find some $n\in\mathbb{N}$ for which $nv_1$ and $nv_2$ preserve the Turtle tiling where $v_1$ and $v_2$ are basis vectors of translations preserving the Rhombille tiling. Following the previous proof, there are the same number ($k$) of GABs parallel to each of $v_1$ and $v_2$ [and also parallel to $\frac{1}{2}(v_1+v_2)$].

Letting $\alpha \rightarrow 120$, the dual Turtle tiling is transformed into a dual Hat tiling and the translations $nv_1$ and $nv_2$ are transformed into translations $t_1^\prime$ and $t_2^\prime$ that preserve both the dual Hat tiling and also \emph{its} underlying Rhombille tiling. As described in section 3 of \cite{smith2023aperiodic}, this transformation $nv_i \mapsto t_i^\prime$ is an affine map between the spaces of translations preserving the two (supposedly periodic) tilings.

Furthermore, we can track the affect of this affine transformation relative to paths crossing the tiles. As illustrated in Figure~\ref{fig:gab_paths}, paths crossing GABs are skewed relative to the direction of the GAB, whilst paths along the red\&black hexes are unaltered.

But the translation vectors $nv_1$ and $nv_2$ each cross exactly $k$ GABs and the orientations of tiles at each of these crossings is the same . So the affine transformation scales and rotates each of $nv_1$ and $nv_2$ by the same amount.

The Turtle tile set consists of 5 rhombuses whilst the Hat tile set has 4, so we see that the affine map scales areas by $\frac{4}{5}$. Since it acts on the vectors $nv_1$ and $nv_2$ equally, it must scale these vectors by $\frac{2}{\sqrt{5}}$.

We arrive at a contradiction because $\sqrt{5}$ can't appear as the distance of any translation vector preserving the triangular lattice that the vertices of the Rhombille tiling belong to. The proof in \cite{smith2023aperiodic} considered this for $\sqrt{2}$, but it holds more generally for the square root of any prime $p$ with $p\equiv 2$ (mod 3). The integers that \emph{can} appear as norms of this $A_2$ lattice are known as Löschian numbers, OEIS[A003136].
\end{proof}

\begin{figure}[ht]
    \centering
    \resizebox{0.15\columnwidth}{!}{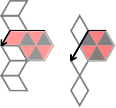}
    % \includesvg[width=0.15\columnwidth]{images/the_hat/gab_paths}
    \caption{Image of a GAB in the Hat tiling and the affect on a path crossing a GAB}
    \label{fig:gab_paths}
\end{figure}

\section{The Spectre}\label{section:spectre}

We conclude with a less formal section exploring the chiral aperiodic monotile, the Spectre, as discovered in \cite{smith2023chiral}. 

Tile(1,1) can tile the plane periodically using both regular and mirrored copies. But when restricting so that mirror images are not allowed, \cite{smith2023chiral} observes that the resulting Spectre tile is forced to tile the plane in a non-periodic way.

In seeking to better understand the Spectre, we take Theorem 3.1 of \cite{smith2023chiral} as our starting point: A Spectre tiling is combinatorially equivalent to a tiling by Hats and Turtles. In particular, we consider the ``Hats in Turtles" tiling: replacing the ``odd" Spectre tiles with Hats and the remainder with Turtles.

The dual version of this makes use of the dual Hat tiles introduced in Figure~\ref{fig:hattilescomponents}, but since the dual Hat always ocurrs together with its black rhomb, we can reduce the dual tile set from four elements to three as shown in Figure~\ref{fig:hatsinturtlestileset}.

\begin{figure}[htb]
    \centering
    \begin{tabular}{c|c}
        \resizebox{0.1\columnwidth}{!}{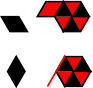}
        % \includesvg[width=0.1\columnwidth]{images/spectre/matching_rules/dual_tile_set_4}
        \hspace{1em}
        &
        \hspace{1em}
        \resizebox{0.1\columnwidth}{!}{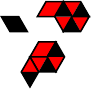}
        % \includesvg[width=0.1\columnwidth]{images/spectre/matching_rules/dual_tile_set_3}
        \\
    \end{tabular}
    \caption{`Hats in Turtles' dual Spectre tile set. Both sets are equivalent.}
    \label{fig:hatsinturtlestileset}
\end{figure}

We aim to draw parallels between Turtle and Spectre tilings. In particular, the ``Hats in Turtle" form of a Spectre tiling can be viewed as consisting of collections of Turtle tiles surrounding a seed tile. The seed tile is a Hat. This is similar to the Turtle tiling where the seed tile is a flipped Turtle. To make this observation a little more precise, note the similarity between Figure~\ref{fig:turtle_forced} for Turtles and Figure~\ref{fig:spectresforced} for Spectres (Hats in Turtles): Placing a single seed `Hat' tile in a Rhombille tiling, results in three sets of Ammann bars extending locally from  the seed tile. Near to the seed, the Turtle tiles are forced to arrange themselves along these Ammann bars. As shown in Figure~\ref{fig:spectresforced}, the result is a cluster of tiles that, if it is to tile the plane, must be forced to tile in a hexagonal arrangement.

In the case of the Turtle tiling, the Golden Ammann Bars (Figure~\ref{fig:hex_turtle_cluster}) pass through the flipped seed tile in a straight line. This results globally in straight Ammann bars in three directions across the Turtle tiling. However, for the Spectre (Hats in Turtles) tiling, the seed tile has the effect of nudging one of the Ammann bars sideways. On the final tiling, the result is Ammann bars that are not straight, but instead drift to the side.

\begin{figure}[htb]
    \centering
    \resizebox{0.7\columnwidth}{!}{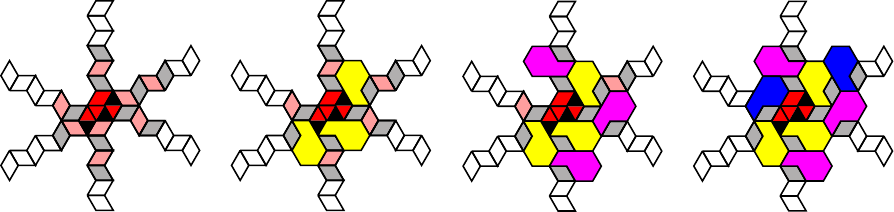}
    % \includesvg[width=0.7\columnwidth]{images/spectre/matching_rules/spectres_forced}
    \caption{Dual Turtle tiles forced to cluster along Ammann Bars around a dual Hat seed tile}
    \label{fig:spectresforced}
\end{figure}

These hexagonal arrangements of clusters are explored in detail in the original paper \cite{smith2023chiral}. Marking these hexagons with an arrow to keep track of their direction, and disallowing reflections, Figure~\ref{fig:hex_spectres} shows the first 4 steps of the spectre substitution rules from \cite{smith2023chiral}. These substitution rules can be explored with the interactive JavaScript app at \cite{csk_spectre_generations}.

\begin{figure}[htb]
\centering
\begin{tabular}{cccc}
     \resizebox{0.1\columnwidth}{!}{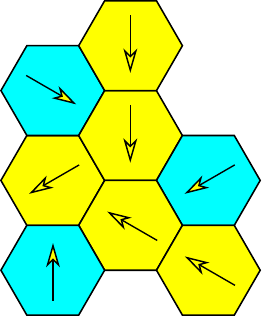}
     % \includesvg[width=0.1\columnwidth]{images/spectre/hex_subs/a1}
     &
     \resizebox{0.25\columnwidth}{!}{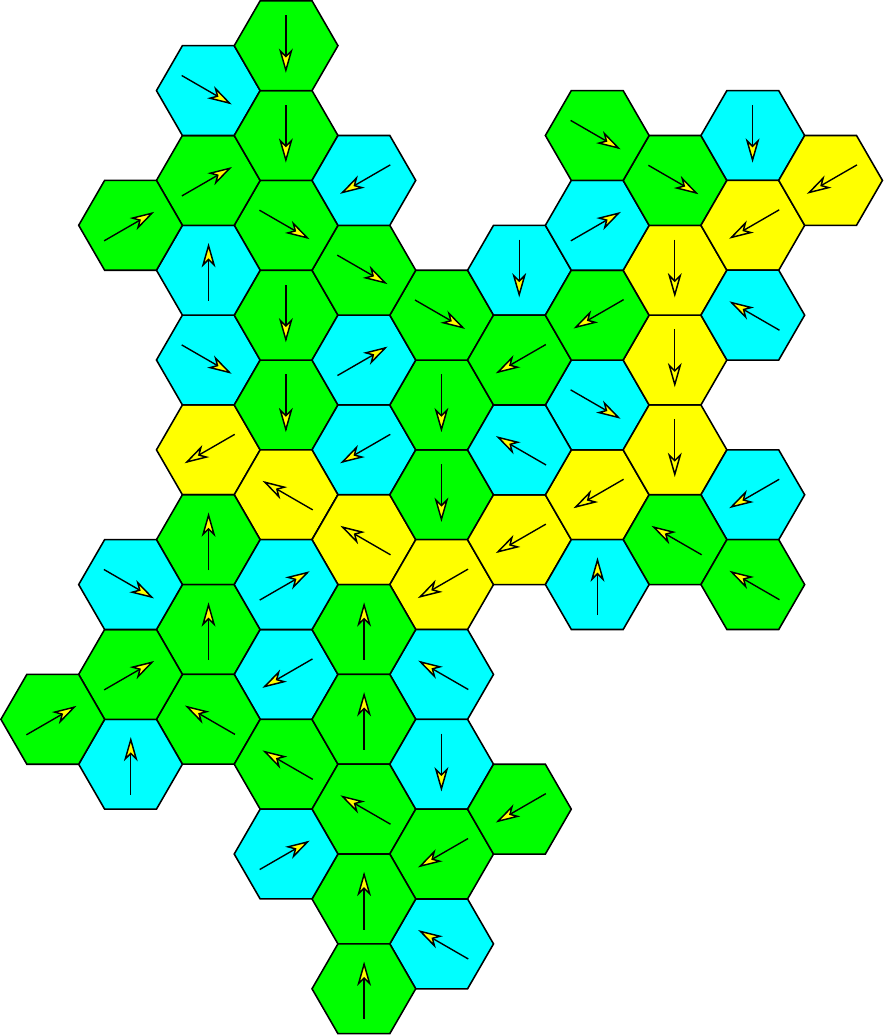}
     % \includesvg[width=0.25\columnwidth]{images/spectre/hex_subs/a2flipped}
     &
     \resizebox{0.25\columnwidth}{!}{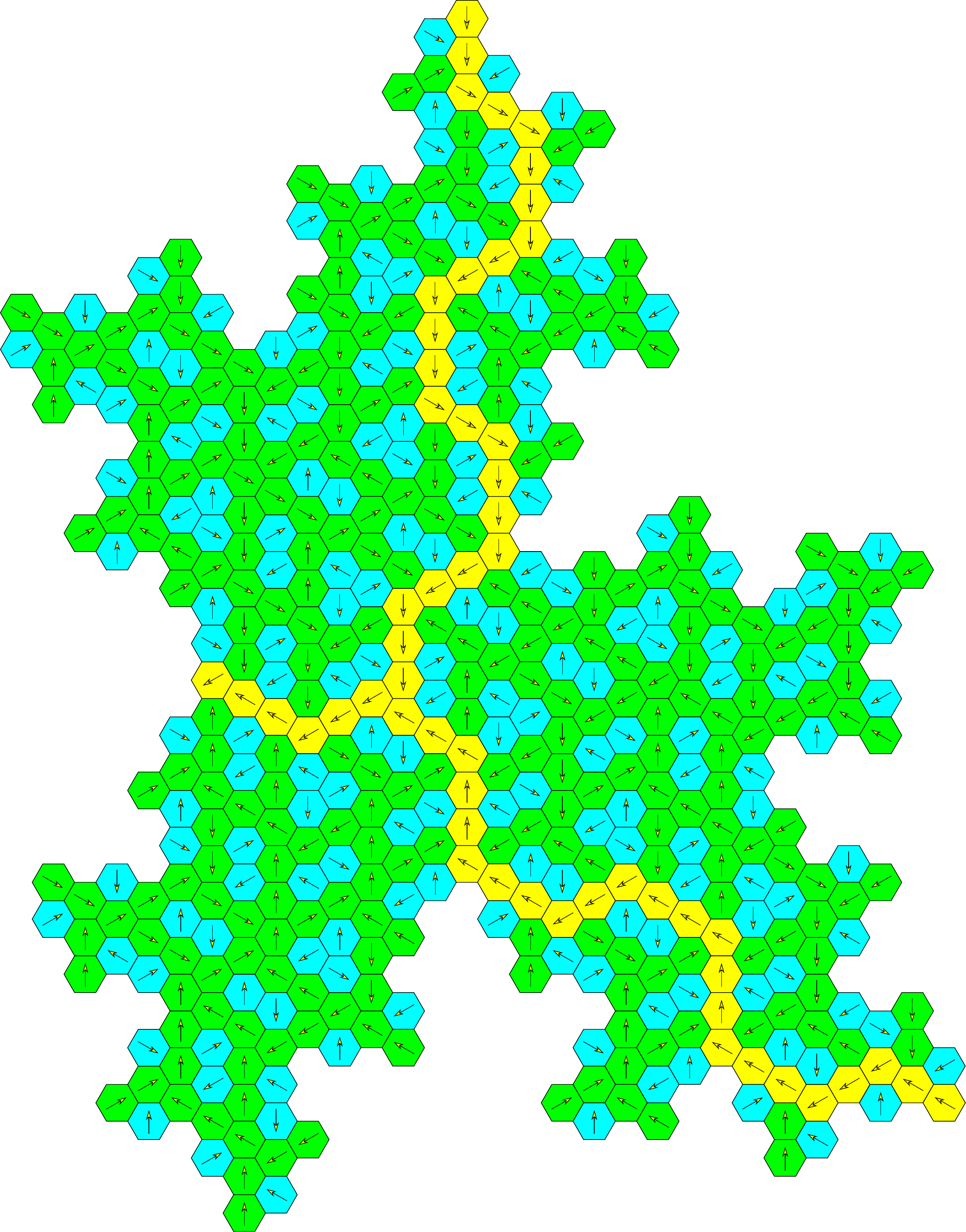}
     % \includesvg[width=0.25\columnwidth]{images/spectre/hex_subs/a3}
     &
     \resizebox{0.25\columnwidth}{!}{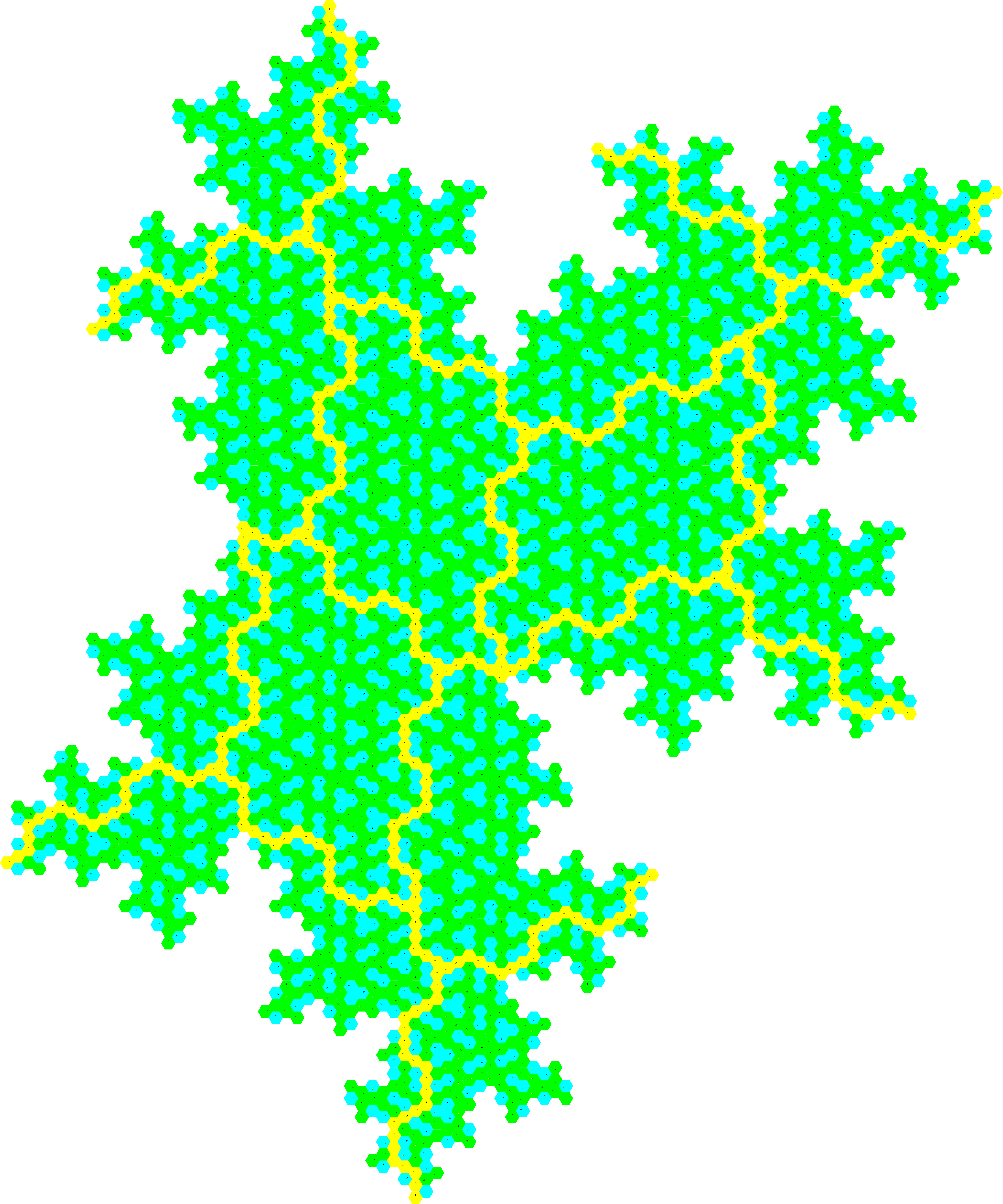}
     % \includesvg[width=0.25\columnwidth]{images/spectre/hex_subs/a4flipped}
     \\
     \hline
     \resizebox{0.1\columnwidth}{!}{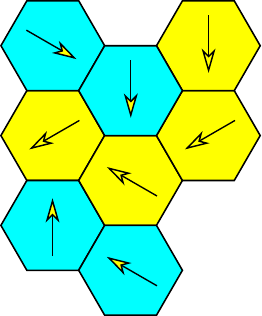}
     % \includesvg[width=0.1\columnwidth]{images/spectre/hex_subs/w1flipped}
     &
     \resizebox{0.25\columnwidth}{!}{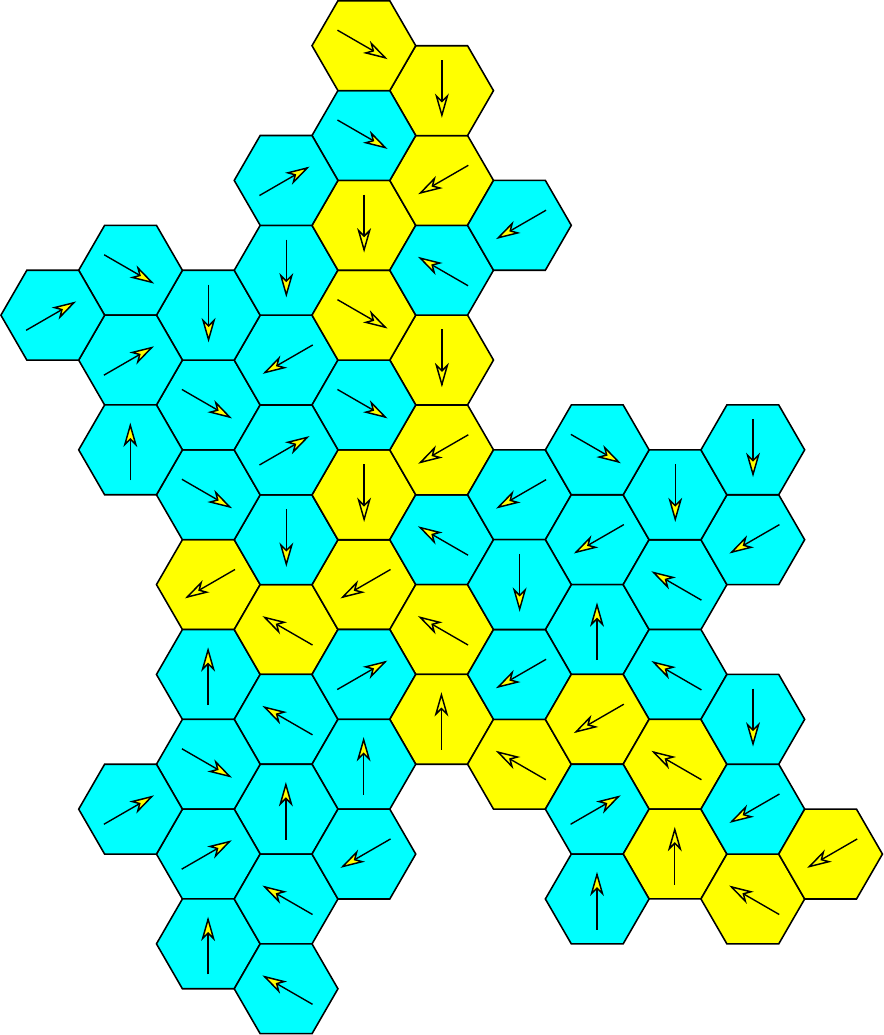}
     % \includesvg[width=0.25\columnwidth]{images/spectre/hex_subs/w2}
     &
     \resizebox{0.25\columnwidth}{!}{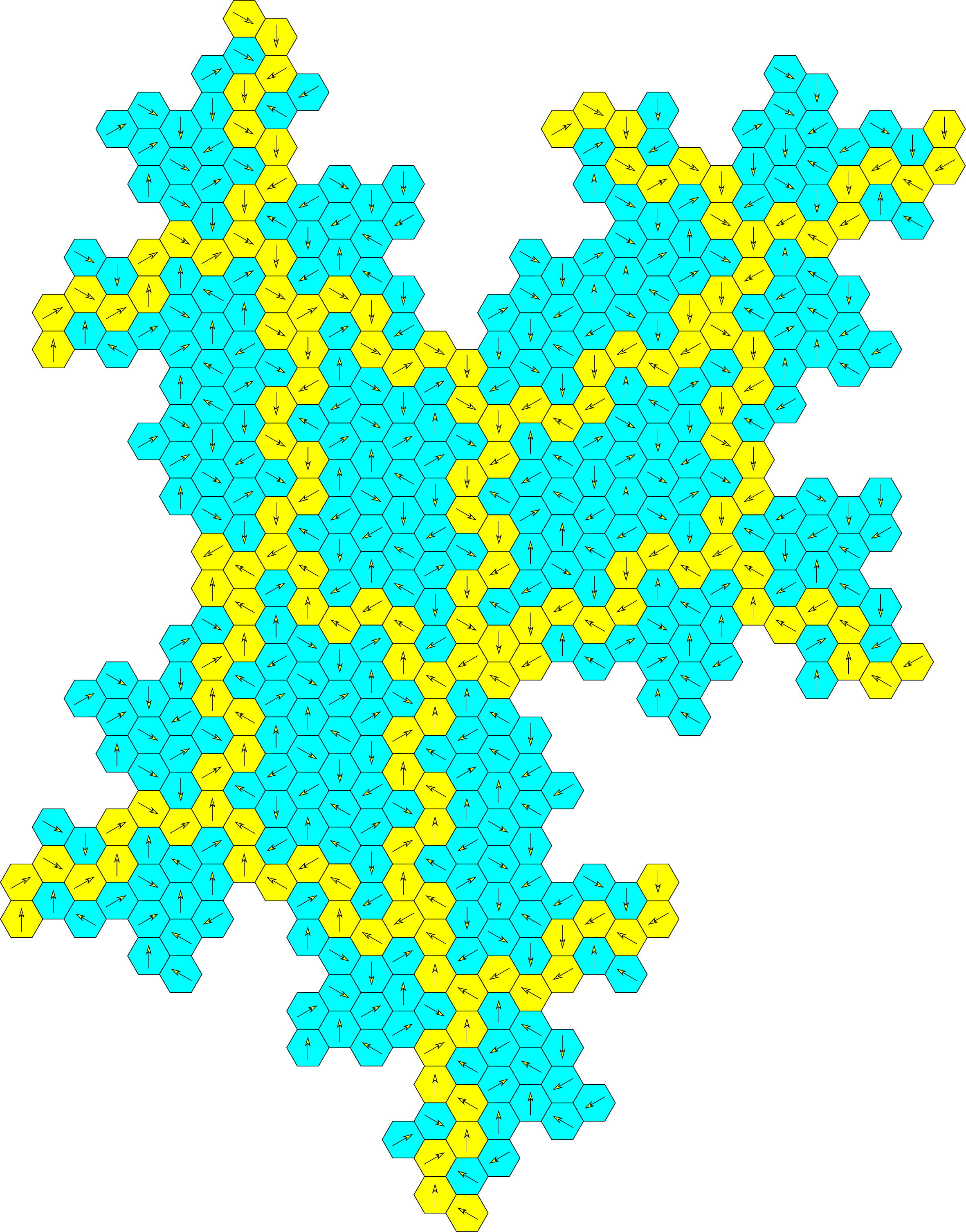}
     % \includesvg[width=0.25\columnwidth]{images/spectre/hex_subs/w3flipped}
     &
     \resizebox{0.25\columnwidth}{!}{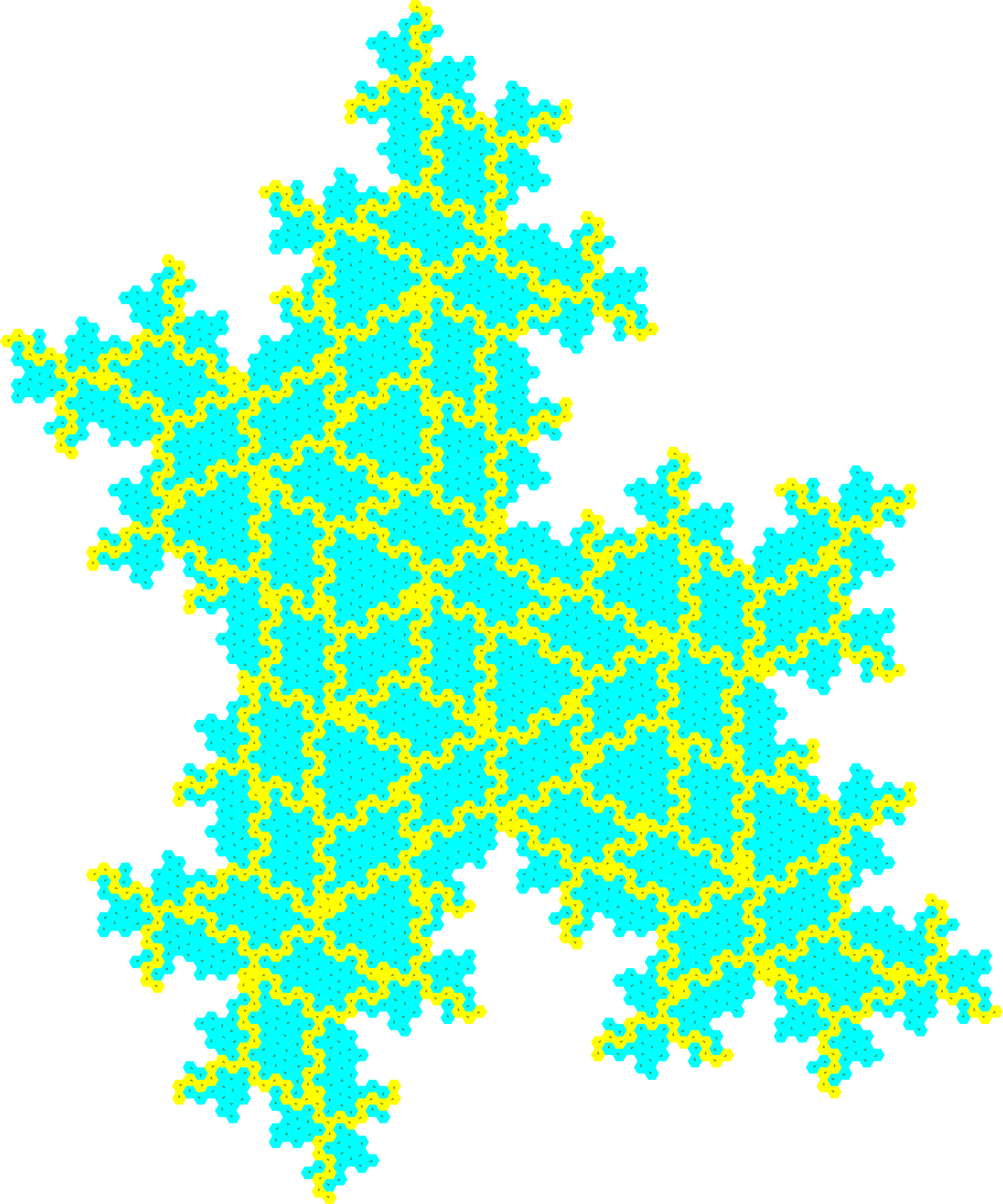}
     % \includesvg[width=0.25\columnwidth]{images/spectre/hex_subs/w4}
     \\
\end{tabular}
\caption{Articulated and Wriggly Spectre hexagons}
\label{fig:hex_spectres}
\end{figure}

In labelling the hexagonal metatile with an arrow, there are six possible choices of direction available. Two of these choices are shown separately in Figure~\ref{fig:hex_spectres} with some paths highlighted. These paths reveal patterns that repeat after two steps of the Spectre substitution rule (each step of the spectre substitution rules mirror the previous arrangement and two steps returns the arrangement back the the previous orientation -- so it can be argued that this pair of steps forms one complete substitution).

These patterns hint at an alternative way of generating a Spectre tiling. In fact, we shall describe a conjugate pair of recursive systems. These recursions are introduced in the next section, and the sense in which they are \emph{conjugate} is discussed in Section~\ref{section:sturmian}.

\FloatBarrier

\subsection{New substitution systems}

This section introduces new sets of rules for tiling the plane with Spectres, as described in \cite{twin_worms}, arising from a collaborative effort with Erhard Künzel and Yoshiaki Araki. I had the pleasure of determining the recursive relations that define Conway worms, which are further explored in Section~\ref{section:sturmian}. The remaining shapes, substitutions, and many of the names introduced here stem from the creativity of Erhard Künzel and Yoshiaki Araki. Notably, Yoshiaki further enriched the work by adding a second ``wriggly" substitution system to complement the initial ``articulated" system, both of which we describe in this section.

Keeping with the theme of dual tilings, these rules are presented here as the dual Hats in Turtles. First, we introduce recursively defined strips of tiles that will act as generators for the system. These consist of two fixed elements which we label $E$ and $O$, and two strips of tiles, labelled $S$ and $I$.

The elements $E$ and $O$ both correspond to the Mystics of \cite{smith2023chiral} and so are identical pairs of Spectre tiles. However, in the final tiling, these Mystics appear in two ways: with the odd tile at the top or with it at the bottom and so we differentiate between even [$E$] and odd [$O$] Mystics.

To recursively define strips $S_k$ and $I_k$, we start with $S_0$ as empty and $I_0$ as a single tile. Then define recursions as shown in Figure~\ref{fig:SandI} by:

\begin{align*}
    I_{k+1} &= O S_{k} I_{k} S_{k} I_{k} S_{k} E \\
    S_{k+1} &= (S_k I_k S_k I_k S_k) E (S_k I_k S_k) O (S_k I_k S_k I_k S_k)
\end{align*}

\begin{figure}[htb]
\setlength{\tabcolsep}{20pt}
\renewcommand{\arraystretch}{1.5}
\begin{tabular}{c c c}
    \begin{tabular}{c|c}
        $E$ & \rotatebox{150}{\resizebox{0.035\columnwidth}{!}{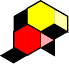}} \\
        % $E$ & \includesvg[scale=0.5,angle=150]{images/spectre/recursion/articulated/worms/even} \\
        $O$ & \rotatebox{150}{\resizebox{0.035\columnwidth}{!}{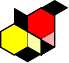}} \\
        % $O$ & \includesvg[scale=0.5,angle=150]{images/spectre/recursion/articulated/worms/odd} \\
        \hline
    \end{tabular}
    &
    \begin{tabular}{c|c}
        $I_0$ & \resizebox{0.015\columnwidth}{!}{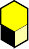} \\
        % $I_0$ & \includesvg[scale=0.5]{images/spectre/recursion/articulated/worms/I0} \\
        $S_0$ & $\varnothing$ \\
        \hline
    \end{tabular}
    &
    \begin{tabular}{c|c}
        $I_1$ & \resizebox{0.15\columnwidth}{!}{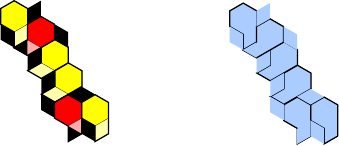} \\
        % $I_1$ & \includesvg[scale=0.4]{images/spectre/recursion/articulated/worms/I1} \\
        $S_1$ & \resizebox{0.18\columnwidth}{!}{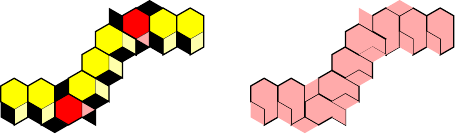} \\
        % $S_1$ & \includesvg[scale=0.4]{images/spectre/recursion/articulated/worms/S1} \\
        \hline
    \end{tabular}
    \\
\end{tabular}
\\
\begin{tabular}{c|c}
$I_2$ & \resizebox{0.15\columnwidth}{!}{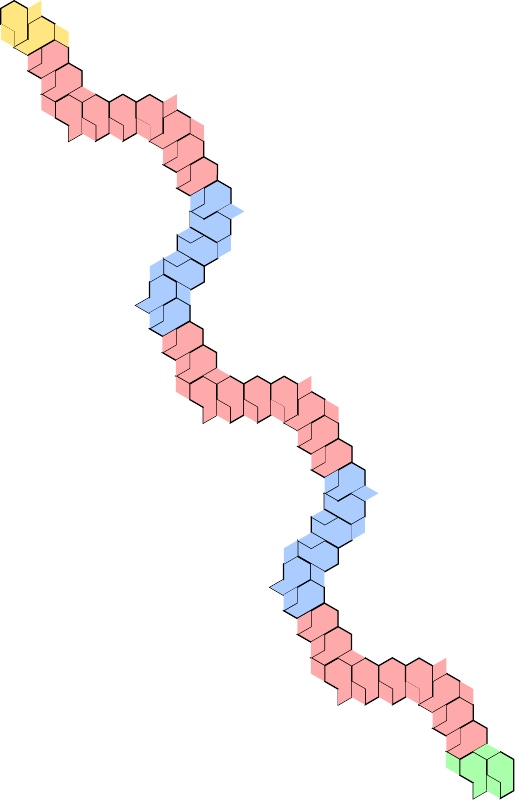} \\
$S_2$ & \resizebox{0.6\columnwidth}{!}{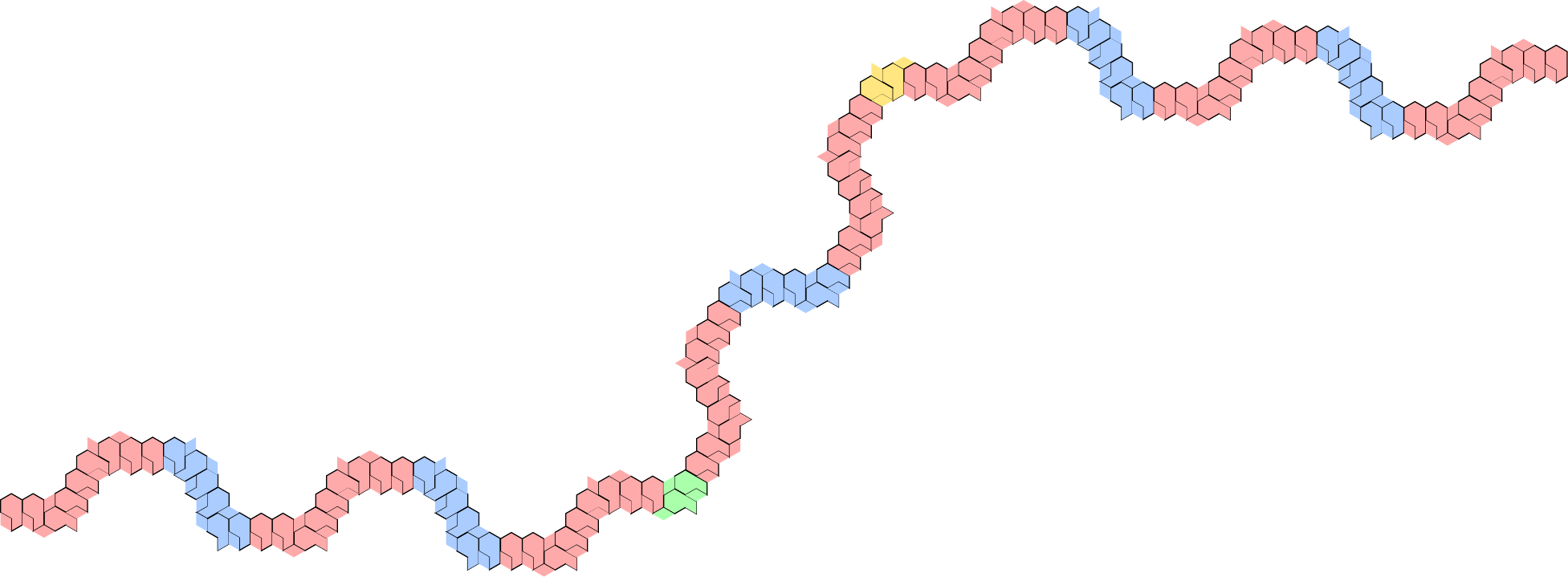} \\
% $I_2$ & \includesvg[scale=0.3]{images/spectre/recursion/articulated/worms/I2} \\
% $S_2$ & \includesvg[scale=0.3]{images/spectre/recursion/articulated/worms/S2} \\
\hline
& $I_{k+1}= O S_{k} I_{k} S_{k} I_{k} S_{k} E$ \\
& $S_{k+1} = (S_k I_k S_k I_k S_k) E (S_k I_k S_k) O (S_k I_k S_k I_k S_k) $ \\ 
\end{tabular}
\caption{Recursions for the generating elements $S$ and $I$}
\label{fig:SandI}
\end{figure}

In addition to the $S_k$ elements, it will be helpful to define two further strips labelled $N$ and $M$ and defined recursively as below and shown in Figure~\ref{fig:articulated_m_n}.

\begin{align*}
    N_{k+1} &= S_{k} I_{k} S_{k} \\
    M_{k+1} &= S_{k} I_{k} S_{k} I_{k} M_{k}
\end{align*}

\begin{figure}[htb]
\setlength{\tabcolsep}{20pt}
\renewcommand{\arraystretch}{1.5}
\centering
\begin{tabular}{c c}
    \begin{tabular}{c|c}
    
        $N_1$ & \resizebox{0.022\columnwidth}{!}{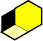} \\
        % $N_1$ & \includesvg[scale=0.5]{images/spectre/recursion/articulated/worms/N0} \\
        \hline
    \end{tabular}
    &
    \begin{tabular}{c|c}
        $N_2$ & \resizebox{0.35\columnwidth}{!}{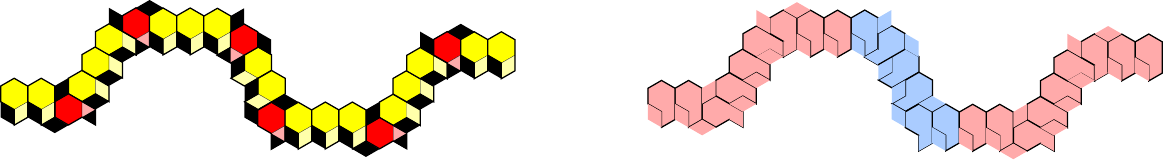}  \\
        % $N_2$ & \includesvg[scale=0.3]{images/spectre/recursion/articulated/worms/N1} \\
        \hline
    \end{tabular}
    \\
\end{tabular}
\\
\vspace{2em}
\begin{tabular}{c|c}
    $N_3$ & \resizebox{0.44\columnwidth}{!}{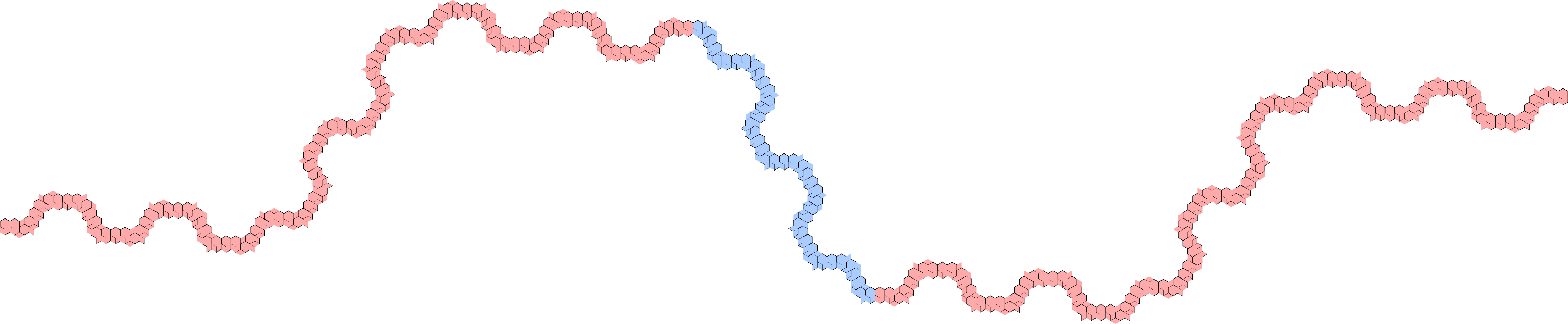}  \\
    % $N_3$ & \includesvg[scale=0.1]{images/spectre/recursion/articulated/worms/N2} \\
    \hline
    & $N_{k+1}= S_{k} I_{k} S_{k}$ \\
\end{tabular}
\\
\vspace{2em}
\begin{tabular}{c c}
    \begin{tabular}{c|c}
        $M_1$ & \resizebox{0.08\columnwidth}{!}{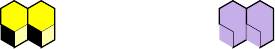} \\
        % $M_1$ & \includesvg[scale=0.5]{images/spectre/recursion/articulated/worms/M0} \\
        \hline
    \end{tabular}
    &
    \begin{tabular}{c|c}
        $M_2$ & \resizebox{0.35\columnwidth}{!}{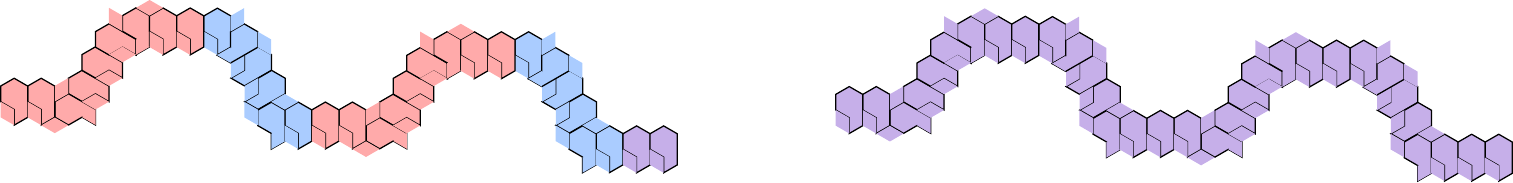} \\
        % $M_2$ & \includesvg[scale=0.3]{images/spectre/recursion/articulated/worms/M1} \\
        \hline
    \end{tabular}
\end{tabular}
\\
\vspace{2em}
\begin{tabular}{c|c}
    $M_3$ & \resizebox{0.5\columnwidth}{!}{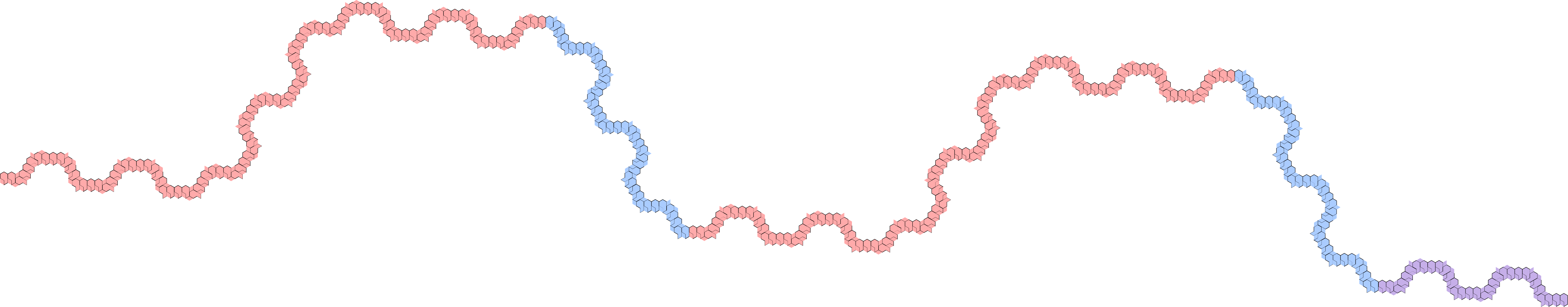} \\
    % $M_3$ & \includesvg[scale=0.1]{images/spectre/recursion/articulated/worms/M2} \\
    \hline
    & $M_{k+1}= S_{k} I_{k} S_{k} I_{k} M_{k}$ \\
\end{tabular}
\caption{Additional elements $M_k$ and $N_k$}
\label{fig:articulated_m_n}
\end{figure}

We can consider the elements $E$ and $O$ as zero dimensional since they are fixed and so become increasingly indistinguishable from points with respect to the expanding tiling. The elements $S_k$, $N_k$ and $M_k$, however, increase in length and can be considered as one dimensional components of the tiling. We now introduce six 2-dimensional components. These are labelled as PA, PB and TX for X in \{A,B,C,D\}. Here, P denotes a parallelogram-like shape with 2-fold symmetry and T denotes a triangular shape with 3-fold symmetry. Where more descriptive names are helpful, these shapes are called \emph{Rose} [TD], \emph{Rhomb} [PA], \emph{Large propeller} [TB], \emph{Small propeller} [TC], \emph{Penguin} [PB] and \emph{Bird} [TA].

Similarly to the Turtle recursive rules, the strips $S_k$ and $I_k$ are palindromes: the top and bottom edges of the strips are identical, and the start and ends edges differ. Three of the 2D shapes include a prototile protrusion that accept the strips in one direction as shown in Figure~\ref{fig:aligned}.

\begin{figure}[htb]
    \centering
    \begin{tabular}{ccc}
        \resizebox{0.13\columnwidth}{!}{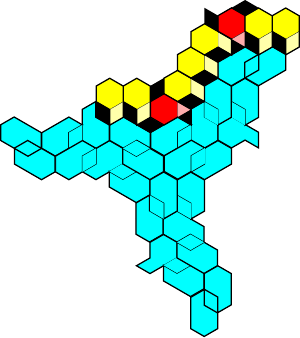}
        % \includesvg[scale=0.5]{images/spectre/recursion/articulated/Sjoin0}
        &
        \resizebox{0.3\columnwidth}{!}{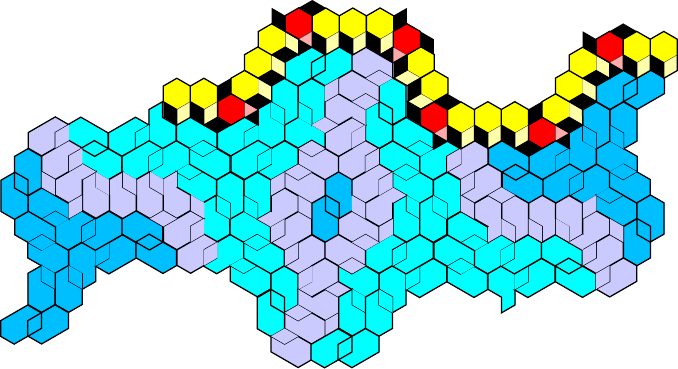}
        % \includesvg[scale=0.5]{images/spectre/recursion/articulated/Njoin}
        &
        \resizebox{0.2\columnwidth}{!}{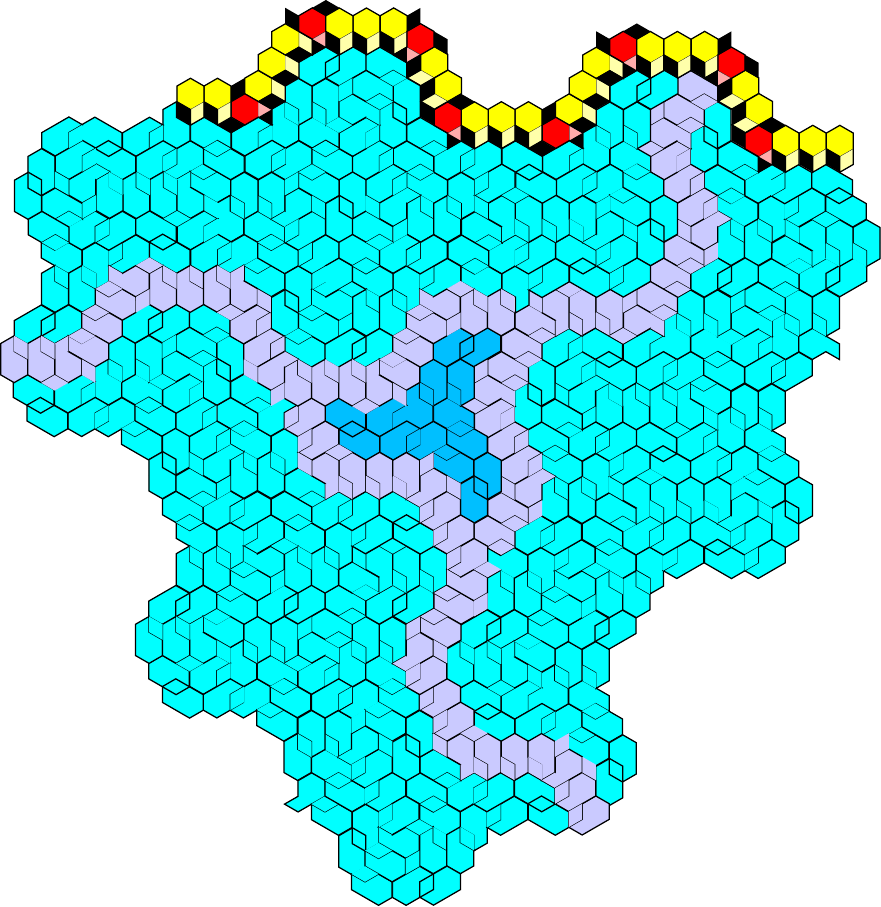}
        % \includesvg[scale=0.25]{images/spectre/recursion/articulated/Mjoin}
        \\
    \end{tabular}
    \caption{Conway worms forced to align with 2D shapes}
    \label{fig:aligned}
\end{figure}
\FloatBarrier

Figure~\ref{fig:three-step-system} shows the elegant substitution system discovered by Erhard Künzel. This consists of three steps repeating cyclically to cover the plane. In each step, a pair of shapes plus one smaller shape from the previous step are joined together by Conway worms of types $S$, $N$ or $M$. By cycling through these three steps, we return back to the first set of shapes. In this way, the substitutions can also be expressed in a single step using fewer shapes, but having a larger inflation factor. The one step substitution is explored in \cite{twin_worms} and also, to a lesser extent, below where we consider substitution matrices. 

\begin{figure}[htb]
\begin{tabular}{c c c}
$\mathrm{TD}_1$ \resizebox{0.03\columnwidth}{!}{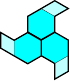}
% $\mathrm{TD}_1$ \includesvg[scale=0.5]{images/spectre/recursion/articulated/TD0}
&
$\mathrm{TD}_2$ \includegraphics[scale=0.25]{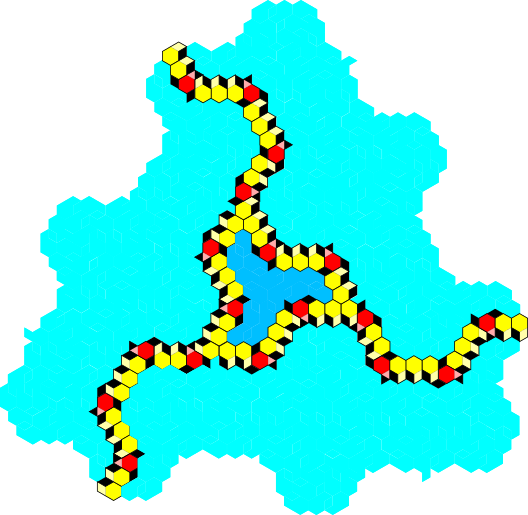}
&
$\mathrm{td}_{k} = 3\mathrm{pb}_{k-1} + \mathrm{tc}_{k-1} + 3\mathrm{n}_k$ 
\\
$\mathrm{PA}_1$ \rotatebox{120}{\resizebox{0.02\columnwidth}{!}{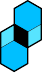}}
% $\mathrm{PA}_1$ \includesvg[scale=0.5,angle=120]{images/spectre/recursion/articulated/PA0} 
&
$\mathrm{PA}_2$ \includegraphics[scale=0.25]{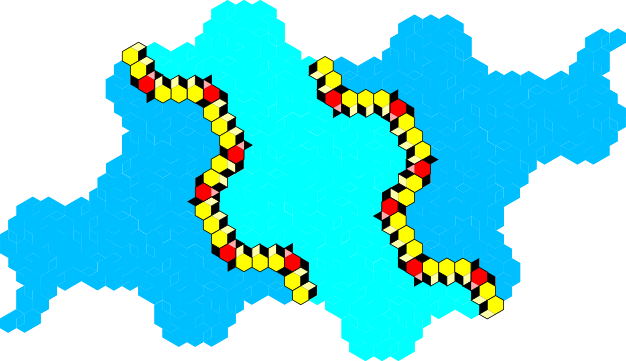}
& $\mathrm{pa}_{k} = \mathrm{pb}_{k-1} + 2\mathrm{ta}_{k-1} + 2\mathrm{n}_k$
\\
\hline
$\mathrm{TB}_1$ \resizebox{0.15\columnwidth}{!}{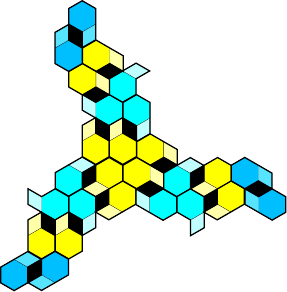}
% $\mathrm{TB}_1$ \includesvg[scale=0.5]{images/spectre/recursion/articulated/TB0}
&
$\mathrm{TB}_2$ \includegraphics[scale=0.1]{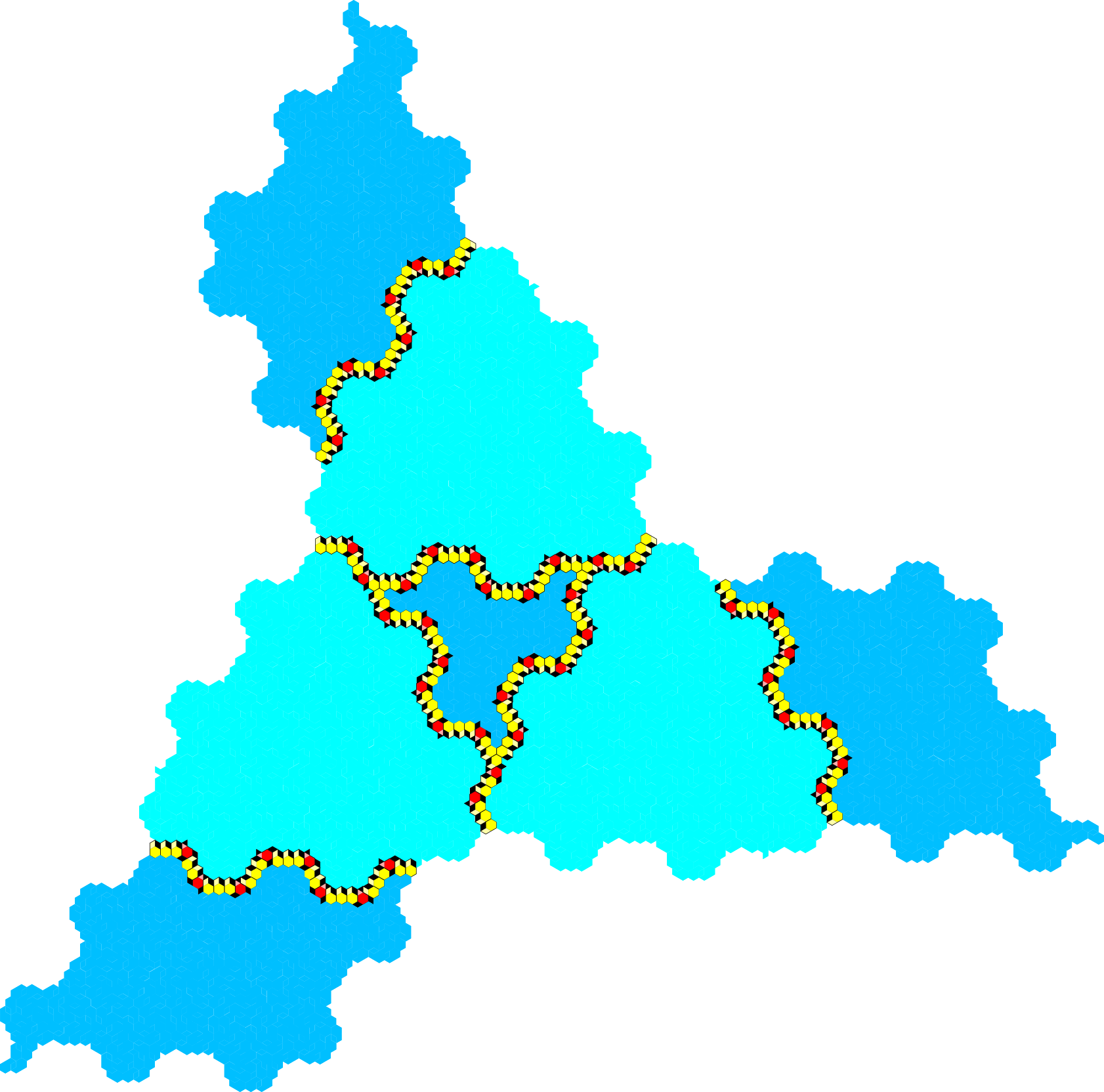}
&
$\mathrm{tb}_{k} = 3\mathrm{td}_k + 3\mathrm{pa}_k + \mathrm{ta}_{k-1} + 6\mathrm{m}_k $
\\
$\mathrm{TC}_1$ \resizebox{0.085\columnwidth}{!}{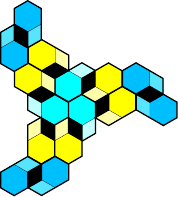}
% $\mathrm{TC}_1$\includesvg[scale=0.5]{images/spectre/recursion/articulated/TC0} 
& 
$\mathrm{TC}_2$ \includegraphics[scale=0.1]{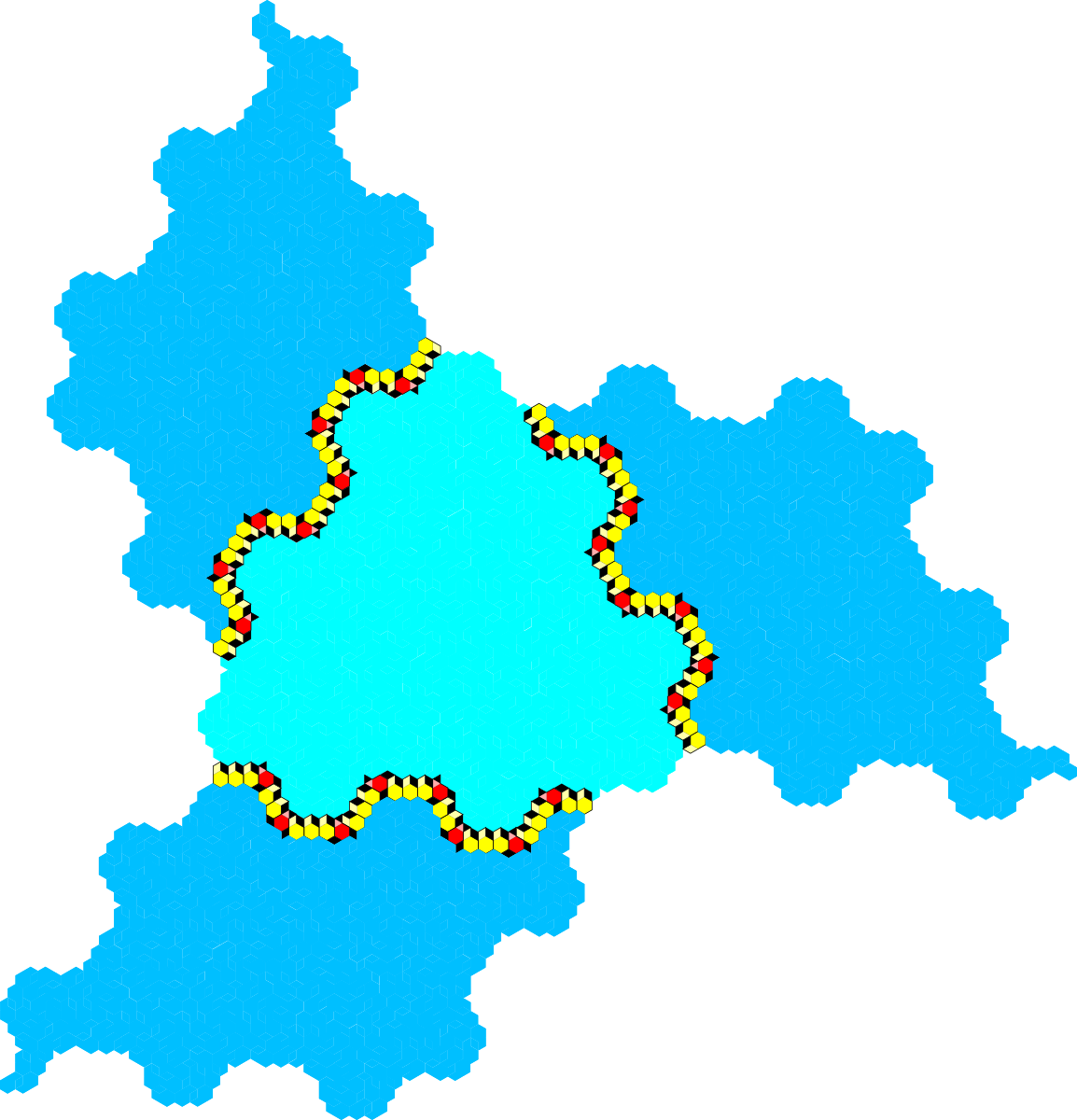}
&
$\mathrm{tc}_{k} = \mathrm{td}_k + 3\mathrm{pa}_k + 3\mathrm{m}_k$
\\
\hline

$\mathrm{PB}_1$ \resizebox{0.18\columnwidth}{!}{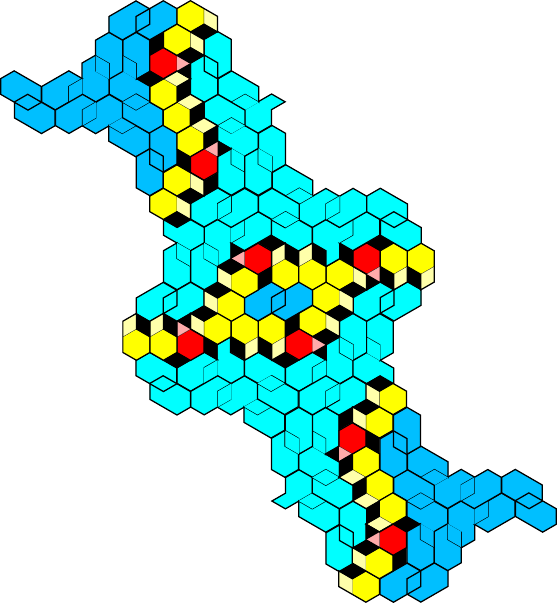}
% $\mathrm{PB}_1$ \includesvg[scale=0.35]{images/spectre/recursion/articulated/PB0}
&
$\mathrm{PB}_2$ \includegraphics[scale=0.04375,angle=120]{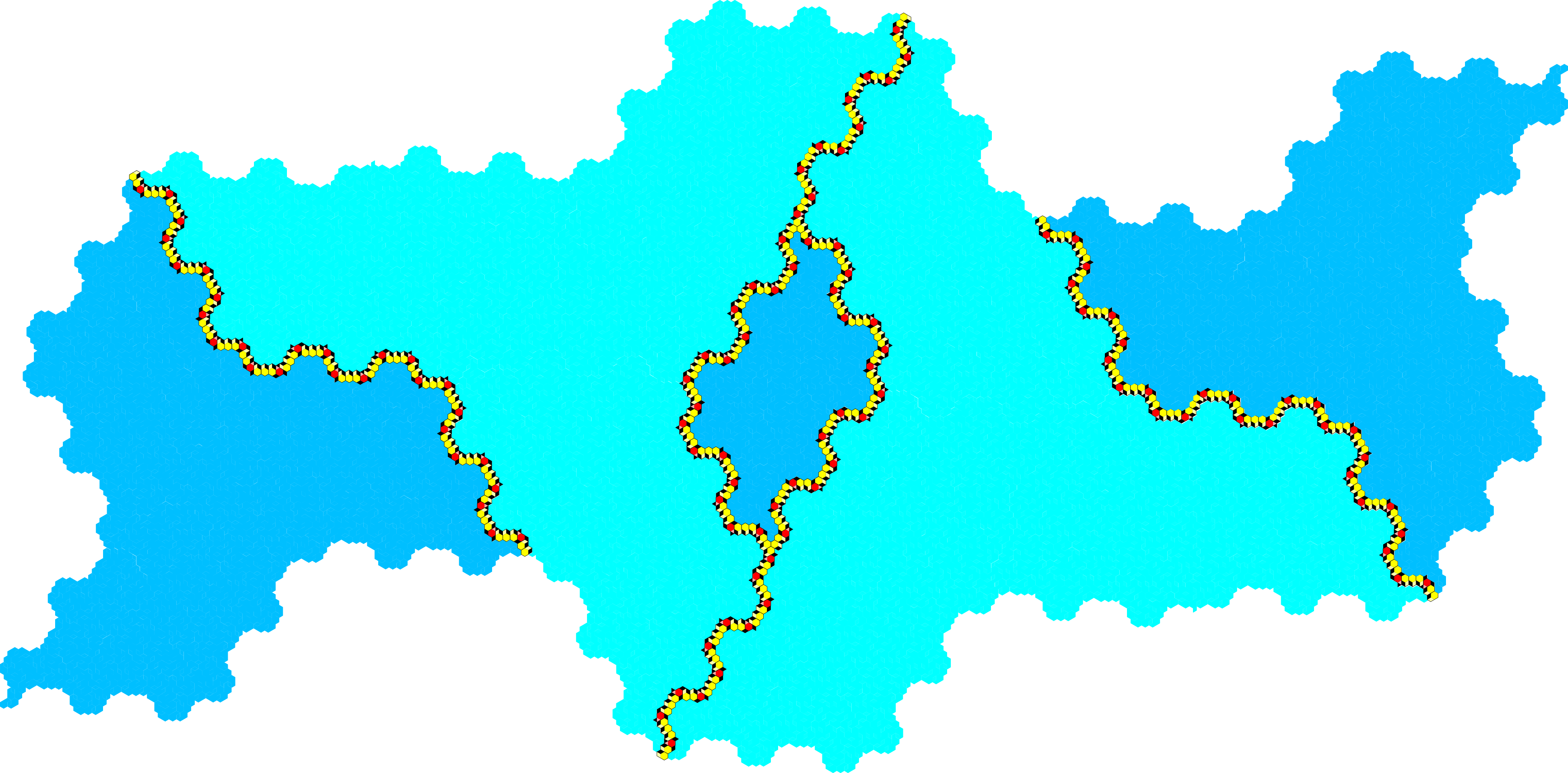}
&
$\mathrm{pb}_{k} = 2\mathrm{tb}_k + 2\mathrm{tc}_k + \mathrm{pa}_k + 4\mathrm{s}_k$
\\
$\mathrm{TA}_1$ \resizebox{0.17\columnwidth}{!}{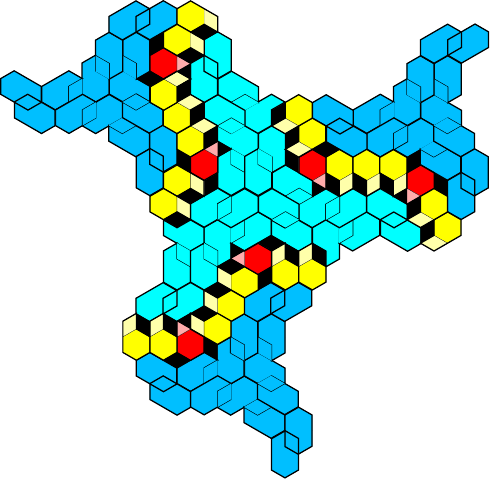}
% $\mathrm{TA}_1$ \includesvg[scale=0.35]{images/spectre/recursion/articulated/TA0} 
&
$\mathrm{TA}_2$ \includegraphics[scale=0.04375]{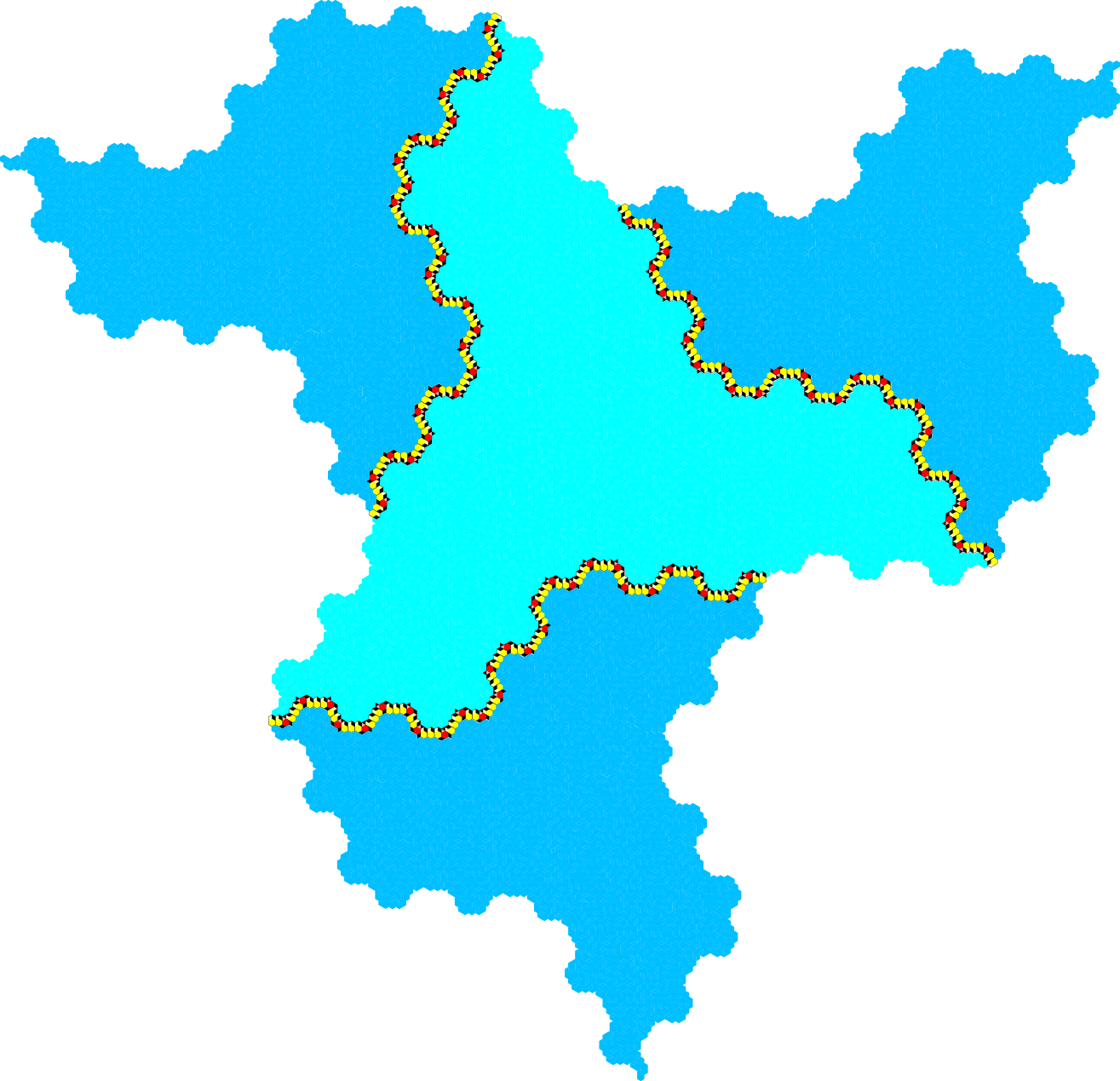}
&
$\mathrm{ta}_{k} = \mathrm{tb}_k + 3\mathrm{tc}_k + 3\mathrm{s}_k$
\\\vspace{1em}
\end{tabular}
\caption{Three-step Spectre substitution system}
\label{fig:three-step-system}
\end{figure}

\FloatBarrier

\begin{figure}[htb]
    \centering
        \resizebox{0.9\columnwidth}{!}{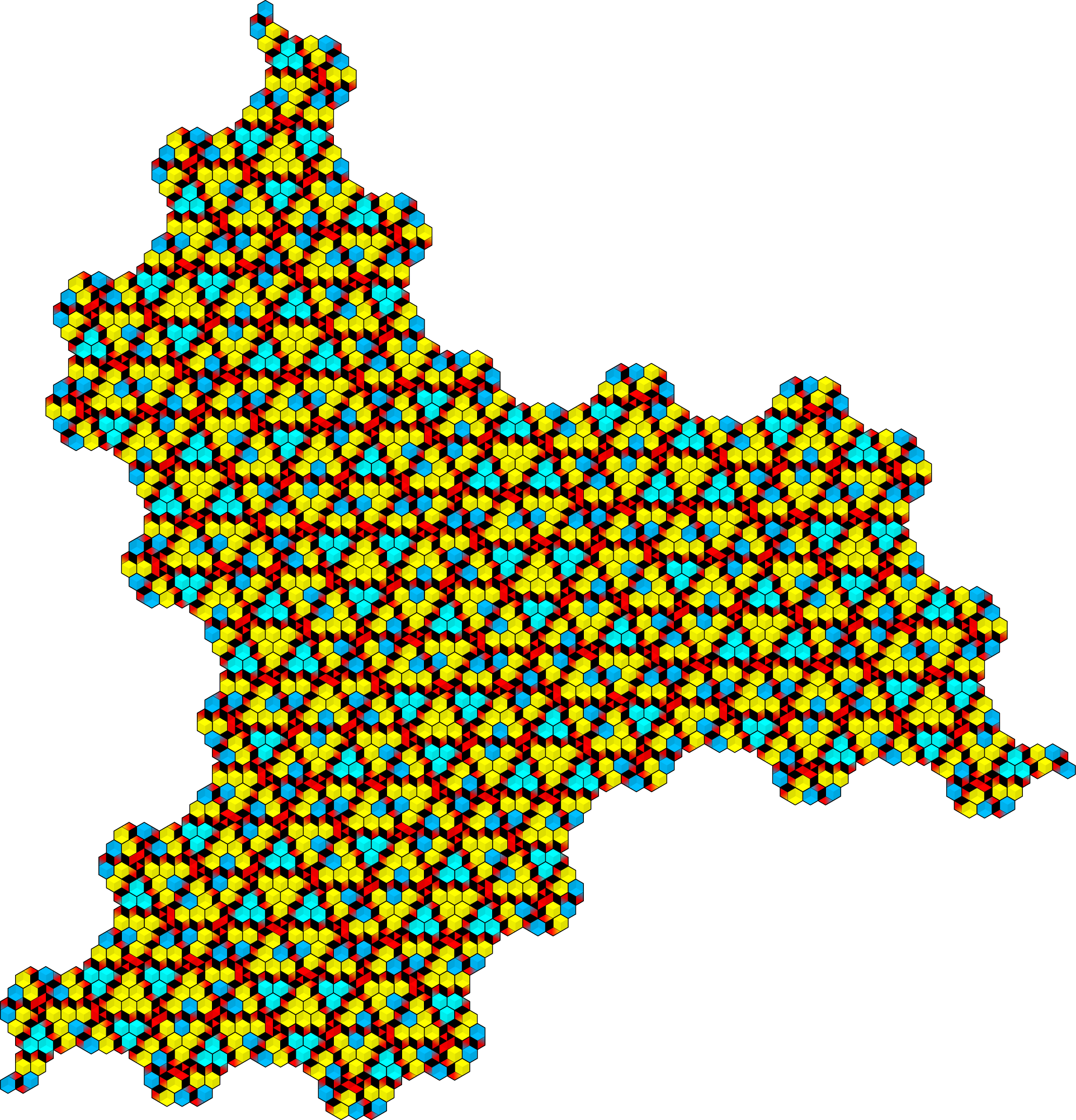}
        % \includesvg[scale=0.5]{images/spectre/recursion/articulated/new/TC2_opt}
    \caption{$\text{TC}_2$}
    \label{fig:tc2}
\end{figure}
\FloatBarrier

These rules may be described as \emph{substitutions} as they are combinatorial substitutions. The word \emph{recursion} may be more fitting since the tiling is generated in the same recursive manner as a 1D Sturmian word. Nevertheless, we can write down the substitution matrix for these rules. The dimension of this matrix is large since there are many 0D, 1D and 2D pieces taking part in the substitutions.

Following \cite{twin_worms}, we using lower case letters to indicate the number of tiles in each shape and find the following numeric recursions:

\begin{alignat*}{2}
    \mathrm{td}_{k} &= 3\mathrm{pb}_{k-1} + \mathrm{tc}_{k-1} && + 3\mathrm{n}_k \\
    \mathrm{pa}_{k} &= \mathrm{pb}_{k-1} + 2\mathrm{ta}_{k-1} && + 2\mathrm{n}_k \\
    \mathrm{tb}_{k} &= 3\mathrm{td}_k + 3\mathrm{pa}_k + \mathrm{ta}_{k-1} && + 6\mathrm{m}_k \\
    \mathrm{tc}_{k} &= \mathrm{td}_k + 3\mathrm{pa}_k && + 3\mathrm{m}_k \\
    \mathrm{pb}_{k} &= 2\mathrm{tb}_k + 2\mathrm{tc}_k + \mathrm{pa}_k && + 4\mathrm{s}_k \\
    \mathrm{ta}_{k} &= \mathrm{tb}_k + 3\mathrm{tc}_k && + 3\mathrm{s}_k \\
\end{alignat*}

Since the main use of the substitution matrix is to read off the inflation factor, we can save space by considering the 2D components in isolation. For large $k$, the 2D elements approximate fractals and the 1D elements added in each step have vanishing area and so don't contribute to the inflation factor. By ignoring the 0D and 1D elements, we're considering the substitution matrix satisfied my the fractal in the limit $k \rightarrow \infty$.

With respect to an ordering $(\mathrm{td}_k, \mathrm{pa}_k, \mathrm{ta}_{k-1}, \mathrm{tb}_k,\mathrm{tc}_k, \mathrm{pa}_{k-1} , \mathrm{pb}_{k}, \mathrm{ta}_{k}, \mathrm{tc}_{k-1})$, the substitution matrix $M$ is given by

\begin{align*}
M = \begin{bmatrix}
    0 & 0 & A \\
    B & 0 & 0 \\
    0 & C & 0 \\
\end{bmatrix} \hspace{2em} \text{where} \hspace{2em}
A = 
\begin{bmatrix}
3 & 0 & 1 \\
1 & 2 & 0 \\  
0 & 1 & 0 \\
\end{bmatrix}\hspace{1em}
B =
\begin{bmatrix}
3 & 3 & 1 \\
1 & 3 & 0 \\
0 & 1 & 0 \\
\end{bmatrix}\hspace{1em}
C =
\begin{bmatrix}
2 & 2 & 1 \\
1 & 3 & 0 \\ 
0 & 1 & 0 \\
\end{bmatrix}
\\
\end{align*}

This has characteristic polynomial $(x^3-1) (x^6 - 62 x^3 + 1)$ with dominant eigenvalue $(4+\sqrt{15})^{2/3}$.

Raising $M$ to the third power corresponds to cycling through the set of three pairs of substitutions to get a reduced substitution system that avoids needing to involve all shapes.
\[
    M^ 3 = \begin{bmatrix}
        ACB & 0 & 0 \\
        0 & BAC & 0 \\
        0 & 0 & CAB \\
    \end{bmatrix}
\]

\[
ACB = \begin{bmatrix}
  25 & 42 & 6 \\
  20 & 37 & 4 \\  
  6 & 12 & 1 \\
\end{bmatrix}\hspace{1em}
BAC = \begin{bmatrix}
31 & 48 & 12 \\
18 & 31 & 6 \\
4 & 8 & 1 \\
\end{bmatrix}\hspace{1em}
CAB =
\begin{bmatrix}
37 & 28 & 8 \\
30 & 25 & 6 \\ 
6 & 6 & 1 \\
\end{bmatrix}
\]

Each 3x3 matrix here has characteristic polynomial $(x-1) (x^2 - 62 x + 1)$ with dominant eigenvalue $(4+\sqrt{15})^2$. This corresponds to fractals following the substitutions shown in Figure~\ref{fig:spectrelimits}.

\begin{figure}[htb]
    \centering
    \begin{tabular}{c}  
        \includegraphics[width=0.9\linewidth]{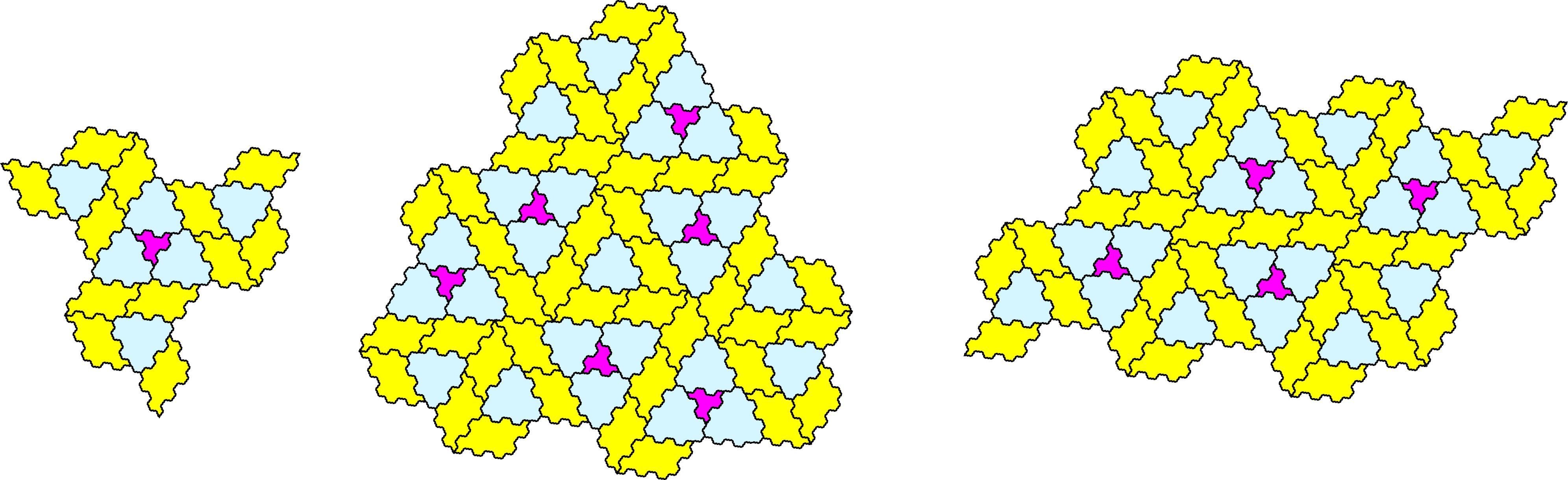} \\
        \\
        \hline
        \\
        \vspace{1em}
        \includegraphics[width=0.9\linewidth]{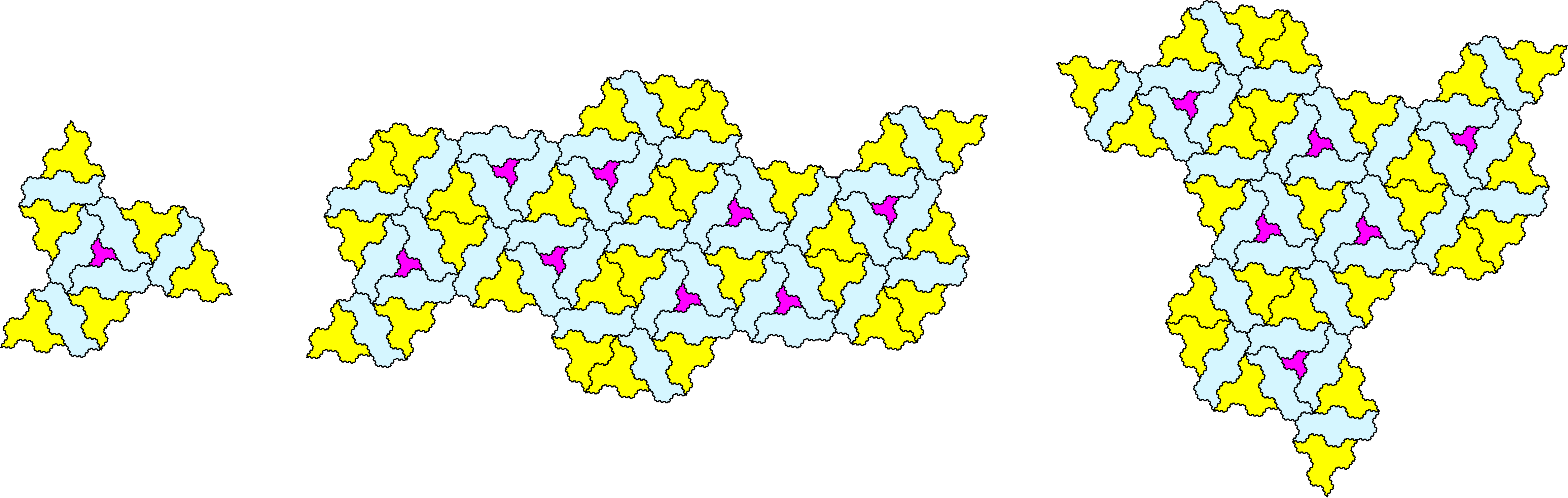} \\
        \\
        \hline
        \\
        \vspace{1em}
        \includegraphics[width=0.9\linewidth]{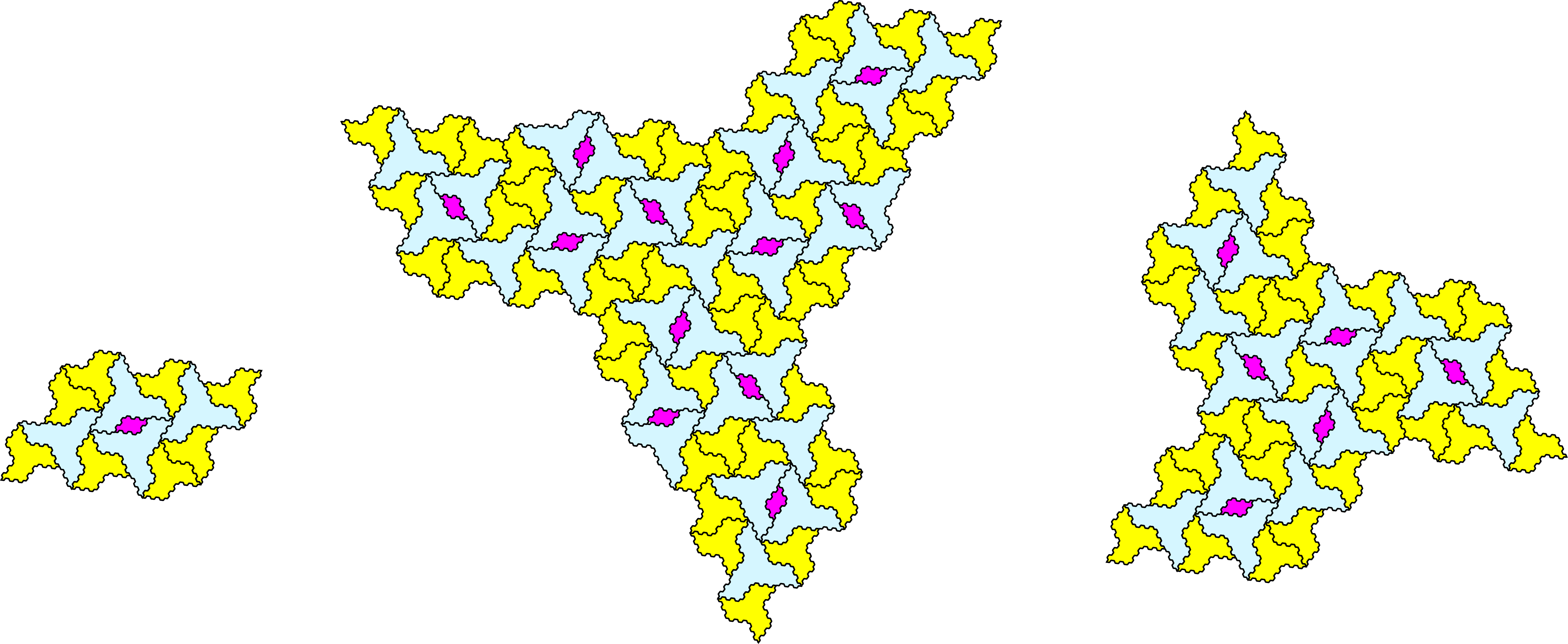} \\
    \end{tabular}
    \caption{Spectre limit substitutions.}
    \label{fig:spectrelimits}
\end{figure}

\FloatBarrier
\subsubsection{Wriggly recursion}

We now show a second substitution system due to Yoshiaki Araki. This has the exact same recursion relations, but with different initial values. This system builds more wriggly Conway worms compared to the articulated worms of the previous system. Figures~\ref{fig:patch-a} and \ref{fig:patch-w} show these side by side as Spectre tiles.

\begin{figure}[htb]
\setlength{\tabcolsep}{20pt}
\renewcommand{\arraystretch}{1.5}
\begin{tabular}{c|c}
$E$ & \rotatebox{150}{\resizebox{0.022\columnwidth}{!}{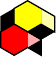}} \\
$O$ & \rotatebox{150}{\resizebox{0.022\columnwidth}{!}{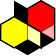}} \\
% $E$ & \includesvg[scale=0.4,angle=150]{images/spectre/recursion/wriggly/worms/wEven} \\
% $O$ & \includesvg[scale=0.4,angle=150]{images/spectre/recursion/wriggly/worms/wOdd} \\
\vspace{2em}\\
$I_0$ & \resizebox{0.08\columnwidth}{!}{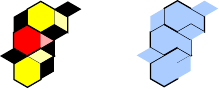} \\
$S_0$ &  \resizebox{0.08\columnwidth}{!}{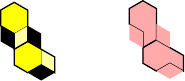} \\
% $I_0$ & \includesvg[scale=0.4]{images/spectre/recursion/wriggly/worms/wI0} \\
% $S_0$ & \includesvg[scale=0.4]{images/spectre/recursion/wriggly/worms/wS0} \\
\hline
\vspace{2em}\\
$I_1$ & \resizebox{0.27\columnwidth}{!}{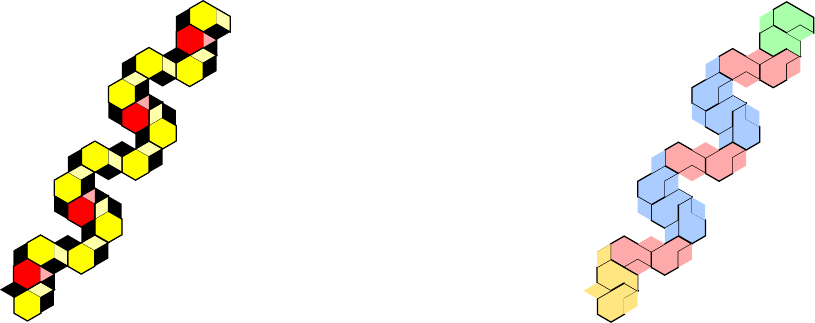} \\
$S_1$ & \resizebox{0.5\columnwidth}{!}{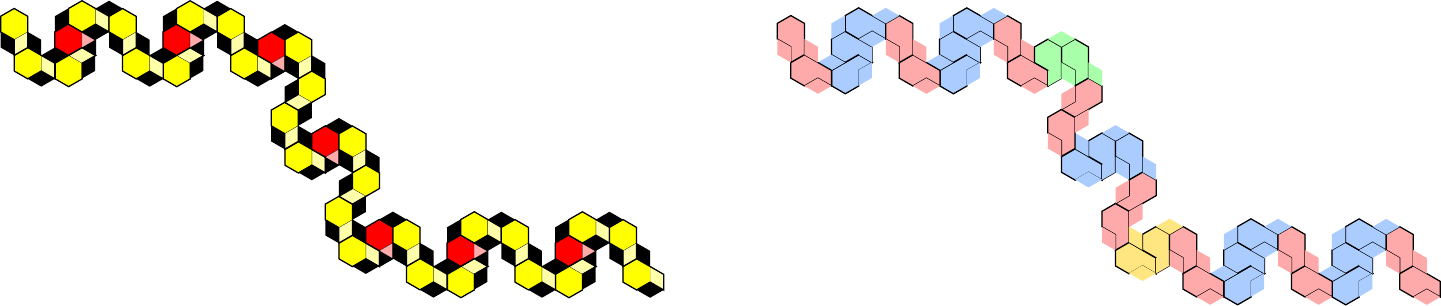} \\
% $I_1$ & \includesvg[scale=0.4]{images/spectre/recursion/wriggly/worms/wI1} \\
% $S_1$ & \includesvg[scale=0.4]{images/spectre/recursion/wriggly/worms/wS1} \\
\hline
\vspace{2em}\\
$I_2$ & \resizebox{0.15\columnwidth}{!}{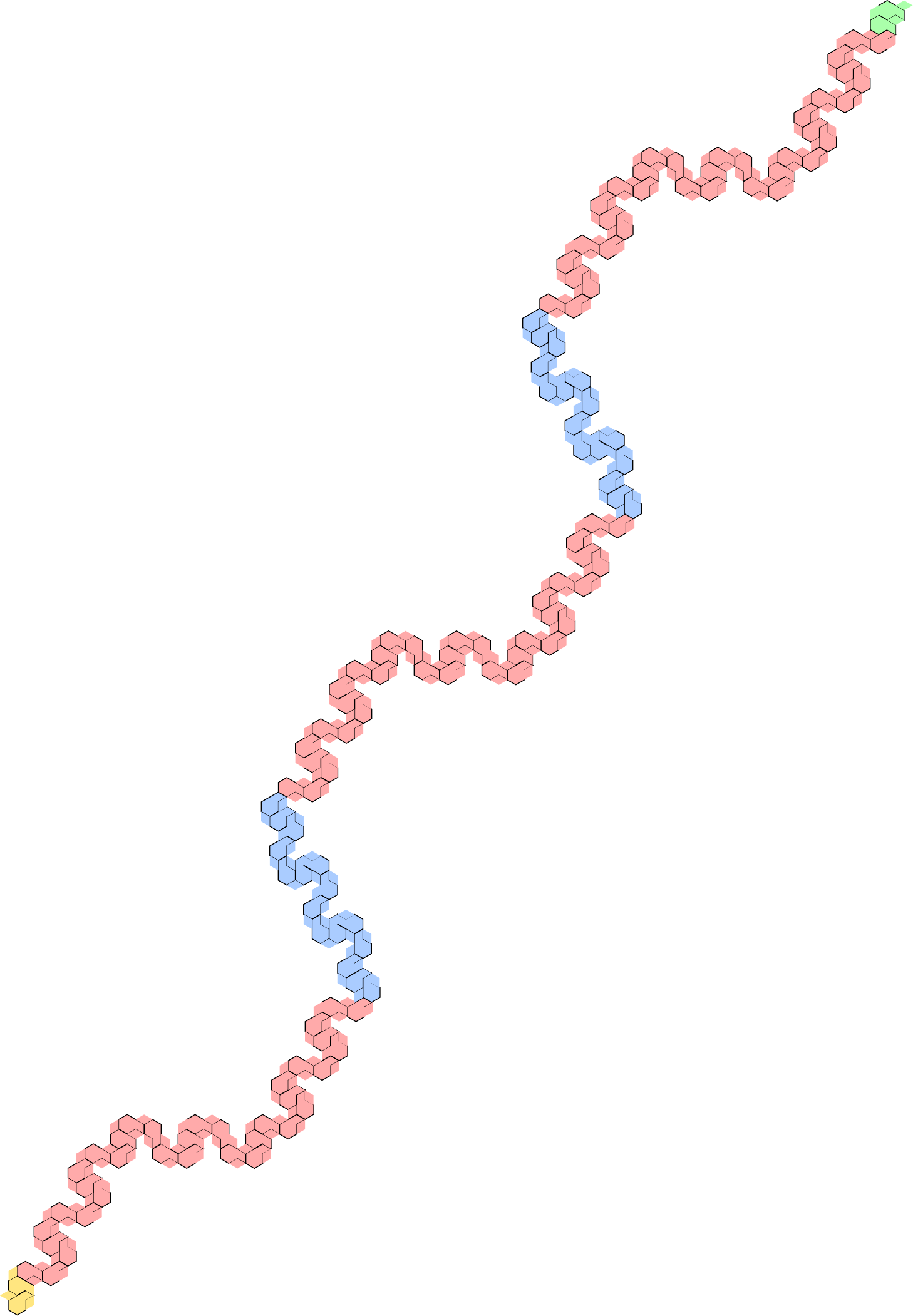} \\
$S_2$ & \resizebox{0.55\columnwidth}{!}{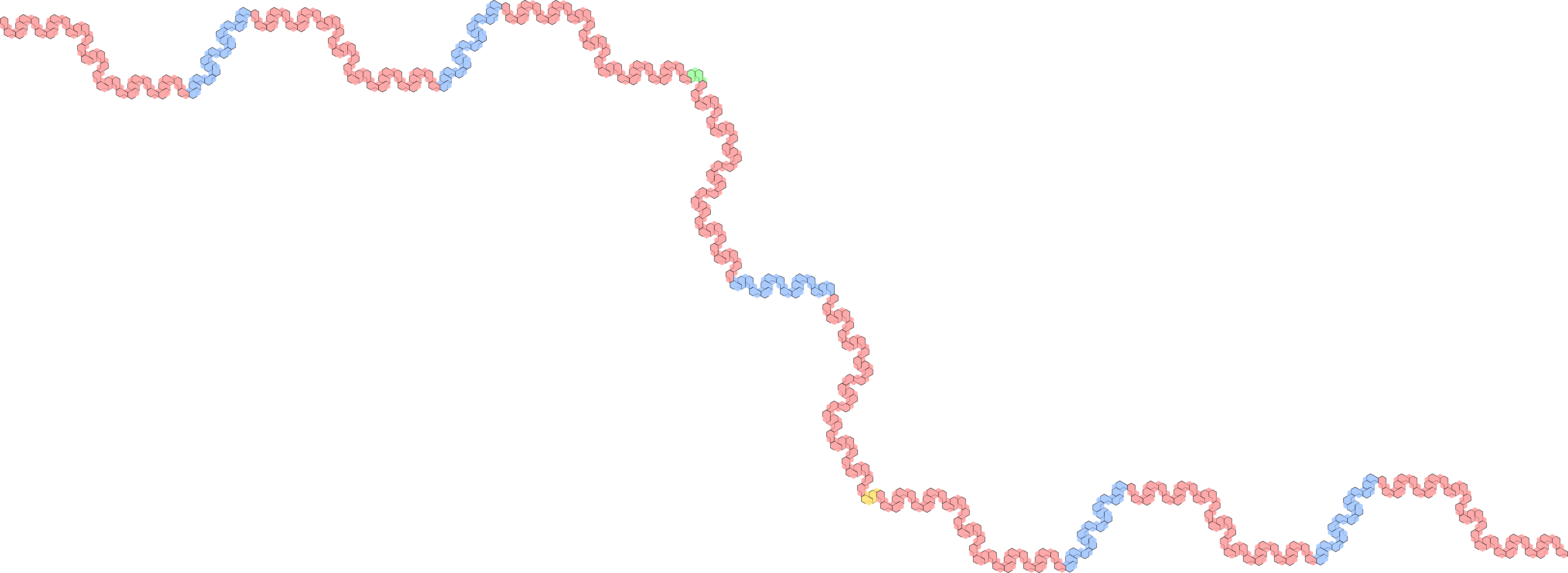} \\
% $I_2$ & \includesvg[scale=0.1]{images/spectre/recursion/wriggly/worms/wI2} \\
% $S_2$ & \includesvg[scale=0.1]{images/spectre/recursion/wriggly/worms/wS2} \\
\hline
& $I_{k+1}= O S_{k} I_{k} S_{k} I_{k} S_{k} E$ \\
& $S_{k+1} = (S_k I_k S_k I_k S_k) E (S_k I_k S_k) O (S_k I_k S_k I_k S_k) $ \\ 
\end{tabular}
\end{figure}

\begin{figure}[htb]
\setlength{\tabcolsep}{20pt}
\renewcommand{\arraystretch}{1.5}
\begin{tabular}{cc}
    \begin{tabular}{c|c}
    $N_1$ & \resizebox{0.06\columnwidth}{!}{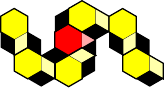} \\
    % $N_1$ & \includesvg[scale=0.5]{images/spectre/recursion/wriggly/worms/wN0} \\
    \hline
    \vspace{2em}\\
    $N_2$ & \resizebox{0.15\columnwidth}{!}{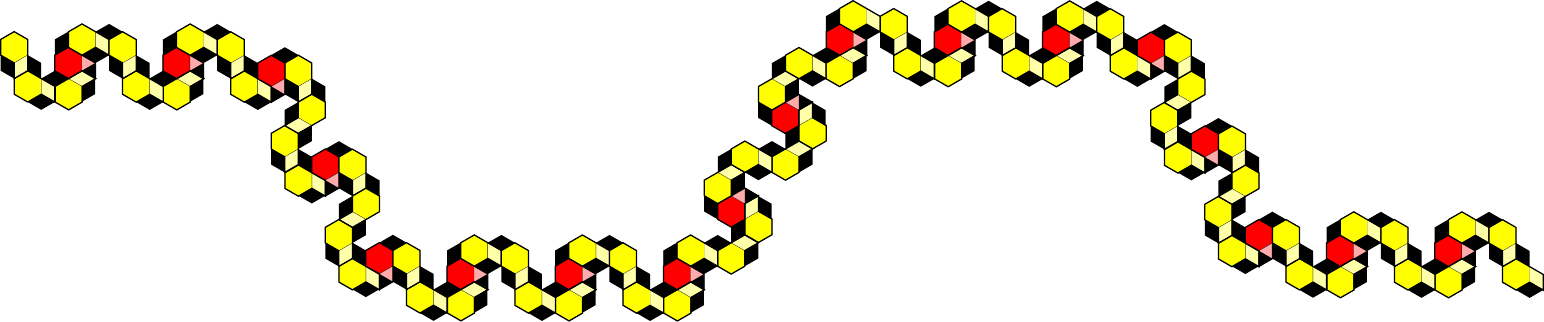} \\
    % $N_2$ & \includesvg[scale=0.1]{images/spectre/recursion/wriggly/worms/wN1} \\
    \hline
    & $N_{k+1}= S_{k} I_{k} S_{k}$ \\
    \end{tabular}
    
    &

    \begin{tabular}{c|c}
    $M_1$ & \resizebox{0.1\columnwidth}{!}{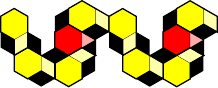} \\
    % $M_1$ & \includesvg[scale=0.5]{images/spectre/recursion/wriggly/worms/wM0} \\
    \hline
    \vspace{2em}\\
    $M_2$ & \resizebox{0.2\columnwidth}{!}{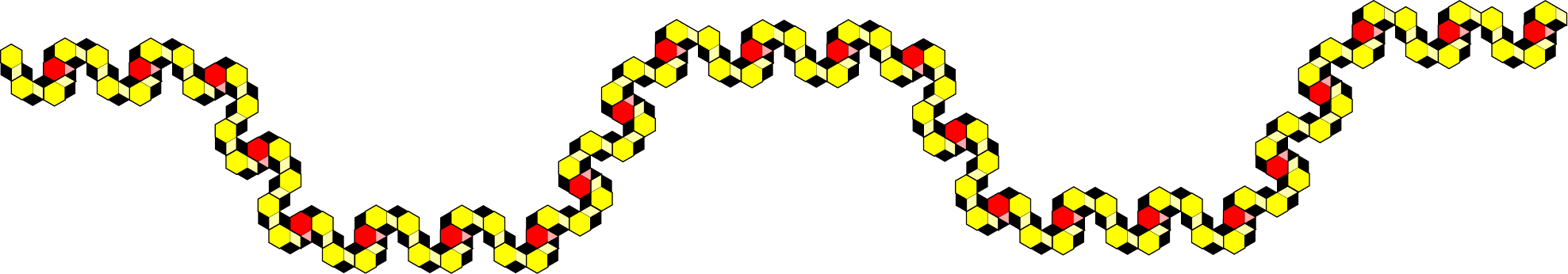} \\
    % $M_2$ & \includesvg[scale=0.1]{images/spectre/recursion/wriggly/worms/wM1} \\
    \hline
    & $M_{k+1}= S_{k} I_{k} S_{k} I_{k} M_{k}$ \\
    \end{tabular}

\end{tabular}

\end{figure}

\begin{figure}[htb]
\setlength{\tabcolsep}{20pt}
\renewcommand{\arraystretch}{1.5}
\centering
\begin{tabular}{c c}
    \begin{tabular}{c}
        $\mathrm{PB}_1$ \rotatebox{120}{\resizebox{0.07\columnwidth}{!}{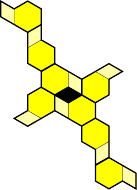}} \\
        $\mathrm{TA}_1$ \resizebox{0.07\columnwidth}{!}{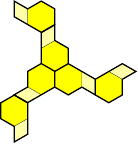} \\
        % $\mathrm{PB}_1$ \includesvg[scale=0.5,angle=120]{images/spectre/recursion/wriggly/wPB0} \\
        % $\mathrm{TA}_1$ \includesvg[scale=0.5]{images/spectre/recursion/wriggly/wTA0} \\
    \end{tabular}
    & 
    \begin{tabular}{c}
        $\mathrm{TD}_1$ \resizebox{0.15\columnwidth}{!}{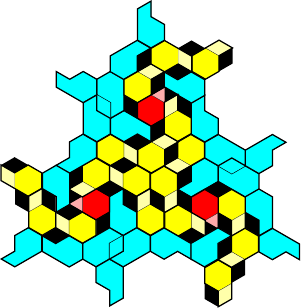} \\
        $\mathrm{PA}_1$ \resizebox{0.17\columnwidth}{!}{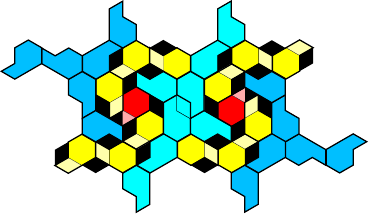} \\
        % $\mathrm{TD}_1$ \includesvg[scale=0.5]{images/spectre/recursion/wriggly/wTD0} \\
        % $\mathrm{PA}_1$ \includesvg[scale=0.5]{images/spectre/recursion/wriggly/wPA0} \\
    \end{tabular}
    \\
    \begin{tabular}{c}
        $\mathrm{TB}_1$ \resizebox{0.22\columnwidth}{!}{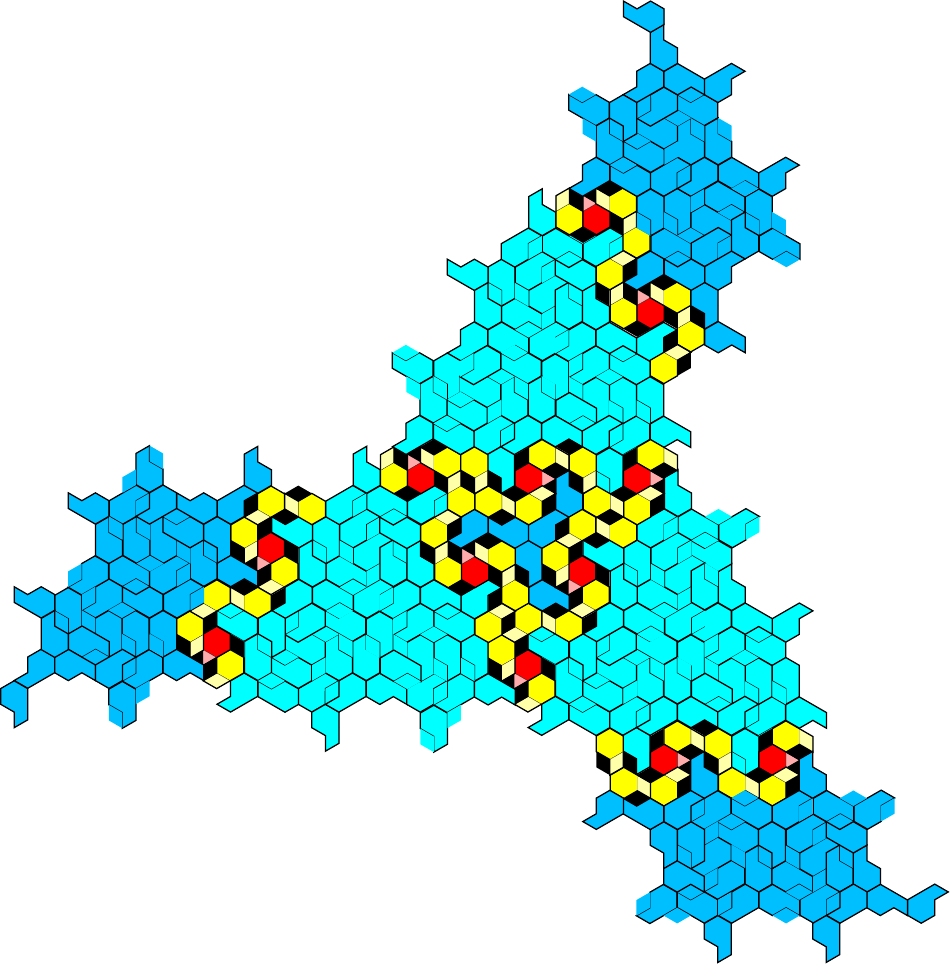} \\
        $\mathrm{TC}_1$ \resizebox{0.15\columnwidth}{!}{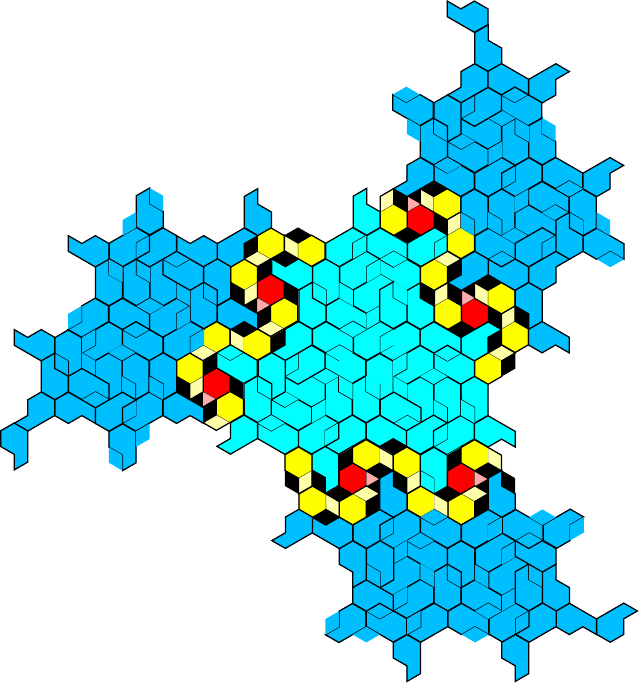} \\
        % $\mathrm{TB}_1$ \includesvg[scale=0.25]{images/spectre/recursion/wriggly/wTB0} \\
        % $\mathrm{TC}_1$ \includesvg[scale=0.25]{images/spectre/recursion/wriggly/wTC0} \\
    \end{tabular}
    &
    \begin{tabular}{c}
        $\mathrm{PB}_2$ \resizebox{0.29\columnwidth}{!}{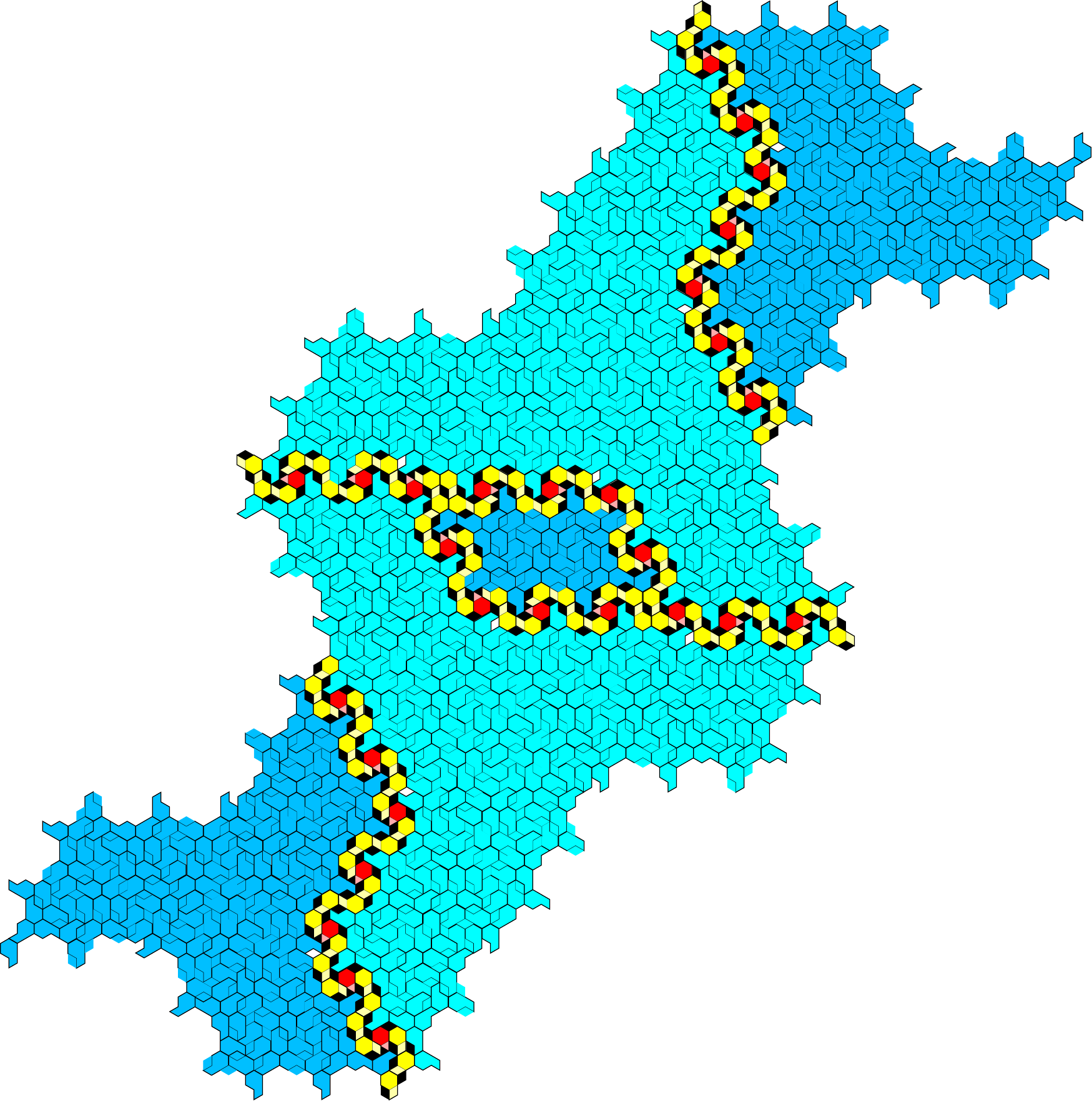} \\
        $\mathrm{TA}_2$ \resizebox{0.25\columnwidth}{!}{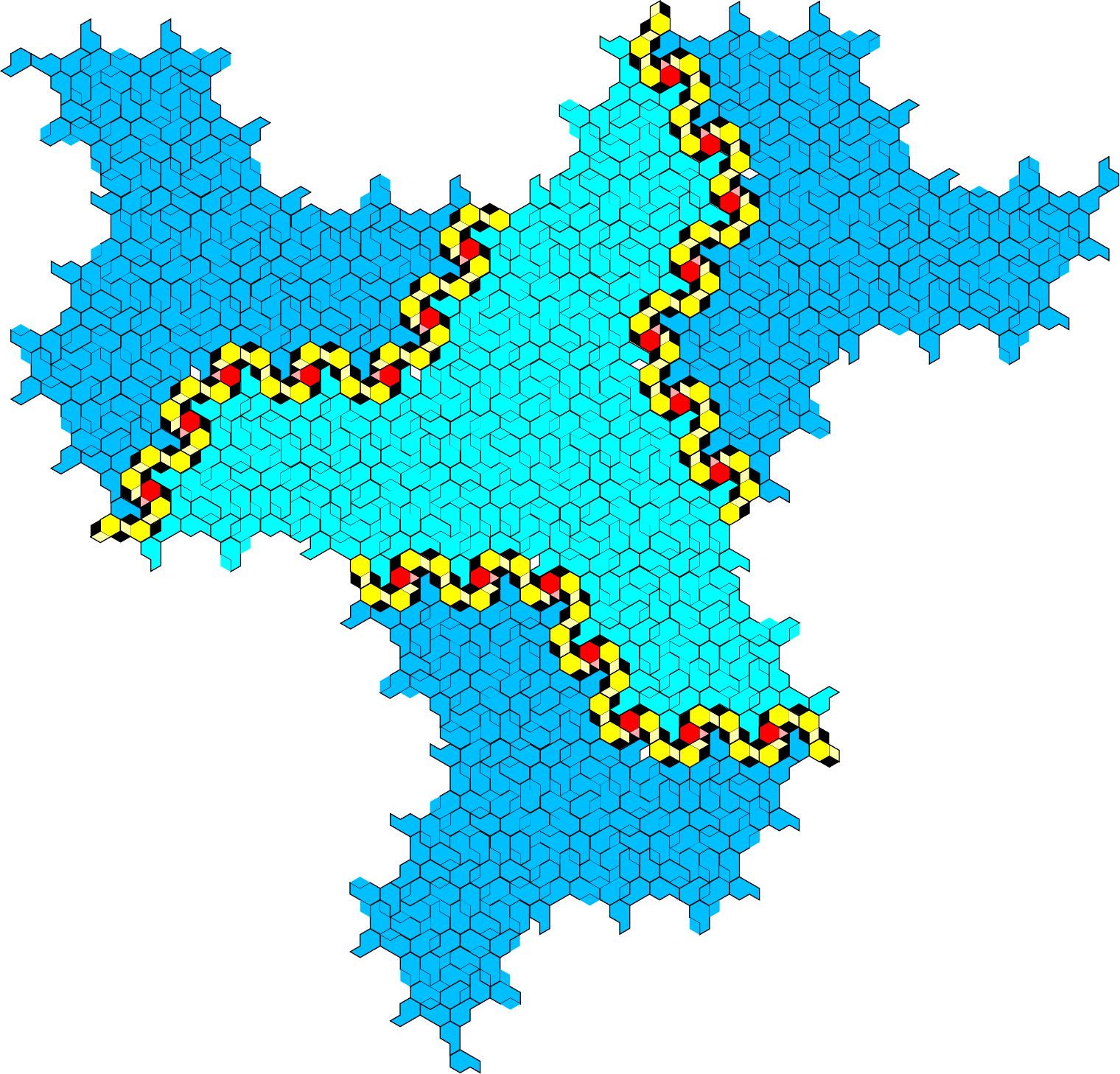} \\
        % $\mathrm{PB}_2$ \includesvg[scale=0.18]{images/spectre/recursion/wriggly/wPB1} \\
        % $\mathrm{TA}_2$ \includesvg[scale=0.18]{images/spectre/recursion/wriggly/wTA1} \\
    \end{tabular}
\end{tabular}
\caption{The wriggly recursive rules}
\end{figure}

\begin{figure}[htb]
    \centering
    \includegraphics[width=1\linewidth]{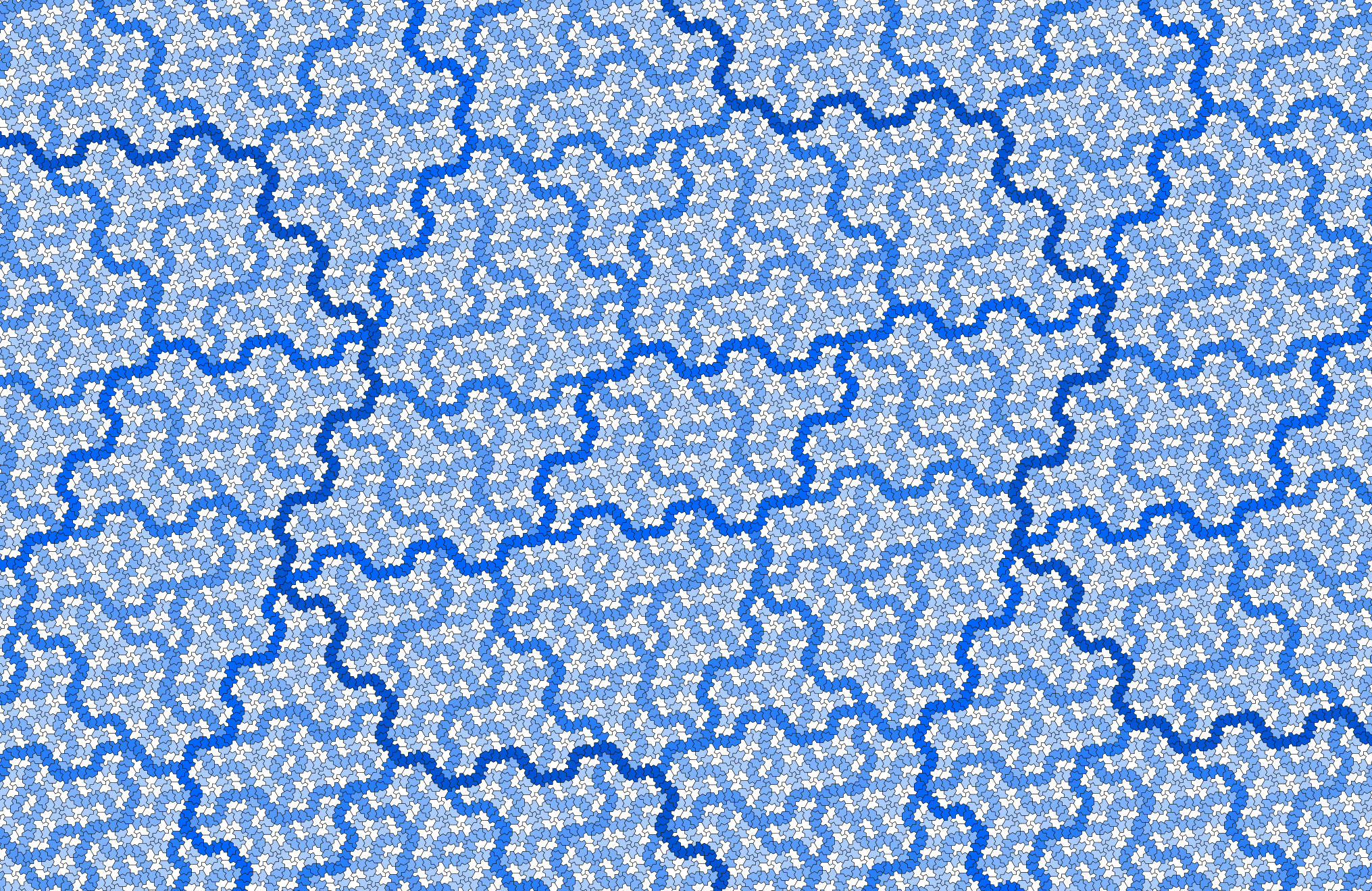}
    \caption{Articulated worms}
    \label{fig:patch-a}
\end{figure}
\begin{figure}[htb]
    \centering
    \includegraphics[width=1\linewidth]{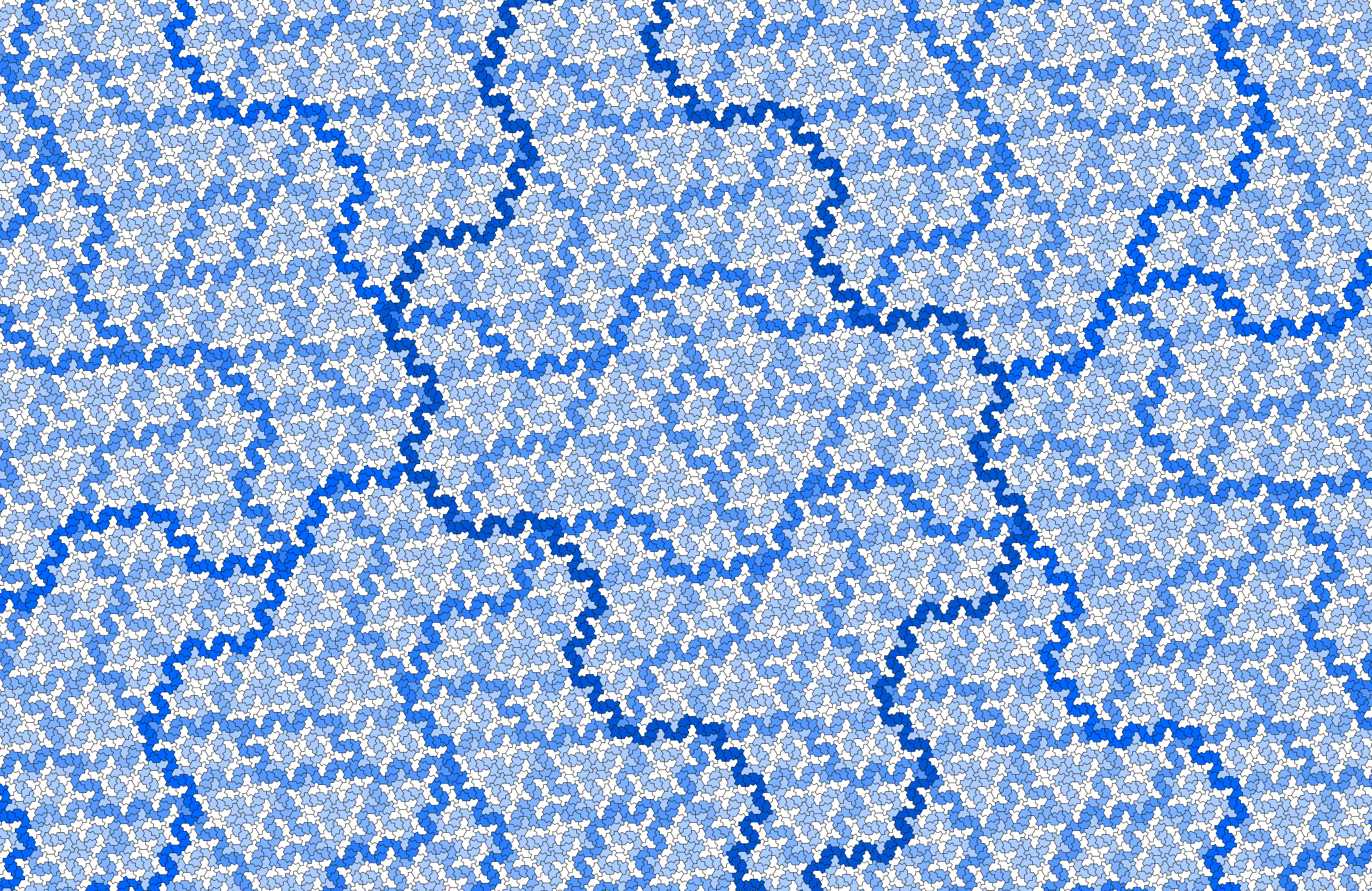}
    \caption{Wriggly worms}
    \label{fig:patch-w}
\end{figure}

\clearpage
\subsection{Sturmian words}\label{section:sturmian}
 
Let $[0; 1 + d_1, d_2, d_3, ... ]$ be the continued fraction representation of a real number $a$ with $0 < a < 1$. The sequence of words defined recursively as

\begin{align*}
    s_0 & = 1\\
    s_1 & = 0\\
    s_{k+1} & = s_{k}^{d_k} s_{k-1}
\end{align*}

is called the standard sequence with slope $a$ and converges to an infinite word called the characteristic word for $a$.

We have already seen a standard sequence related to the Turtle tiling: the Fibonacci words from Section~\ref{section:fibinacci} are the standard word corresponding to the value $a = \frac{1}{10}(5-\sqrt{5}) = [0; 3, \bar{1}]$. Values represented by periodic continued fractions also satisfy a quadratic equation. In this case, $a$ satisfies $x(1-x)=\frac{1}{5}$ and the second root $a^\prime = \frac{1}{10}(5+\sqrt{5})$ is also in the range $(0,1)$ and has the continued fraction representation $[0;1,2,\bar{1}]$.

The standard sequences for the two roots are as follows:

\bigskip
\begin{tabular}{c|c|c}
        & $[0;3,\bar{1}]$ & $[0;1,2,\bar{1}]$ \\
    $s_0$ & 1 & 1 \\
    $s_1$ & 0 & 0 \\
    $s_2$ & 001 & 1 \\
    $s_3$ & 0010 & 110 \\
    $s_4$ & 0010001 & 1101 \\
    $s_5$ & ... & 1101110 \\
\end{tabular}
\bigskip

That is, the standard sequences of the two conjugate roots are the binary complement of each other. This makes sense since the standard sequences converge to the cutting sequence for a line with slope $a$ (see \cite{combinatoricsonwords}): a line of slope $a$ intersects vertical and horizontal grid lines. Each intersection adds a `0' or `1' to the sequence depending on whether the grid line was horizontal or vertical. The binary complement is the operation which sends a line of slope $a$ to a line of slope $1-a$. In our case, the two roots of $x(1-x)=\frac{1}{5}$ have $a^\prime = 1 - a$.

We've seen that the Spectre tiles also have recursive rules defining strips of tiles that build into ever increasing patches to tile the plane. Let's calculate the standard words for these recursive strips of tiles. For this, make a choice of labelling an odd tile as `1' and an even tile as `0'. 

For articulated worms, the sequence has $S_0 = \varnothing$ and $I_0 = 0$,
 $S_1 = 001000100$ and $I_1 = 010010$. Additionally, writing  $E = 10$ (the even number 2 in binary), and $O = 01$, we have the recursions.

\begin{align*}
   S_{k+1} &= (S_kI_iS_kI_iS_k) E (S_kI_iS_k) O (S_kI_iS_kI_iS_k)  \\
   I_{k+1} &= O (S_kI_iS_kI_iS_k) E  
\end{align*}

Since the word lengths increase rapidly, we will use these symbols $S$, $I$, $O$ and $E$ to make the words more manageable. Also it will be helpful to note the string equality $OS_kI_k = I_kS_kE$: this is trivially true for $k=0$, for larger $k$ it follows by induction since

\begin{align*}
  OS_kI_k &:= O\ (SISISESISOSISIS)\ OSISISE \\
          & = OSISISE\ SIS(OSI)SISOSISIS\ E \\
          &= OSISISE\ SIS(ISE)SISOSISIS\ E \\ 
          &=: I_kS_kE
\end{align*}

\begin{table}[htb]
    \centering
    \begin{tabular}{c|c}
         & \\
        $s_0$ & 1  \\
        $s_1$ & 0 \\
        $s_2 = s_1^{3-1}s_0$ & 00 1 \\
        & \\
        $s_3 = s_2s_1$ & (001)0 = $S_0I_0S_0I_0S_0\ \ E$ \\
        $s_4 = s_3^2s_2$ & (0010)(0010)(001) = $S_1 O$ \\
        $s_5 = s_4s_3$ & (00100010001)(0010) = $S_1I_1$  \\
        $s_6 = s_5s_4$ & $(S_1I_1)(S_1O)$  \\
        & \\
        $s_7 = s_6s_5$ & $(S_1I_1S_1O)(S_1I_1)$ = $S_1I_1S_1I_1S_1\ \ E$ \\
        $s_8 = s_7^2s_6$ & $(S_1I_1S_1I_1S_1E)(S_1I_1S_1I_1S_1E)(S_1I_1S_1O)$ = $S_2O$ \\
        $s_9 = s_8s_7$ & $(S_2O)(S_1I_1S_1I_1S_1E)$ = $S_2I_2$  \\
        $s_{10}=s_9s_8$ & $(S_2I_2)(S_2O)$  \\
        
    \end{tabular}
    \caption{$[0; 3, \overline{1, 2, 1, 1}]$}
    \label{table:articulatedsturmian}
\end{table}

\begin{table}[htb]
    \centering
    \begin{tabular}{c|c}
         & \\
        $s_0$ & 1  \\
        $s_1$ & 0 \\
        $s_2 = s_1^{4-1}s_0$ & 000 1 \\
        $s_3 = s_2s_1$ & (0001)(0) = 00\ 010 =: $S_0I_0$\\
        $s_4 = s_3s_2$ & (00010)(0001) = 00 010 00 01 =: $S_0I_0S_0O$ \\
        & \\
        $s_5 = s_4s_3$ & $(S_0I_0S_0O)(S_0I_0)$ = $S_0I_0S_0I_0S_0\ \ E$  \\
        $s_6 = s_5^2s_4$ & $S_1O$ \\
        $s_7 = s_6s_5$ & $S_1I_1$ \\
        $s_8 = s_7s_6$ & $(S_1I_1)(S_1O)$ \\
        
    \end{tabular}
    \caption{$[0; 4, \overline{1, 1, 1, 2}]$}
    \label{table:wrigglysturmian}
\end{table}

Table~\ref{table:articulatedsturmian} makes use of this identity to show that the recursion for the Spectre's articulated Conway worms is given by the standard word with slope $[0; 3, \overline{1, 2, 1, 1}] = \frac{2}{19} (5 - \sqrt{6})$. Similarly, with it's different initial values, but the same recursion relation, the odd tiles along the wriggly Conway worm correspond to $[0; 4, \overline{1, 1, 1, 2}] = \frac{1}{19} (9 - 2\sqrt{6})$.

These values are roots of $19 x^2 - 20 x + 4 = 0$ and $19 x^2 - 18 x + 3 = 0$ which are related by the identity $19(1-x)^2 - 20 (1-x) + 4 = 19 x^2 - 18 x + 3$ so that the binary complement of each sequence corresponds to conjugate root.

Hence it makes sense to refer to the articulated and wriggly Spectre constructions as being conjugate.
\FloatBarrier

\section{Acknowledgements}

I would like to express my sincere thanks to Erhard Künzel and Yoshiaki Araki for their advice and suggestions in preparing this work. I benefited from their friendly and collaborative spirit and I hope that this final result does justice to their contributions.
\FloatBarrier

\printbibliography

\begin{figure}[h]
    \centering
    \includesvg[width=0.05\columnwidth]{images/spectre/squaremystic}
\end{figure}

\end{document}